\documentclass[
reqno,
10pt,
oneside
]{article}

\usepackage[ngerman, english]{babel}
\usepackage[utf8]{inputenc}
\usepackage[normalem]{ulem}
\usepackage[babel,french=guillemets,german=swiss]{csquotes}
\usepackage[center,font=footnotesize]{caption}
\usepackage{cite}
\usepackage{bbm}
\usepackage[raggedright]{titlesec}
\usepackage{xcolor}
\usepackage{enumerate}

\titleformat{\section}{\normalfont\large\bfseries}{\thesection}{1em}{}
\titleformat{\subsection}{\normalfont\bfseries}{\thesubsection}{1em}{}

\usepackage{tocbasic}
\DeclareTOCStyleEntry[
beforeskip=.2em plus 1pt,
]{tocline}{section}

\usepackage[%
left=1.25in,
right=1.25in,
bottom=1in,
top=0.8in,
footskip=0.4in,
]{geometry}

\usepackage{color}
\definecolor{LinkColor}{rgb}{0,0,1}
\definecolor{LinkColor2}{rgb}{0,0.5,0}
\definecolor{lbcolor}{rgb}{0.85,0.85,0.85}
\definecolor{FrameColor}{rgb}{0.85,0.85,0.85}
\definecolor{rosso}{rgb}{0.8,0,0}
\definecolor{lightgray}{rgb}{0.75,0.75,0.75}
\definecolor{violet}{rgb}{0.65,0,0.65}
\definecolor{darkgreen}{rgb}{0,0.5,0}

\usepackage{enumitem}

\usepackage{graphicx}
\usepackage{placeins}
\usepackage{overpic}

\usepackage{amsmath}
\usepackage{amssymb}
\usepackage{dsfont}
\usepackage{empheq}
\allowdisplaybreaks
\usepackage{amsthm}

\usepackage[%
pdftitle={Titel},%
pdfauthor={Autor},%
pdfcreator={LaTeX, LaTeX with hyperref and KOMA-Script},
pdfsubject={Betreff}, 
pdfkeywords={Keywords}
]{hyperref} 

\hypersetup{%
	colorlinks	=true,
	linkcolor	=LinkColor,%
	anchorcolor	=LinkColor,%
	citecolor	=LinkColor2,%
	filecolor	=LinkColor,%
	menucolor	=LinkColor,%
	urlcolor	=LinkColor,%
}
\usepackage{marginnote}

\usepackage{verbatim}

\newtheorem{theorem}{Theorem}[section]
\newtheorem{lemma}[theorem]{Lemma}

\newtheorem{corollary}[theorem]{Corollary}
\newtheorem{definition}[theorem]{Definition}

\theoremstyle{definition}
\newtheorem{remark}[theorem]{Remark}

\makeatletter
\renewenvironment{proof}[1][\proofname]{%
	\par\pushQED{\qed}\normalfont%
	\topsep6\p@\@plus6\p@\relax
	\trivlist\item[\hskip\labelsep\bfseries#1\@addpunct{.}]%
	\ignorespaces
}{%
	\popQED\endtrivlist\@endpefalse
}
\makeatother

\makeatletter
\renewcommand\paragraph{\@startsection{paragraph}{4}{\z@}%
	{1ex \@plus1ex \@minus.2ex}%
	{-1em}%
	{\normalfont\normalsize\bfseries}}
\renewcommand\subparagraph{\@startsection{paragraph}{4}{\z@}%
	{1ex \@plus1ex \@minus.2ex}%
	{-1em}%
	{\normalfont\normalsize\itshape}}
\makeatother

\DeclareMathAlphabet\mathbfcal{OMS}{cmsy}{b}{n}

\newcommand{\abs}[1]{\left| #1 \right|}

\newcommand{\norm}[1]{\| #1 \|}

\newcommand{\ang}[2]{ \langle #1 , #2  \rangle}
\newcommand{\bigang}[2]{ \big< #1 , #2  \big>}

\newcommand{\scp}[2]{ \left( #1 , #2  \right)}
\newcommand{\bigscp}[2]{\big( #1 , #2 \big)}

\newcommand{\meano}[1]{{\langle #1 \rangle}_{\Omega}}
\newcommand{\meang}[1]{{\langle #1 \rangle}_{\Gamma}}
\newcommand{\mean}[2]{\textnormal{mean}\scp{#1}{#2}}

\newcommand{\R}{\mathbb R}
\newcommand{\N}{\mathbb N}
\newcommand{\n}{\mathbf{n}}
\newcommand{\bu}{\mathbf{u}}
\newcommand{\bv}{\mathbf{v}}
\newcommand{\bw}{\mathbf{w}}
\newcommand{\ww}{\widehat{\bw}}
\newcommand{\wv}{\widehat{\bv}}
\newcommand{\tv}{\widetilde{\bv}}
\newcommand{\tw}{\widetilde{\bw}}
\newcommand{\wphi}{\widehat{\phi}}
\newcommand{\wpsi}{\widehat{\psi}}
\newcommand{\wmu}{\widehat{\mu}}
\newcommand{\wtheta}{\widehat{\theta}}
\newcommand{\hJ}{\widehat{\mathbf{J}}}
\newcommand{\hK}{\widehat{\mathbf{K}}}
\newcommand{\J}{\mathbf{J}}
\newcommand{\K}{\mathbf{K}}
\newcommand{\T}{\mathbf{T}}
\newcommand{\I}{\mathbf{I}}
\newcommand{\D}{\mathbf{D}}
\newcommand{\f}{\mathbf{f}}

\newcommand{\A}{\mathbf{A}}
\newcommand{\M}{\mathbf{M}}
\renewcommand{\L}{\mathbf{L}}
\newcommand{\G}{\mathbf{G}}
\renewcommand{\S}{\mathbf{S}}
\newcommand{\Y}{\mathbf{Y}}

\newcommand{\Dg}{\mathbf{D}_\Gamma}
\newcommand{\Z}{\mathbf{Z}}
\newcommand{\intO}{\int_\Omega}

\newcommand{\intG}{\int_\Gamma}

\newcommand{\ep}{\varepsilon}

\newcommand{\dtau}{\;\mathrm d\tau}

\newcommand{\dx}{\;\mathrm{d}x}

\newcommand{\ds}{\;\mathrm ds}

\newcommand{\dxt}{\;\mathrm{d}x\;\mathrm{d}t}
\newcommand{\dGt}{\;\mathrm{d}\Ga\;\mathrm{d}t}
\newcommand{\dxs}{\;\mathrm{d}x\;\mathrm{d}s}
\newcommand{\dGs}{\;\mathrm{d}\Ga\;\mathrm{d}s}

\newcommand{\dG}{\;\mathrm d\Ga}

\newcommand{\ddt}{\frac{\mathrm d}{\mathrm dt}}

\newcommand{\del}{\partial}
\newcommand{\delt}{\partial_{t}}

\newcommand{\deln}{\partial_\n}
\newcommand{\deltb}{\delt^\bullet}
\newcommand{\deltc}{\delt^\circ}
\newcommand{\delphi}{\del_\phi}
\newcommand{\delpsi}{\del_\psi}
\newcommand{\delgphi}{\del_{\Grad\phi}}
\newcommand{\delgpsi}{\del_{\Gradg\psi}}

\newcommand{\Grad}{\nabla}
\newcommand{\Lap}{\Delta}
\newcommand{\Div}{\textnormal{div}}
\newcommand{\Gradg}{\nabla_\Ga}
\newcommand{\Lapg}{\Delta_\Ga}
\newcommand{\Divg}{\textnormal{div}_\Ga}
\newcommand{\Divgt}{{\textnormal{div}_{\Ga}}}

\newcommand{\emb}{\hookrightarrow}

\newcommand{\ov}{\overline}
\newcommand{\suchthat}{\;\ifnum\currentgrouptype=16 \middle\fi|\;}

\newcommand{\e}{\mathbf{e}}

\newcommand{\Om}{\Omega}
\newcommand{\Ga}{\Gamma}

\usepackage{tikz}

\newcommand{\extended}[1]{{\color{lightgray}#1}}
\renewcommand{\extended}[1]{}

\begin{document}

%
%

\title{\bfseries 
    A thermodynamically consistent model for bulk-surface viscous fluid mixtures:\\
    Model derivation and mathematical analysis
\\[-1.5ex]$\;$}

\author{
    Patrik Knopf \footnotemark[1]
    \and Jonas Stange \footnotemark[1]}

\date{ }

\maketitle

\renewcommand{\thefootnote}{\fnsymbol{footnote}}

\footnotetext[1]{
    Universität Regensburg,
    Fakultät für Mathematik,
    93053 Regensburg, 
    Germany \newline
	\tt(%
        \href{mailto:patrik.knopf@ur.de}{patrik.knopf@ur.de},
        \href{mailto:jonas.stange@ur.de}{jonas.stange@ur.de}%
        ).
}

\begin{center}
    $\phantom{x}$\\[-5ex]
	\scriptsize
	\color{white}
	{
		\textit{This is a preprint version of the paper. Please cite as:} \\  
		P.~Knopf, J.~Stange, 
        \textit{[Journal]} \textbf{xx}:xx 000-000 (2024), \\
		\texttt{https://doi.org/...}
	}
    \normalsize
\end{center}

\medskip

%
%

\begin{small}
\begin{center}
    \textbf{Abstract}
\end{center}
We derive and analyze a new diffuse interface model for incompressible, viscous fluid mixtures with bulk-surface interaction.  
Our system consists of a Navier--Stokes--Cahn--Hilliard model in the bulk that is coupled to a surface Navier--Stokes--Cahn--Hilliard model on the boundary.
Compared with previous models, the inclusion of an additional surface Navier--Stokes equation is motivated, for example, by biological applications such as the seminal \textit{fluid mosaic model} (Singer \& Nicolson, \textit{Science}, 1972) in which the surface of biological cells is interpreted as a thin layer of viscous fluids.
We derive our new model by means of local mass balance laws, local energy dissipation laws, and the Lagrange multiplier approach. 
Moreover, we prove the existence of global weak solutions via a semi-Galerkin discretization. 
The core part of the mathematical analysis is the study of a novel bulk-surface Stokes system and its corresponding bulk-surface Stokes operator. Its eigenfunctions are used as the Galerkin basis to discretize the bulk-surface Navier--Stokes subsystem. 
\\[1ex]
\\[1ex]	
; Secondary:
 35D30, 
 35Q35, 
 76D03, 
 76D05, 
 76D07, 
 76T99. 
\end{small}

\begin{small}
\setcounter{tocdepth}{2}
\hypersetup{linkcolor=black}
\tableofcontents
\end{small}

\setlength\parskip{1ex}
\allowdisplaybreaks
\numberwithin{equation}{section}
\renewcommand{\thefootnote}{\arabic{footnote}}

\section{Introduction} 
\label{Section:Introduction}

The mathematical description of two immiscible materials is an important topic in materials science and fluid dynamics with vast applications in biology, chemistry, and engineering. To describe the motion of such mixtures, one has to capture the moving interface separating the different components. This has led to two fundamental approaches, namely the \textit{sharp interface methods} and the \textit{diffuse interface methods}. For a comparison of these two approaches, we refer, e.g., to \cite{Abels2018, Du2020, Giga2018, Pruess2016}.

In sharp interface models, the interface is described as an evolving hypersurface in the surrounding domain, which leads to the formulation of a free boundary problem. In contrast, in diffuse interface models, the interface between the two materials is represented by a thin layer whose width is proportional to a small parameter $\varepsilon > 0$. The location of the two components is represented by an \textit{order parameter}, also called \textit{phase-field}. This phase-field describes the difference of the local concentrations of the materials, and is therefore expected to attain values close to $-1$ and $1$ in the regions where the respective fluids are present. On the other hand, at the diffuse interface, it should have a continuous transition between these values. The advantage of the diffuse interface method is that the description of the phase-field is purely Eulerian and does not require any direct tracking of the interface separating the involved materials. Moreover, it allows for topology changes of the material components and has a firm physical basis for multiphase flows, see, e.g., \cite{Liu2010, Chen2014}.

\subsection{Diffuse interface models for multi-phase flows}

To describe the motion of two viscous, incompressible fluids with constant densities, one of the most fundamental models is the \textit{Model H}, which was originally proposed by Hohenberg and Halperin in \cite{Hohenberg1977} in the context of dynamic critical phenomena. It consists of an incompressible Navier--Stokes equation coupled to a convective Cahn--Hilliard equation and was later rigorously derived in \cite{Gurtin1996}. The fundamental assumption in its derivation is that the density of the mixture is constant, which means that both fluids have the same constant density.
In the past two decades, various generalizations of the Model H have been proposed in the literature to describe fluids with different densities, compressible fluids, or phase-transitions. We refer to \cite{Abels2025a, Aki2014, Boyer2002, Ding2007, Freistuhler2017, Giga2018, Heida2012, Lowengrub1998, Shen2013, Shokrpour2018, tenEikelder2024} for a selection of representative contributions.

A significant advance was achieved by the work of Abels, Garcke, and Grün \cite{Abels2009}, who derived a thermodynamically consistent diffuse-interface model for two-phase flows with different densities. The so-called \textit{AGG model} generalizes the Model H by incorporating non-matched densities, and it ensures energy dissipation, frame-indifference, and a consistent sharp-interface limit. It has since become the basis for many different works on diffuse-interface models of incompressible two-phase flows, see, e.g., \cite{Abels2013, Abels2013a, Abels2024, Giorgini2021, Giorgini2022}. 

In this framework, the treatment of boundary conditions plays a crucial role. Classical choices are no-slip boundary conditions for the velocity field, while the Cahn--Hilliard variables are usually equipped with homogeneous Neumann boundary conditions. However, numerous applications demonstrate that these static boundary conditions are insufficient when an accurate description of interfacial behaviour near the boundary is required. For instance, in the modelling of moving contact lines, the use of the generalized Navier boundary condition, relating tangential stresses to slip velocities and uncompensated Young stresses, has proven indispensable, see, e.g., \cite{Dussan1979, Qian2006, Seppecher1996}. Similarly, imposing homogeneous Neumann boundary conditions on the phase-field forces the diffuse interface to intersect the boundary at a right angle. This is often unrealistic, as in most applications, the contact angle deviates from ninety degrees and may even evolve dynamically over time. Furthermore, the homogeneous Neumann boundary condition prohibits any transfer of material between the bulk and the boundary. While such conservation of bulk mass is appropriate in some situations, other applications, such as absorption phenomena or chemical reaction at the boundary, require an exchange of material between bulk and surface. These limitations have also been recognized in the physics literature (see, e.g., \cite{Binder1991,Fischer1997,Fischer1998,Kenzler2001}), where the introduction of a surface free energy has been proposed to better capture short-range interactions between the bulk and the boundary.

To overcome these limitations, Navier--Stokes--Cahn--Hilliard models with dynamic boundary conditions on the Cahn--Hilliard subsystem and a generalized Navier-slip boundary condition on the velocity field have been proposed in the literature.
To motivate our new bulk-surface Navier--Stokes--Cahn--Hilliard model, which will be derived and analyzed in the present paper, we first recall the model from \cite{Giorgini2023}, which reads as follows:
\begin{subequations}\label{eqs:GK}
    \begin{alignat}{3}
        \label{eqs:GK:1}
        &\delt(\rho(\phi)\bv) + \Div(\bv\otimes(\rho(\phi)\bv + \J)) = \Div\;\T , &&\quad\Div\;\bv = 0 &&\qquad\text{in~}Q, 
        \\
        \label{eqs:GK:2}
        &[\T\n + \gamma(\psi)\bv]_\tau = \big[-\psi\Gradg\theta + \tfrac12(\J\cdot\n)\bv\big]_\tau &&{} &&\qquad\text{on~}\Sigma, 
        \\
        \label{eqs:GK:3}
        &\bv\cdot\n = 0 &&{} &&\qquad\text{on~}\Sigma, 
        \\[1ex]
        \label{eqs:GK:4}
        &\delt\phi  + \Div(\phi\bv) = \Div(m_\Om(\phi)\Grad\mu) &&{} &&\qquad\text{in~}Q,
        \\
        \label{eqs:GK:5}
        &\mu = -\ep\Lap\phi + \tfrac1\ep F^\prime(\phi) &&{} &&\qquad\text{in~}Q, 
        \\
        \label{eqs:GK:6}
        &\delt\psi + \Divg(\psi\bv) = \Divg(m_\Ga(\psi)\Gradg\theta) - \beta m_\Om(\phi)\deln\mu &&{} &&\qquad\text{on~}\Sigma, 
        \\
        \label{eqs:GK:7}
        &\theta = -\delta\kappa\Lapg\psi + \tfrac1\delta G^\prime(\psi) + \ep\deln\phi &&{}  &&\qquad\text{on~}\Sigma, 
        \\
        \label{eqs:GK:8}
        & \phi = \psi, \qquad Lm_\Om(\phi)\deln\mu = \beta\theta - \mu, &&\quad L\in[0,\infty] &&\qquad\text{on~}\Sigma.
        \end{alignat}
\end{subequations}
We also refer to \cite{Colli2023} for a related Cahn--Hilliard--Brinkman model with dynamic boundary conditions, and to \cite{Knopf2025a} for a free boundary variant of \eqref{eqs:GK}, where the fluids are considered in an evolving domain, whose evolution is driven by the velocity field. A model similar to \eqref{eqs:GK} was previously investigated in \cite{Gal2016,Gal2019}. It involves a second-order Allen--Cahn type dynamic boundary condition instead of \eqref{eqs:GK:6}-\eqref{eqs:GK:7} to account for the contact angle dynamics.

In system \eqref{eqs:GK}, $\Om\subset\R^d$ with $d\in\{2,3\}$ is a bounded domain with boundary $\Ga \coloneqq \partial\Om$, $T > 0$ is a prescribed final time, and we set $Q = \Om\times(0,T)$ and $\Sigma = \Ga\times(0,T)$ for brevity. The outward unit normal vector field on $\Ga$ is denoted by $\n$, while $\Gradg, \Divg$ and $\Lapg$ stand for the surface gradient, the surface divergence and the Laplace--Beltrami operator on $\Ga$, respectively.
Furthermore, $\bv:Q\rightarrow\R^d$ denotes the bulk velocity field, which describes the motion of the mixture in $\Om$. The location of the two materials in the bulk is represented by the bulk phase-field $\phi:Q\rightarrow\R$. Its trace on the boundary $\Ga$, the surface phase-field, is denoted by $\psi = \phi\vert_\Ga :\Sigma\rightarrow\R$. The functions $\mu:Q\rightarrow\R$ and $\theta:\Sigma\rightarrow\R$ stand for the bulk chemical potential and the surface chemical potential, respectively.

The time evolution of $\bv$ is described by the bulk Navier--Stokes equation \eqref{eqs:GK:1}. The flow in the bulk is assumed to be incompressible, that is, the velocity field is divergence-free. The bulk Navier--Stokes equation is complemented with the boundary condition \eqref{eqs:GK:2}, which can be regarded as an inhomogeneous Navier-slip boundary condition. It enhances the Navier-slip boundary condition previously proposed in \cite{Qian2006}. 
Here, $\gamma(\phi)$ represents the slip parameter, which accounts for tangential friction at the boundary and may depend on the phase-field.
In \eqref{eqs:GK:1}, $\rho(\phi)$ stands for the phase-field dependent density in the bulk, and is given by
\begin{align}\label{System:9}
    \rho(\phi) \coloneqq \frac{\tilde\rho_2 - \tilde\rho_1}{2}\phi + \frac{\tilde\rho_2 + \tilde\rho_1}{2} \qquad\text{in~}Q,
\end{align}
where $\tilde\rho_1$ and $\tilde\rho_2$ denote the individual constant densities of the two fluids, respectively. As the derivative of $\rho$ is constant, we often write
\begin{align*}
    \rho^\prime \equiv \rho^\prime(\phi) = \frac{\tilde\rho_2 - \tilde\rho_1}{2}
\end{align*}
for simplicity. Furthermore, $\J$ represents the \textit{mass flux} related to interfacial motion, and $\T$ is the \textit{stress tensor}. They are given by the formulas
\begin{alignat}{2}
    \J &= - \rho^\prime \,
    m_\Om(\phi)\Grad\mu &&\qquad\text{in~}Q, \label{System:14} \\
    \T &= \S - p\I - \varepsilon\Grad\phi\otimes\Grad\phi \quad\text{with}\quad \S \coloneqq 2\nu_\Om(\phi)\D\bv &&\qquad\text{in~}Q. \label{System:11}
\end{alignat}
The coefficient $\nu_\Om:\R\rightarrow\R$ is the viscosity of the mixtures in the bulk, which usually depends on the phase-field. 
Provided that both fluids have a constant viscosity that is denoted by $\nu_1$ and $\nu_2$, respectively, a typical choice is the affine interpolation
\begin{align*}
    \nu_\Om(\phi) = \frac{\nu_2 - \nu_1}{2}\phi + \frac{\nu_2 + \nu_1}{2} \qquad\text{in~}Q.
\end{align*}
The matrix $\S$ is referred to as the \textit{viscous stress tensor}. It depends on the \textit{symmetric gradient} of $\bv$, which is defined as
\begin{align*}
    \D\bv \coloneqq \frac12\big((\Grad\bv) + (\Grad\bv)^\top\big).
\end{align*}

The time evolution of $\phi, \psi, \mu$ and $\theta$ is given by the \textit{convective bulk-surface Cahn--Hilliard subsystem} \eqref{eqs:GK:4}-\eqref{eqs:GK:8}, which consists of a \textit{bulk Cahn--Hilliard equation} \eqref{eqs:GK:4}-\eqref{eqs:GK:5}, that is coupled to a \textit{surface Cahn--Hilliard equation} \eqref{eqs:GK:6}-\eqref{eqs:GK:7} through the normal derivatives $\deln\phi$ and $\deln\mu$, as well as through the boundary conditions \eqref{eqs:GK:8}, which has to be understood in the following sense:
\begin{align}
    \begin{cases}
        \mu = \beta\theta &\text{if~} L = 0, \\
        L m_\Om(\phi)\deln\mu = \beta\theta - \mu &\text{if~} L\in(0,\infty), \\
        m_\Om(\phi)\deln\mu = 0 &\text{if~} L = \infty,
    \end{cases}
    \qquad\text{on~}\Sigma.
\end{align}

The functions $m_\Om, m_\Ga:\R\rightarrow\R$ are the so-called \textit{Onsager mobilities}. In general, they may depend on the phase-field variables $\phi$ and $\psi$, respectively, and describe both the spatial distribution and the magnitude of the underlying diffusion mechanisms.

The functions $F^\prime$ and $G^\prime$ are the derivatives of double-well potentials $F$ and $G$, respectively. A physically motivated example of such a double-well potential, especially in applications related to materials science, is the \textit{logarithmic potential}, which is also referred to as the \textit{Flory--Huggins potential}. It is defined as
\begin{align}\label{LogPot}
    W_{\mathrm{log}}(s) = \frac{\Theta}{2}\big[(1+s)\ln(1+s) + (1-s)\ln(1-s)\big] - \frac{\Theta_c}{2}s^2 \qquad\text{for~}s\in[-1,1].
\end{align}
Here, we assume that $0 < \Theta < \Theta_c$, where $\Theta$ denotes the temperature of the mixture, and $\Theta_c$ represents the critical temperature below which phase separation processes are to be expected. We point out that, since $W_{\mathrm{log}}^\prime(s) \rightarrow\pm\infty$ as $s\rightarrow\pm 1$, the potential $W_{\mathrm{log}}$ is classified as a singular potential.

Furthermore, $\varepsilon$ is a parameter related to the thickness of the diffuse interface in the bulk, whereas $\delta > 0$ corresponds to the width of the diffuse interface on the boundary. The parameter $\kappa \ge 0$ serves as a weight for the surface Dirichlet energy $\psi\mapsto\intG\abs{\Gradg\psi}^2\dG$. If $\kappa>0$, this term has a regularizing effect on the phase separation at the boundary.

Although the model \eqref{eqs:GK} allows for unmatched densities, variable contact angles, an improved description of contact line movement via the generalized Navier boundary condition \eqref{eqs:GK:2}, and a transfer of material between bulk and surface (if $L<\infty$), it still exhibits some important limitations. 

\subsection{Limitations concerning the description of surface materials}

In many applications from materials science, chemistry, engineering, or life sciences, the materials on the boundary might actually be different from those in the bulk. In such situations, the rheological properties of the surface materials often need to be taken into account, cf.~\cite{Sagis2011}.

An important example is the classical \textit{fluid-mosaic model} for the structure of cell membranes from the seminal work \cite{Singer1972}. There, biological membranes are conceptualized as laterally incompressible, two-dimensional viscous fluid layers with embedded proteins.
This justifies the treatment of the boundary as an active viscous layer supporting both lateral momentum transport and compositional evolution.
These concepts were formalized mathematically in \cite{Arroyo2009} as a continuum model, where lipid bilayers are described by a \textit{surface Navier--Stokes equation} with positive, but finite viscosity, which is expected to dominate relaxation dynamics. These findings and observations motivate the inclusion of a \textit{surface Navier--Stokes equation} in system \eqref{eqs:GK} to account for viscous and convective effects at the boundary. 

\FloatBarrier
\begin{figure}[h!]
    \centering
    \includegraphics[width=0.8\textwidth]{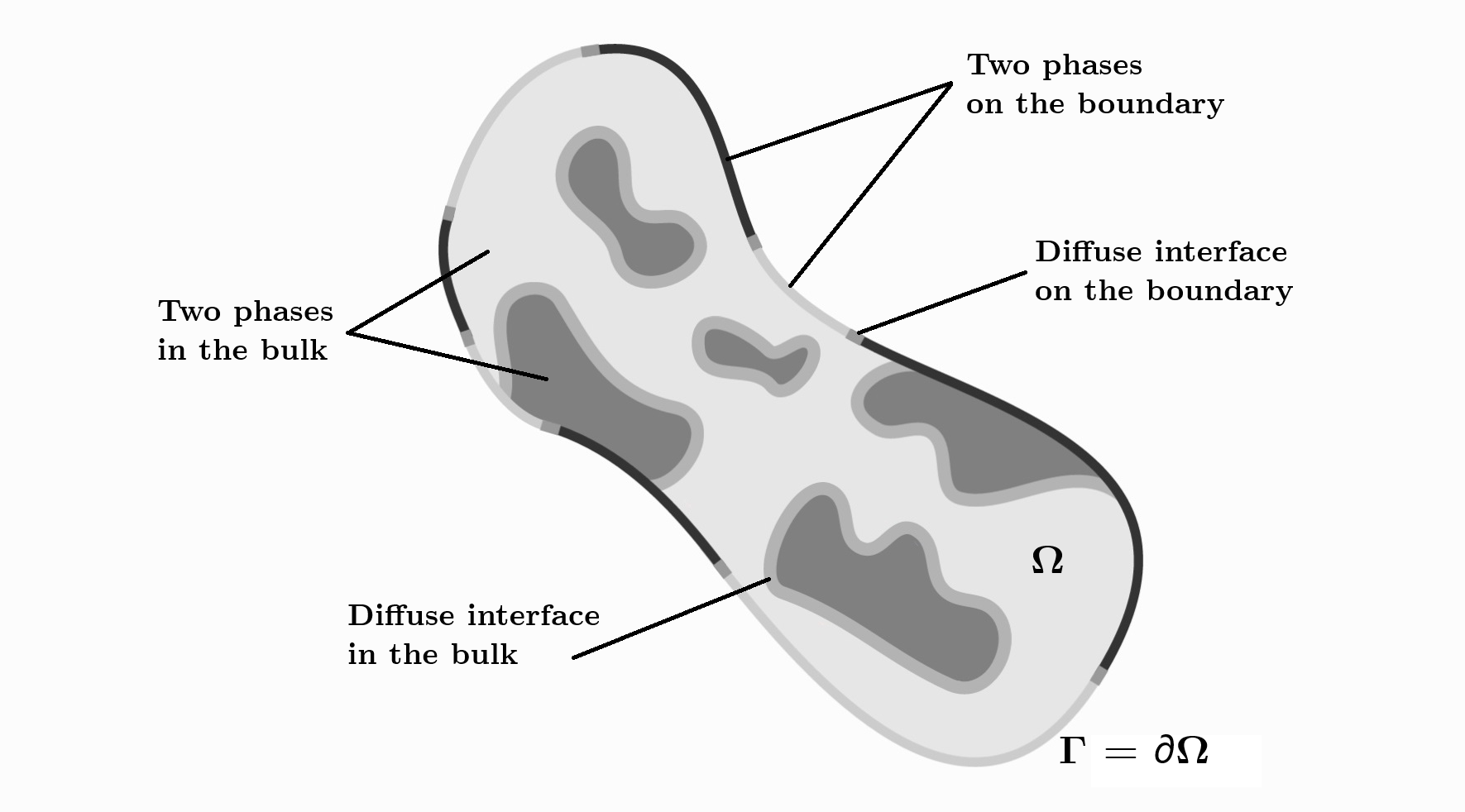}
    \caption{Sketch of a domain (representing, e.g., a biological cell) with two phases (fluids) in the bulk and two phases (fluids) on the boundary, each separated by a diffuse interface.}
\end{figure}
\FloatBarrier

However, if the materials on the boundary actually differ from those in the bulk, 
there is no particular reason why a trace relation between $\phi$ and $\psi$, as imposed by $\eqref{eqs:GK:8}_1$, should be fulfilled.
To account for the interaction of different materials, the more general transmission condition
\begin{align}
    \label{BoundaryCondition:KL*}
    \varepsilon K\deln\phi = H(\psi) - \phi \qquad\text{on~}\Sigma
\end{align}
was proposed in \cite{Colli2019a,Knopf2020}. It involves a parameter $K\in[0,\infty]$ and a transmission function $H$. 
Similar to the boundary condition on the chemical potentials, \eqref{BoundaryCondition:KL*} is to be understood as
\begin{align}\label{BoundaryCondition:KL}
    \begin{cases}
        \phi = H(\psi) &\text{if~}K = 0, \\
        \varepsilon K\deln\phi = H(\psi) - \phi &\text{if~}K\in(0,\infty), \\
        \deln\phi = 0 &\text{if~} K = \infty,
    \end{cases}
    \qquad\text{on~}\Sigma.
\end{align}
The most popular choice is a linear transmission condition, which means that $H$ is given by $H(s) = \alpha s$ for all $s\in\R$ and some constant $\alpha\in\R$.

A further mathematical limitation of the model \eqref{eqs:GK} is that, at least up to now, the existence of a global weak solution in the situation of unmatched densities is ensured only if $L\in (0,\infty]$ (see \cite{Gal2023a}). In the case $L=0$, the existence of a global weak solution was established in \cite{Giorgini2023,Gal2023a} only in the case of matched densities.
Moreover, in all cases $L\in [0,\infty]$, the regularity of weak solutions to system \eqref{eqs:GK} is rather limited. Regularity results for the bulk-surface convective Cahn--Hilliard subsystem (with prescribed velocity fields) have been obtained in \cite{Knopf2024, Knopf2025, Giorgini2025}. However, these results are not applicable for weak solutions of \eqref{eqs:GK} as the trace of the velocity field on the boundary is not regular enough. In the present paper, we will see that the additional inclusion of a surface Navier--Stokes equation helps to overcome this issue, as the inclusion of internal friction has a regularizing effect on the velocity field on the boundary. As discussed in Remark~\ref{REM:REG:PP}, the regularity results established in \cite{Giorgini2025} are actually applicable for weak solutions of our new model, and in this way, the regularity of the phase-fields $\phi$ and $\psi$ can be increased significantly.

\subsection{A new bulk-surface Navier--Stokes--Cahn--Hilliard model}

In the present paper, we derive and analyze a new bulk-surface Navier--Stokes--Cahn--Hilliard model, which reads as follows:
\begin{subequations}\label{System}
\begin{alignat}{3}
    \label{System:1}
    &\delt(\rho(\phi)\bv) + \Div(\bv\otimes(\rho(\phi)\bv + \J)) 
    = \Div\;\T , 
    &&\quad\Div\;\bv = 0 
    &&\qquad\text{in~}Q,
    \\
    \label{System:2}
    &\delt(\sigma(\psi)\bw) + \Divgt(\bw\otimes(\sigma(\psi)\bw + \K)) = \Divgt\T_\Ga + \Z, 
    &&\quad \Divg\;\bw = 0 
    &&\qquad\text{on~}\Sigma,
    \\
    \label{System:3}
    &\bv\vert_\Ga = \bw,\qquad \bv\cdot\n = 0 
    &&{}
    &&\qquad\text{on~}\Sigma,
    \\[0.5em]
    \label{System:4}
    &\delt\phi  + \Div(\phi\bv) = \Div(m_\Om(\phi)\Grad\mu) 
    &&{}
    &&\qquad\text{in~}Q, 
    \\
    \label{System:5}
    &\mu = -\ep\Lap\phi + \tfrac1\ep F^\prime(\phi) 
    &&{}
    &&\qquad\text{in~}Q, 
    \\
    \label{System:6}
    &\delt\psi + \Divg(\psi\bw) = \Divg(m_\Ga(\psi)\Gradg\theta) - \beta m_\Om(\phi)\deln\mu 
    &&{}
    &&\qquad\text{on~}\Sigma, 
    \\
    \label{System:7}
    &\theta = -\delta\kappa\Lapg\psi + \tfrac1\delta G^\prime(\psi) + \alpha\ep\deln\phi 
    &&{}
    &&\qquad\text{on~}\Sigma, 
    \\
    \label{System:8}
    &\ep K\deln\phi = \alpha\psi - \phi, \qquad Lm_\Om(\phi)\deln\mu = \beta\theta - \mu, &&\quad K,L\in[0,\infty], &&\qquad\text{on~}\Sigma.
\end{alignat}
\end{subequations}
Here, viscous dynamics on the boundary are represented by an additional surface velocity field $\bw:\Sigma\to\R^d$. In this model, we assume that the bulk and surface velocity fields are coupled through the trace relation \eqref{System:3}$_1$. This means that the motions of the bulk materials and the surface materials are directly related. Alternatively, stress-based coupling conditions between $\bv$ and $\bw$ might also be conceivable and could be discussed in future contributions.
As we demand the \textit{non-penetration boundary condition} \eqref{System:3}$_2$, the trace relation \eqref{System:3}$_1$ further entails that the surface velocity field $\bw$ is an evolving \textit{tangential} vector field on $\Gamma$.
The location of the two materials in the bulk is represented by the bulk phase-field $\phi:Q\to\R$, whereas the two materials on the boundary are represented by the surface phase-field $\psi:\Sigma\to \R$.

The time evolution of $\bv$ and $\bw$ is described by the \textit{bulk Navier--Stokes equation} \eqref{System:1} and the \textit{tangential surface Navier--Stokes equation} \eqref{System:2}, respectively. Both the flow in the bulk and the flow on the surface are assumed to be incompressible, that is, the velocity fields are divergence-free. Here, $\sigma(\psi)$ denotes the phase-field dependent density on the surface, and is given by
\begin{alignat}{2}
    \label{System:10}
    \sigma(\psi) &\coloneqq \frac{\tilde{\sigma}_2 - \tilde{\sigma}_1}{2}\psi + \frac{\tilde{\sigma}_2 + \tilde{\sigma}_1}{2} &&\qquad\text{on~} \Sigma.
\end{alignat}
Similarly to the density in the bulk, the derivative of $\sigma$ is constant. Therefore, we often write
\begin{align*}
    \sigma^\prime \equiv \sigma^\prime(\psi)
    = \frac{\tilde{\sigma}_2 - \tilde{\sigma}_1}{2}\,.
\end{align*}
Furthermore, $\T_\Ga$ is a stress tensor, whereas $\Z$ represents a surface force density that accounts for mechanical interactions between the bulk materials and the surface materials. They are given by
\begin{alignat}{2}
    \label{System:12}
    \T_\Ga &\coloneqq \S_\Ga - q\I - \delta\kappa\Gradg\psi\otimes\Gradg\psi \quad\text{with}\quad \S_\Ga\coloneqq 2\nu_\Ga(\psi)\Dg\bw &&\qquad\text{on~}\Sigma, \\
    \label{System:13}
    \Z &\coloneqq - \big[\S\,\n\big]_\tau + \tfrac12(\J\cdot\n)\bw + \tfrac12(\J_\Ga\cdot\n)\bw + \alpha\ep\deln\phi\, \Gradg\psi - \gamma(\phi,\psi)\bw &&\qquad\text{on~}\Sigma, 
\end{alignat}
The coefficient $\nu_\Ga:\R\rightarrow\R$ denotes the viscosity of the mixtures on the surface. As above, $\gamma(\phi,\psi)$ denotes a slip coefficient associated with tangential friction. It may now depend on both phase-fields, as they are not necessarily coupled via the trace relation as in \eqref{eqs:GK:8}. Moreover,
\begin{align*}
    \Dg\bw \coloneqq \frac12\big((\Gradg\bw) + (\Gradg\bw)^\top\big)
\end{align*}
denotes the \textit{surface symmetric gradient} of $\bw$. The quantity $\K$ denotes a mass flux related to the interfacial motion on the boundary, while $\J_\Ga$ represents the mass flux corresponding to the transfer of material between bulk and surface. They are given by
\begin{alignat}{2}
    \label{System:15}
    \K &\coloneqq - \sigma^\prime\,
    \qquad
    \J_\Ga\coloneqq - \beta \rho^\prime\,
    &&\qquad\text{on~}\Sigma.
\end{alignat}

Employing the usual decompositions
\begin{alignat*}{2}
    -\ep\Div(\Grad\phi\otimes\Grad\phi) &= -\Grad\left(\frac\ep2\abs{\Grad\phi}^2 + \frac1\ep F(\phi)\right) + \mu\Grad\phi &&\qquad\text{in~}Q, \\
    -\delta\kappa\Divg(\Gradg\psi\otimes\Gradg\psi) &= -\Gradg\left(\frac{\delta\kappa}{2}\abs{\Gradg\psi}^2 + \frac1\delta G(\psi)\right) + \theta\Gradg\psi - \alpha\ep\deln\phi\Gradg\psi &&\qquad\text{on~}\Sigma,
\end{alignat*}
and replacing the bulk and surface pressure by
\begin{alignat*}{2}
    \overline{p} &\coloneqq p + \frac\ep2\abs{\Grad\phi}^2 + \frac1\ep F(\phi) &&\qquad\text{in~}Q, \\
    \overline{q} &\coloneqq q + \frac{\delta\kappa}{2}\abs{\Gradg\psi}^2 + \frac1\delta G(\psi) &&\qquad\text{on~}\Sigma,
\end{alignat*}
respectively, we can rewrite \eqref{System:1}-\eqref{System:2} as
\begin{subequations}\label{eqs:NSCH:pressure}
    \begin{alignat}{2}
        &\delt(\rho(\phi)\bv) + \Div(\bv\otimes(\rho(\phi)\bv + \J)) - \Div(2\nu_\Om(\phi)\D\bv) + \Grad\overline{p} = \mu\Grad\phi &&\qquad\text{in~}Q, \\
        &\delt(\sigma(\psi)\bw) + \Divgt(\bw\otimes(\sigma(\psi)\bw + \K)) - \Divgt(2\nu_\Ga(\psi)\Dg\bw) + \Gradg\overline{q}\nonumber \\
        &\qquad = \theta\Gradg\psi - 2\nu_\Om(\phi)\big[ \D\bv\,\n \big]_\tau + \tfrac12(\J\cdot\n)\bw + \tfrac12(\J_\Ga\cdot\n)\bw - \gamma(\phi,\psi)\bw &&\qquad\text{on~}\Sigma.
    \end{alignat}
\end{subequations}
Moreover, for the subsequent analysis, it will prove useful to rewrite \eqref{eqs:NSCH:pressure} in the so-called \textit{non-conservative form}
\begin{subequations}\label{eqs:NSCH:NC}
    \begin{alignat}{2}
        &\rho(\phi)\delt\bv + \rho(\phi)(\bv\cdot\Grad)\bv + (\J\cdot\Grad)\bv - \Div(2\nu_\Om(\phi)\D\bv) + \Grad\overline{p} = \mu\Grad\phi &&\qquad\text{in~}Q, \label{ConsForm:v} \\
        &\sigma(\psi)\delt\bw + \sigma(\psi)(\bw\cdot\Gradg)\bw + (\K\cdot\Gradg)\bw - \Divgt(2\nu_\Ga(\psi)\Dg\bw) + \Gradg\overline{q}\nonumber \\
        &\qquad = \theta\Gradg\psi - 2\nu_\Om(\phi)\big[ \D\bv\,\n \big]_\tau + \tfrac12(\J\cdot\n)\bw - \tfrac12(\J_\Ga\cdot\n)\bw - \gamma(\phi,\psi)\bw &&\qquad\text{on~}\Sigma. \label{ConsForm:w}
    \end{alignat}
\end{subequations}

The energy functional associated with system \eqref{System} reads as
\begin{align}
    \begin{split}
        E(\bv,\bw,\phi,\psi) &= \intO \frac{\rho(\phi)}{2}\abs{\bv}^2\dx + \intG \frac{\sigma(\psi)}{2}\abs{\bw}^2\dG + \intO \frac\ep2\abs{\Grad\phi}^2 + \frac1\ep F(\phi)\dx 
        \\
        &\quad + \intG \frac{\delta\kappa}{2}\abs{\Gradg\psi}^2 + \frac1\delta G(\psi) + \frac12\chi(K)(\alpha\psi - \phi)^2 \dG.
    \end{split}
\end{align}
Here, the function
\begin{align}\label{DEF:H}
    \chi:[0,\infty]\rightarrow[0,\infty), \quad \chi(r) \coloneqq\begin{cases}
        r^{-1}, &\text{if~}r\in(0,\infty), \\
        0, &\text{if~} r\in\{0,\infty\},
    \end{cases}
\end{align}
is used to distinguish the different cases of the parameters $K$ and $L$. Sufficiently regular solutions to system \eqref{System} satisfy the \emph{energy dissipation law}
\begin{align}\label{diss}
    &\ddt E(\bv,\bw,\phi,\psi) + \intO 2\nu_\Om(\phi)\abs{\D\bv}^2\dx + \intG 2\nu_\Ga(\psi)\abs{\Dg\bw}^2\dG + \intG \gamma(\phi,\psi)\abs{\bw}^2 \dG \\
    &\qquad + \intO m_\Om(\phi)\abs{\Grad\mu}^2\dx + \intG m_\Ga(\psi)\abs{\Gradg\theta}^2\dG + \chi(L)\intG (\beta\theta - \mu)^2\dG = 0 \nonumber
\end{align}
on $[0,T]$
as well as the \emph{mass conservation law}
\begin{align}\label{masscons}
    \begin{dcases}
        \beta\intO \phi(t)\dx + \intG \psi(t)\dG = \beta\intO \phi(0) \dx + \intG \psi(0)\dG, &\textnormal{if~} L\in[0,\infty), \\
        \intO\phi(t)\dx = \intO\phi(0)\dx \quad\textnormal{and}\quad \intG\psi(t)\dG = \intG\psi(0)\dG, &\textnormal{if~} L = \infty,
    \end{dcases}
\end{align}
for all $t\in[0,T]$. This means that the mass is conserved separately in $\Om$ and $\Ga$ if $L = \infty$, while in the case $L\in[0,\infty)$, a transfer of material between bulk and surface is expected to occur (cf.~\cite{Knopf2021a}).

For prescribed velocity fields $\bv$ and $\bw$, the convective bulk-surface Cahn--Hilliard subsystem \eqref{System:4}-\eqref{System:7} subject to the coupling conditions \eqref{System:8} with $K,L\in[0,\infty]$ was first analyzed in \cite{Knopf2024} in the case of regular potentials, and then in \cite{Knopf2025} and \cite{Giorgini2025} for singular potentials.
The \textit{non-convective} version of \eqref{System:4}-\eqref{System:7} (i.e., $\bv\equiv 0$ and $\bw\equiv 0$) has been the subject of extensive investigation in the literature. We refer to \cite{Colli2014,Fukao2021,Garcke2020,Garcke2022,Goldstein2011,Miranville2020,Liu2019,Lv2024,Lv2024a,Lv2024b,Stange2025}
for a selection of representative contributions. For comprehensive overviews on the Cahn--Hilliard equation with classical or dynamic boundary conditions, we recommend the recent review paper \cite{Wu2022} as well as the book \cite{Miranville-Book}.

\subsection{Strategy of the mathematical analysis and plan of the paper}

Our main result (Theorem~\ref{Theorem:Main}) establishes the existence of a global-in-time weak solution to \eqref{System} for arbitrary coupling parameters $K,L\in[0,\infty]$ in both two and three dimensions. Since the convective bulk-surface Cahn--Hilliard subsystem has already been investigated in detail (see \cite{Knopf2025, Giorgini2025}), the central idea of our proof is to employ a semi-Galerkin scheme, in which the velocity fields are approximated on suitable finite-dimensional subspaces. For standard boundary conditions, such as no-slip or Navier-slip boundary conditions, this is typically achieved through the use of eigenfunctions of the Stokes operator subject to the corresponding boundary condition. However, in the present setting, the novel setting of dynamic boundary conditions \eqref{System:2}-\eqref{System:3} (especially the surface divergence condition \eqref{System:2}$_2$) imposed on the velocity field renders these eigenfunctions unsuitable. To overcome this difficulty, we introduce and analyze a related bulk-surface Stokes system along with the corresponding bulk-surface Stokes operator. By means of spectral theory, we show that this operator admits a countable family of eigenvalues with associated eigenfunctions, which in turn serve as a natural basis for approximating the velocity fields $\bv$ and $\bw$. 

This argument, however, can only be carried out in the case $L\in(0,\infty]$, since if $L = 0$, admissible test functions are required to satisfy a specific trace relation. To address this issue, we establish the existence of weak solutions for $L = 0$ by considering the asymptotic limit $L \rightarrow 0$ of weak solutions associated with $L > 0$. However, in the case $L = 0$, we require the densities of the bulk and surface materials to satisfy the compatibility condition
\begin{align}\label{CompCond:Denisities}
    \beta(\tilde\sigma_2 - \tilde\sigma_1) = -(\tilde\rho_2 - \tilde\rho_1).
\end{align}
It arises because our boundary condition on $\bv$ involves not only the flux term $\J\cdot\n$, but also the surface flux term $\J_\Ga\cdot\n$ (see \eqref{System:13}). Note that while for the system \eqref{eqs:GK} from \cite{Giorgini2023} the term $\J\cdot\n$ led to the requirement of matched densities (i.e., $\tilde\rho_1=\tilde\rho_2$) if $L = 0$ (see \cite{Gal2023a}), our compatibility condition \eqref{CompCond:Denisities} is more general and even allows for unmatched densities in the case $L = 0$.

The remainder of this paper is structured as follows. In Section~\ref{Section:Preliminaries}, we present all the necessary notation and preliminaries. Then, Section~\ref{Section:Model} is devoted to the derivation of the bulk-surface Navier--Stokes--Cahn--Hilliard model \eqref{System}, which is carried out by means of local mass balance laws and local energy dissipation laws (both in the bulk and on the surface) as well as the Lagrange multiplier approach. Afterwards, in Section~\ref{Section:MainResults}, we introduce the notion of a weak solution to \eqref{System} and state our main result regarding the existence of a global-in-time weak solution. In Section~\ref{Section:BulkSurfaceStokes}, we study a new bulk-surface Stokes system together with its associated bulk-surface Stokes operator. Section~\ref{Section:Proof:L>0} is dedicated to the construction of a weak solution in the case $L\in(0,\infty]$. Lastly, in Section~\ref{Section:Proof:L=0}, we investigate the asymptotic limit $L\rightarrow 0$, which proves the existence of a weak solution to \eqref{System} in the limit case $L=0$.

\section{Notation and preliminaries}
\label{Section:Preliminaries}

\subsection{General notation and function spaces}

We write $\N$ to denote the set of natural numbers (excluding zero). For any real Banach space $X$ with $\norm{\cdot}_X$, its dual space is denoted by $X^\prime$. The corresponding duality pairing of elements $\phi\in X^\prime$ and $\zeta\in X$ is denoted by $\ang{\phi}{\zeta}_X$. If $X$ is a Hilbert space, we write $\scp{\cdot}{\cdot}_X$ to denote its inner product. Given $p\in[1,\infty]$ and an interval $I\subset[0,\infty)$, the Bochner space $L^p(I,X)$ consists of all Bochner measurable, $p$-integrable functions defined on $I$ with values in $X$. We further define $W^{1,p}(I;X)$ as the space of all functions $f\in L^p(I;X)$ with vector-valued distributional time derivative $\delt f\in L^p(I;X)$. In particular, we set $H^1(I;X) = W^{1,2}(I;X)$. The set of continuous functions $f:I\rightarrow X$ is denoted by $C(I;X)$. Moreover, $C_w(I;X)$ denotes the space of functions $f:I\to X$ which are continuous on $I$ with respect to the weak topology on $X$, i.e., the map $I\ni t\mapsto \ang{\phi}{f(t)}_X$ is continuous for all $\phi\in X^\prime$. 
At this point, we also recall the following lemma, which can be found in \cite[Lemma~4.1]{Abels2009}.
\begin{lemma}\label{Prelim:Lemma:BC_w}
    Let $I\subset\R$ be a compact interval, and let $X,Y$ be Banach spaces with $Y\emb X$ and $X^\prime\emb Y^\prime$ densely. Then the continuous embedding $L^\infty(I;Y)\cap C(I;X)\emb C_w(I;Y)$ holds.
\end{lemma}

In the remainder of this section, let $\Om\subset\R^d$, $d=2,3$, be a bounded domain with sufficiently smooth boundary $\Ga\coloneqq\partial\Om$, and let $\n$ denote the exterior unit normal vector field on $\Gamma$. For any $s\geq 0$ and $p\in[1,\infty]$, the Lebesgue and Sobolev–Slobodeckij spaces for functions mapping from $\Om$ to $\R$ are denoted as $L^p(\Om)$ and $W^{s,p}(\Om)$, respectively. We write $\norm{\cdot}_{L^p(\Om)}$ and $\norm{\cdot}_{W^{s,p}(\Om)}$ to denote the standard norms on these spaces. If $p = 2$, we use the notation $H^s(\Om) = W^{s,2}(\Om)$, with the convention that $H^0(\Om)$ is identified with $L^2(\Om)$. For the Lebesgue and Sobolev–Slobodeckij spaces on the boundary $\Ga$, we use an analogous notation. The corresponding spaces of vector-valued mapping into $\R^d$ or matrix-valued functions mapping into $\R^{d\times d}$ are denoted by boldface letters, namely $\mathbf L^p$, $\mathbf W^{s,p}$ and $\mathbf H^s$. In particular, we write $\mathbf L^p_\tau(\Gamma)$, $\mathbf W^{s,p}_\tau(\Gamma)$ and $\mathbf H^s_\tau(\Gamma)$ to denote the corresponding spaces of tangential vector fields on $\Gamma$.

Moreover, for $p\in[1,\infty]$, we further introduce the first-order homogeneous Sobolev spaces
\begin{align*}
    \dot W^{1,p}(\Omega) 
    &\coloneqq
    \big\{ u \in L^p_\mathrm{loc}(\Omega) 
    \,:\, \Grad u \in \mathbf L^p(\Omega) \big\}\big/_{\sim} \,,
    \\
    \dot W^{1,p}(\Gamma) 
    &\coloneqq
    \big\{ v \in L^p(\Gamma) 
    \,:\, \Gradg v \in \mathbf L^p(\Gamma) \big\}\big/_{\sim} \,.
\end{align*}
Here, $L^p_\mathrm{loc}(\Omega)$ is the collection of all functions, which belong to $L^p(K)$ for every compact subset $K\subset \Omega$. On the boundary, we simply consider $L^p(\Gamma)$, since $\Gamma$ is itself a compact submanifold.
Moreover, the notation $/_{\sim}\,$ indicates that functions differing only by an additive constant are considered equivalent and thus identified with each other. Consequently, the functionals ${\norm{\Grad\,\cdot}_{\mathbf L^p(\Omega)}}$ and ${\norm{\Gradg\,\cdot}_{\mathbf L^p(\Gamma)}}$ define norms on the spaces $\dot W^{1,p}(\Omega)$ and $\dot W^{1,p}(\Gamma)$, respectively.

Furthermore, we write $\mathbf{P}^\Gamma$ to denote the orthogonal projection of $\R^d$ into the tangent space of the submanifold $\Gamma$.
This means that $\mathbf{P}^\Gamma$ is a linear map, which can be represented as
\begin{align*}
    \mathbf{P}^\Gamma = \mathbf{I} - \n\otimes\n,
\end{align*}
where $\mathbf{I}\in\R^{d\times d}$ is the unity matrix.
For any vector field $\bw$ on $\Gamma$, we also use the notation
\begin{align*}
    [\bw]_\tau = \mathbf{P}^\Gamma(\bw).
\end{align*}

\subsection{Differential operators in the bulk and on the surface} \label{NOT:DIFFOP}
Let $f:\Omega\to\R$ and $g:\Gamma\to\R$ be scalar functions, let $\bv:\Omega\to \R^d$ and $\bw:\Gamma\to\R^d$ be vector fields, and let $\mathbf{A}:\Omega\to\R^{d\times d}$ and $\mathbf{B}:\Gamma\to\R^{d\times d}$ be matrix-valued functions. 
For now, we assume that the boundary $\Gamma$ and all these functions are sufficiently regular.
As usual, $\Grad f$, $\Div\,\bv$, and $\Lap f = \Div(\Grad f)$ denote the gradient of $f$, the divergence of $\bv$ and the Laplacian of $f$. 
Analogously, $\Gradg g$, $\Divg\,\bw$, and $\Lapg g = \Divg(\Gradg g)$ are the tangential gradient of $g$, the surface divergence of $\bw$ and the Laplace--Beltrami operator applied to $g$. 
We point out that $\Grad f$ and $\Gradg g$ are to be understood as column vectors in $\R^d$.
For bulk vector fields, the gradient is defined as
\begin{align*}
    \Grad \bv = 
    \begin{pmatrix}
        (\Grad \bv_1)^T \\ \vdots \\ (\Grad \bv_d)^T
    \end{pmatrix}
    \in \R^{d\times d},
\end{align*}
where $\bv_1,...,\bv_d$ denote the components of $\bv$. In other words, $\Grad\bv$ is the Jacobian of $\bv$. On the surface $\Gamma$, the projective (covariant) gradient of $\bw$ is defined as 
\begin{align*}
    \Gradg\bw = \mathbf{P}^\Gamma \Grad\tw \mathbf{P}^\Gamma,
\end{align*}
where $\tw$ is an extension of $\bw$ to an open neighborhood of $\Gamma$ in $\R^d$. However, the expression $\Gradg\bw$ does not depend on the concrete choice of this extension. If $\bw$ is a tangential vector field, its gradient can also be expressed as
\begin{align*}
    \Gradg \bw = \mathbf{P}^\Gamma
    \begin{pmatrix}
        (\Gradg \bw_1)^T \\ \vdots \\ (\Gradg \bw_d)^T
    \end{pmatrix}
    \in \R^{d\times d},
\end{align*}
where $\bw_1,...,\bw_d$ denote the components of $\bw$.  
For the matrix-valued functions $\mathbf{A}$ and $\mathbf{B}$, we further define the divergences
\begin{align*}
    \Div\,\mathbf{A} = 
    \begin{pmatrix}
        \Div\, \mathbf{A}_{1\ast} \\ \vdots \\ \Div\, \mathbf{A}_{d\ast}
    \end{pmatrix},
    \qquad
    \Divg\,\mathbf{B} = \mathbf{P}^\Gamma
    \begin{pmatrix}
        \Divg\, \mathbf{B}_{1\ast} 
        \\ 
        \vdots 
        \\ \Divg\, \mathbf{B}_{d\ast}  
    \end{pmatrix}.
\end{align*}
Here, $\mathbf{A}_{1\ast},...,\mathbf{A}_{d\ast}$ and $\mathbf{B}_{1\ast},...,\mathbf{B}_{d\ast}$ denote the rows of $\mathbf{A}$ and $\mathbf{B}$, respectively.

\subsection{Bulk-surface product spaces}
For any $s\geq 0$ and $p\in[1,\infty]$, we define
\begin{align*}
	\mathcal{L}^p \coloneqq L^p(\Om)\times L^p(\Ga), \qquad \mathcal{W}^{s,p} \coloneqq W^{s,p}(\Om)\times W^{s,p}(\Ga).
\end{align*}
We write $\mathcal{H}^s = \mathcal{W}^{s,2}$ and identify $\mathcal{L}^2$ with $\mathcal{H}^0$. We point out that $\mathcal{H}^s$ is a Hilbert space with respect to the inner product
\begin{align*}
	\bigscp{\scp{\phi}{\psi}}{\scp{\zeta}{\xi}}_{\mathcal{H}^s} \coloneqq \scp{\phi}{\zeta}_{H^s(\Om)} + \scp{\psi}{\xi}_{H^s(\Ga)} \qquad\text{for all~}(\phi,\psi), (\zeta,\xi)\in\mathcal{H}^s,
\end{align*}
and its induced norm $\norm{\cdot}_{\mathcal{H}^s} \coloneqq \scp{\cdot}{\cdot}_{\mathcal{H}^s}^{1/2}$. We further recall that the duality pairing on $\mathcal{H}^1$ satisfies
\begin{align*}
	\bigang{\scp{\phi}{\psi}}{\scp{\zeta}{\xi}}_{\mathcal{H}^1} = \scp{\phi}{\zeta}_{L^2(\Om)} + \scp{\psi}{\xi}_{L^2(\Ga)}
\end{align*}
for all $(\zeta,\xi)\in\mathcal{H}^1$ if $(\phi,\psi)\in\mathcal{L}^2$. 

For every $\phi\in H^1(\Om)^\prime$, we denote by $\meano{\phi} = \abs{\Om}^{-1}\ang{\phi}{1}_{H^1(\Om)}$ its generalized mean value over $\Om$. If $\phi\in L^1(\Om)$, its spatial mean can simply be expressed as $\meano{\phi} = \abs{\Om}^{-1}\intO\phi\dx$. The spatial mean of any $\psi\in H^1(\Ga)^\prime$, denoted by $\meang{\psi}$, is defined similarly. We further define for any $p\in[2,\infty]$ the space
\begin{align*}
    \mathcal{L}^p_{(0)} \coloneqq L^p_{(0)}(\Om)\times L^p_{(0)}(\Ga),
\end{align*}
where
\begin{align*}
	L^p_{(0)}(\Om) 
    \coloneqq \big\{\phi\in L^p(\Om): \meano{\phi} = 0\big\}, 
    \qquad 
    L^p_{(0)}(\Ga) 
    \coloneqq \big\{\psi\in L^p(\Ga): \meang{\psi} = 0\big\}.
\end{align*}

Now, let $L\in[0,\infty]$ and $\beta\in\R$. We introduce the linear subspace
\begin{align*}
	\mathcal{H}^1_{L,\beta} \coloneqq \begin{cases}
	\mathcal{H}^1, &\text{if~} L\in(0,\infty], \\
	\big\{(\phi,\psi)\in\mathcal{H}^1: \phi = \beta\psi \text{~a.e.~on~}\Ga\big\}, &\text{if~} L = 0.
	\end{cases}
\end{align*}
Subject to the inner product $\scp{\cdot}{\cdot}_{\mathcal{H}^1_{L,\beta}} \coloneqq \scp{\cdot}{\cdot}_{\mathcal{H}^1}$ and its induced norm, $\mathcal{H}^1_{L,\beta}$ is a Hilbert space. Furthermore, we define the product
\begin{align*}
	\bigang{\scp{\phi}{\psi}}{\scp{\zeta}{\xi}}_{\mathcal{H}^1_{L,\beta}} = \scp{\phi}{\zeta}_{L^2(\Om)} + \scp{\psi}{\xi}_{L^2(\Ga)}
\end{align*}
for all $(\phi,\psi), (\zeta,\xi)\in\mathcal{L}^2$. Using the Riesz representation theorem, this product can be extended to a duality pairing on $(\mathcal{H}^1_{L,\beta})^\prime\times\mathcal{H}^1_{L,\beta}$, which will also be denoted as $\ang{\cdot}{\cdot}_{\mathcal{H}^1_{L,\beta}}$.

Next, for $(\phi,\psi)\in(\mathcal{H}^1_{L,\beta})^\prime$, we define the generalized bulk-surface mean
\begin{align*}
	\mean{\phi}{\psi} \coloneqq \frac{\bigang{\scp{\phi}{\psi}}{\scp{\beta}{1}}_{\mathcal{H}^1_{L,\beta}}}{\beta^2\abs{\Om} + \abs{\Ga}},
\end{align*}
which reduces to
\begin{align*}
	\mean{\phi}{\psi} = \frac{\beta\abs{\Om}\meano{\phi} + \abs{\Ga}\meang{\psi}}{\beta^2\abs{\Om} + \abs{\Ga}}
\end{align*}
if $(\phi,\psi)\in\mathcal{L}^2$. We then define the closed linear subspace
\begin{align*}
	\mathcal{V}^1_{L,\beta} = \begin{cases}
	\{(\phi,\psi)\in\mathcal{H}^1_{L,\beta}:  \mean{\phi}{\psi} = 0\}, &\text{if~}L\in[0,\infty), \\
	\{(\phi,\psi)\in\mathcal{H}^1: \meano{\phi} = \meang{\psi} = 0\}, &\text{if~} L = \infty.
	\end{cases}
\end{align*}
Note that $\mathcal{V}^1_{L,\beta}$ is a Hilbert spaces with respect to the inner product $\scp{\cdot}{\cdot}_{\mathcal{H}^1}$. We further introduce the bilinear form
\begin{align*}
	\bigscp{\scp{\phi}{\psi}}{\scp{\zeta}{\xi}}_{L,\beta} \coloneqq &\intO \Grad\phi\cdot\Grad\zeta\dx + \intG \Gradg\psi\cdot\Gradg\xi\dG \\
	&\qquad + \chi(L) \intG (\beta\psi - \phi)(\beta\xi - \zeta)\dG
\end{align*}
for all $(\phi,\psi), (\zeta,\xi)\in\mathcal{H}^1$, where
\begin{align*}
	\chi:[0,\infty]\rightarrow [0,\infty), \qquad 
    \chi(r) \coloneqq 
    \begin{cases} 
    \frac{1}{r}, &\text{if~}r\in(0,\infty), \\
	0, &\text{if~} r\in\{0,\infty\}.
	\end{cases}
\end{align*}
Moreover, we define
\begin{align*}
	\norm{(\phi,\psi)}_{L,\beta} \coloneqq \bigscp{\scp{\phi}{\psi}}{\scp{\phi}{\psi}}_{L,\beta}^{1/2}
\end{align*}
for all $\scp{\phi}{\psi}\in\mathcal{H}^1$. The bilinear form defines an inner product on $\mathcal{V}^1_{L,\beta}$, and $\norm{\cdot}_{L,\beta}$ defines a norm on $\mathcal{V}^1_{L,\beta}$, that is equivalent to the norm $\norm{\cdot}_{\mathcal{H}^1}$, see, e.g., \cite[Corollary A.2]{Knopf2021}. In particular, $(\mathcal{V}^1_{L,\beta},\scp{\cdot}{\cdot}_{L,\beta},\norm{\cdot}_{L,\beta})$ is a Hilbert space. Next, we define the spaces
\begin{align*}
	\mathcal{V}^{-1}_{L,\beta} \coloneqq \begin{cases} \{(\phi,\psi)\in(\mathcal{H}^1_{L,\beta})^\prime: \mean{\phi}{\psi} = 0\}, &\text{if~} L\in[0,\infty), \\
	\{(\phi,\psi)\in(\mathcal{H}^1)^\prime: \meano{\phi} = \meang{\psi} = 0\}, &\text{if~} L = \infty.
	\end{cases}
\end{align*}
Now, using the Lax--Milgram theorem, we infer that for any $(\phi,\psi)\in\mathcal{V}^{-1}_{L,\beta}$, there exists a unique weak solution $\mathcal{S}_{L,\beta}(\phi,\psi) = (\mathcal{S}_{L,\beta}^\Om(\phi,\psi), \mathcal{S}_{L,\beta}^\Ga(\phi,\psi))\in\mathcal{V}_{L,\beta}^1$ to the elliptic boundary value problem
\begin{alignat*}{2}
	-\Lapg\mathcal{S}_{L,\beta}^\Om &= - \phi, &\qquad\text{in~}\Om, \\
	-\Lapg\mathcal{S}_{L,\beta}^\Ga + \beta\deln\mathcal{S}_{L,\beta}^\Om &= - \psi, &\qquad\text{on~}\Ga, \\
	L\deln\mathcal{S}_{L,\beta}^\Om &= \beta\mathcal{S}_{L,\beta}^\Ga - \mathcal{S}_{L,\beta}^\Om, &\qquad\text{on~}\Ga.
\end{alignat*}
To be precise, $\mathcal{S}_{L,\beta}(\phi,\psi)$ satisfies the weak formulation
\begin{align*}
	\bigscp{\mathcal{S}_{L,\beta}(\phi,\psi)}{\scp{\zeta}{\xi}}_{L,\beta} = - \bigang{\scp{\phi}{\psi}}{\scp{\zeta}{\xi}}_{\mathcal{H}^1_{L,\beta}}
\end{align*}
for all $(\zeta,\xi)\in\mathcal{H}^1_{L,\beta}$. We directly infer that
\begin{align*}
	\norm{\mathcal{S}_{L,\beta}(\phi,\psi)}_{L,\beta} \leq C\norm{(\phi,\psi)}_{(\mathcal{H}^1_{L,\beta})^\prime}
\end{align*}
for a constant $C>0$ that is independent of $(\phi,\psi)$.
In this way, we can define the solution operator
\begin{align*}
	\mathcal{S}_{L,\beta}:\mathcal{V}_{L,\beta}^{-1}
    \rightarrow\mathcal{V}_{L,\beta}^1, \qquad (\phi,\psi)\mapsto\mathcal{S}_{L,\beta}(\phi,\psi) = (\mathcal{S}_{L,\beta}^\Om(\phi,\psi), \mathcal{S}_{L,\beta}^\Ga(\phi,\psi)).
\end{align*}
Furthermore, an inner product and its induced norm on $\mathcal{V}_{L,\beta}^{-1}$ are given by
\begin{align*}
	\bigscp{\scp{\phi}{\psi}}{\scp{\zeta}{\xi}}_{L,\beta,\ast} &\coloneqq \bigscp{\mathcal{S}_{L,\beta}(\phi,\psi)}{\mathcal{S}_{L,\beta}(\zeta,\xi)}_{L,\beta}, \\
	\norm{(\phi,\psi)}_{L,\beta,\ast} &\coloneqq \bigscp{\scp{\phi}{\psi}}{\scp{\phi}{\psi}}_{L,\beta,\ast}^{1/2},
\end{align*}
for $(\phi,\psi), (\zeta,\xi)\in\mathcal{V}_{L,\beta}^{-1}$. This norm is equivalent to the standard norm $\norm{\cdot}_{(\mathcal{H}^1_{L,\beta})^\prime}$ on $\mathcal{V}_{L,\beta}^{-1}$. For the case $L\in(0,\infty)$, we refer to \cite[Theorem 3.3 and Corollary 3.5]{Knopf2021} for a proof of these statements. In the other cases, these results can be proven analogously. We further recall the following bulk-surface Poincar\'{e} inequality, which was established in \cite[Lemma A.1]{Knopf2021}.

\medskip

\begin{lemma}\label{Lemma:Poincare}
	Let $K\in[0,\infty)$ and $\alpha,\beta\in\R$ with $\alpha\beta\abs{\Om} + \abs{\Ga} \neq 0$. Then there exists a constant $C_P > 0$ depending only on $K,\alpha,\beta$ and $\Om$ such that
	\begin{align*}
		\norm{(\phi,\psi)}_{\mathcal{L}^2} \leq C_P \norm{(\phi,\psi)}_{K,\alpha}
	\end{align*}
	for all pairs $(\phi,\psi)\in\mathcal{H}^1_{K,\alpha}$ satisfying $\mean{\phi}{\psi} = 0$.
\end{lemma}

\subsection{Spaces of tangential and divergence-free vector fields}

For $p\in[2,\infty]$, we introduce the spaces 
\begin{align*}
	&\mathbf{L}^p_\Div(\Om) \coloneqq \{ \bv\in\mathbf{L}^p(\Om): \Div\,\bv = 0 \ \text{in~}\Om, \ \bv\cdot\n = 0 \ \text{on~}\Ga\} \,, \\
	&\mathbf{L}^p_\Div(\Ga) \coloneqq \{ \bw\in\mathbf{L}^p(\Ga): \Divg\,\bw = 0, \ \bw\cdot\n = 0 \ \text{on~}\Ga\},
\end{align*}
and we set
\begin{align*}
    \mathbfcal{L}^p_\Div \coloneqq \mathbf{L}^p_\Div(\Om)\times\mathbf{L}^p_\Div(\Ga).
\end{align*}
Then, for $s\geq 0$ and $p\in[2,\infty]$, we define $\mathbf{W}^{s,p}_\Div(\Om) = \mathbf{W}^{s,p}(\Om)\cap \mathbf{L}^2_\Div(\Om)$. Analogously, we define $\mathbf{W}^{s,p}_\Div(\Ga) = \mathbf{W}^{s,p}(\Ga)\cap \mathbf{L}^2_\Div(\Ga)$, and then set $\mathbfcal{W}^{s,p}_\Div = \mathbf{W}^{s,p}_\Div(\Om)\times\mathbf{W}^{s,p}_\Div(\Ga)$. Furthermore, for $s>\frac 12$, we set 
\begin{align*}
    \mathbfcal{W}^{s,p}_0 &\coloneqq \{(\bv,\bw)\in\mathbfcal{W}^{s,p}: \bv\cdot\n = 0, \ \bv\vert_\Ga = \bw \ \text{on~}\Ga\}\,,
    \\
    \mathbfcal{W}^{s,p}_{0,\Div} &\coloneqq \mathbfcal{W}^{s,p}_0 \cap \mathbfcal{W}^{s,p}_\Div \,.
\end{align*}
As before, we use the notation $\mathbfcal{H}^s_0 = \mathbfcal{W}^{s,2}_0$ as well as $\mathbfcal{H}^s_{0,\Div} = \mathbfcal{W}^{s,2}_{0,\Div}$. 

\section{Model derivation}
\label{Section:Model}
In this section, we will derive the bulk-surface Navier--Stokes--Cahn--Hilliard system \eqref{System}.
\subsection{Local mass balance laws}
We want to describe the time evolution of two fluids, indexed with $j=1,2$, in a (sufficiently smooth) domain $\Om\subset\R^d$ with $d\in\{2,3\}$ on a time interval $[0,T]$ with $T > 0$.
We write $Q = \Om\times(0,T)$ and $\Sigma = \Ga\times(0,T)$ with $\Ga = \partial\Om$.
Each bulk material is assumed to have a constant individual density $\tilde{\rho}_j$, $j=1,2$, whereas each surface material has a constant individual density $\tilde{\sigma}_j$, $j=1,2$. 
We further assume that the mass densities $\rho_j:Q\rightarrow\R$, $j=1,2$, of the two fluids in the bulk satisfy the following mass balance law:
\begin{align}\label{BL:MASS:B}
    \delt\rho_j + \Div\;\hJ_j = 0, \quad j=1,2,\quad\text{in~}Q.
\end{align}
Here, $\hJ_j$ are the mass fluxes that correspond to the motion of the two materials in the bulk. In our model, we also want to allow for a transfer of material between bulk and surface, which is, for instance, needed to describe absorption processes. To this end, we additionally consider functions $\sigma_j:\Sigma\rightarrow\R$, $j=1,2$, which represent the mass densities of the two fluids on the surface. In general, the materials on the surface might differ from those in the bulk, which is, for example, the case if the materials are transformed by chemical reactions occurring at the boundary. Therefore, we interpret the densities $\sigma_j$ as independent functions that are not necessarily the traces of the densities $\rho_j$, respectively.
They are supposed to satisfy the following mass balance law on the boundary:
\begin{align}\label{BL:MASS:S}
    \delt\sigma_j + \Divg\;\hK_j = \hJ_{\Ga,j}\cdot\n, \quad j=1,2,\quad\text{on~}\Sigma.
\end{align}
Here, $\hK_j$ stands for the mass flux on the boundary, and the right-hand side $\hJ_{\Ga,j}$ describes the transfer of mass between bulk and surface.
This exchange of material is assumed to be balanced by the relation
\begin{align}\label{CONS:VOL}
    \frac{\hJ_{\Ga,1}}{\tilde{\sigma}_1}\cdot\n + \frac{\hJ_{\Ga,2}}{\tilde{\sigma_2}}\cdot\n = 0 \quad\text{on~}\Sigma.
\end{align}
We next assume that the motion of our materials is described by individual velocity fields $\bv_j:Q\rightarrow\R^d$ and $\bw_j:\Sigma\rightarrow\R^d$, $j=1,2$ in the bulk and on the surface, respectively. These can be expressed as
\begin{align}
    \bv_j = \frac{\hJ_j}{\rho_j}\quad\text{in~}Q,
    \qquad
    \bw_j = \frac{\hK_j}{\sigma_j}\quad\text{on~}\Sigma,
    \qquad j=1,2.
\end{align}
Consequently, the mass balance laws \eqref{BL:MASS:B} and \eqref{BL:MASS:S} can be rewritten as
\begin{alignat}{2}
    \delt\rho_j + \Div(\rho_j\bv_j) &= 0, \quad\quad&&j=1,2,\quad\text{in~} Q, \label{BL:MASS:B'}\\
    \delt\sigma_j + \Divg(\sigma_j\bw_j) &= \hJ_{\Gamma,j} \cdot \n, \quad&&j=1,2,\quad\text{on~}\Sigma. \label{BL:MASS:S'}
\end{alignat}
We now write $\phi_j:Q\rightarrow\R$ and $\psi_j:\Sigma\rightarrow\R$, $j=1,2$, to denote the volume fractions of the two fluids in the bulk and on the surface, respectively. Provided that each component has a constant density $\tilde{\rho}_j$ and $\tilde{\sigma}_j$, $j=1,2$, respectively, the volume fractions can be identified as
\begin{align}\label{DEF:PHIPSI*}
    \phi_j = \frac{\rho_j}{\tilde{\rho}_j}\quad\text{in~}Q,
    \qquad
    \psi_j = \frac{\sigma_j}{\tilde{\sigma}_j}\quad\text{on~}\Sigma,
    \qquad j=1,2.
\end{align}
Assuming that the excess volume is zero, we have
\begin{align}\label{CONS:PHIPSI}
    \phi_1 + \phi_2 = 1 \quad\text{in~}Q,
    \qquad
    \psi_1 + \psi_2 = 1 \quad\text{on~}\Sigma.
\end{align}
We further define the order parameters
\begin{align}\label{DEF:PHIPSI}
    \phi = \phi_2 - \phi_1\quad\text{in~}Q,
    \qquad
    \psi = \psi_2 - \psi_1\quad\text{on~}\Sigma.
\end{align}
Furthermore, let $\bv:Q\rightarrow\R^d$ and $\bw:\Sigma\rightarrow\R^d$ denote the volume-averaged velocity fields associated with the two fluids in the bulk and on the surface, respectively. By means of the order parameters, they can be expressed as
\begin{alignat}{2}
    \bv &= \phi_1\bv_1 + \phi_2\bv_2\quad&&\text{in~}Q,\label{DEF:VAVG:B}\\
    \bw &= \psi_1\bw_1 + \psi_2\bw_2\quad&&\text{on~}\Sigma.\label{DEF:VAVG:S}
\end{alignat}
We assume that neither fluid can penetrate the boundary $\Ga$, and that the velocity field on the boundary is an extension of the one in the bulk. This leads to the conditions
\begin{align}\label{BC:V}
    \bv\vert_\Ga = \bw, \qquad \bv\cdot\n = 0 \qquad\text{on~}\Sigma.
\end{align}
As the second relation entails that $\bv\vert_\Gamma$ is a tangential vector field, the same holds for $\bw$ because of the first identity.
We now introduce
\begin{alignat*}{2}
    \J_j &= \hJ_j - \rho_j\bv,\quad&&j=1,2,\quad\text{in~}Q, \\
    \J_{\Ga,j} &= \hJ_{\Ga,j} - \sigma_j\bw,\qquad&&j=1,2,\quad\text{on~}\Sigma, \\
    \K_j &= \hK_j - \sigma_j\bw\quad&&j=1,2,\quad\text{on~}\Sigma,
\end{alignat*}
to denote the mass fluxes relative to the volume-averaged velocities. In view of \eqref{BC:V}, the mass balance equations \eqref{BL:MASS:B'} and \eqref{BL:MASS:S'} can thus be rewritten as
\begin{alignat}{2}
    \delt\rho_j + \Div(\rho_j\bv) + \Div\;\J_j &= 0,\quad&&j=1,2,\quad\text{in~}Q, \label{BL:MASS:B*}\\
    \delt\sigma_j + \Divg(\sigma_j\bw) + \Divg\;\K_j &= \J_{\Gamma,j}\cdot\n, \quad&&j=1,2,\quad\text{on~}\Sigma.\label{BL:MASS:S*}
\end{alignat}
We now define the total mass densities
\begin{alignat}{2}
    \rho &= \rho_1 + \rho_2
    &&\quad\text{in~}Q,\label{DEF:RHO}\\
    \sigma &= \sigma_1 + \sigma_2
    &&\quad\text{on~}\Sigma.\label{DEF:SIGMA}
\end{alignat}
By means of \eqref{BL:MASS:B*} and \eqref{BL:MASS:S*}, this leads to the relations
\begin{alignat}{2}
    \delt\rho + \Div(\rho\bv) + \Div(\J_1 + \J_2) &= 0\quad\quad&&\text{in~}Q, \label{BL:MASS:B**}\\
    \delt\sigma + \Divg(\sigma\bw) + \Divg(\K_1 + \K_2) &= (\J_{\Gamma,1} + \J_{\Gamma,2})\cdot\n \quad\quad&&\text{on~}\Sigma. \label{BL:MASS:S**}
\end{alignat}

\subsection{Local balance laws for the linear momentum}
The next step is to consider the local balance laws for the linear momentum:
\begin{alignat}{2}
    \delt(\rho\bv) + \Div(\rho\bv\otimes\bv) &= \Div\;\tilde{\T}\quad&&\text{in~}Q, \label{BL:CLM:B}\\
    \delt(\sigma\bw) + \Divgt(\sigma\bw\otimes\bw) &= \Divgt\tilde{\T}_\Ga + \Z\quad&&\text{on~}\Sigma.\label{BL:CLM:S}
\end{alignat}
Here, $\tilde\T$ and $\tilde\T_\Ga$ are stress tensors that need to be specified later by means of constitutive assumptions, and $\Z$ is the surface force density that accounts for mechanical interactions between bulk and surface materials. We point out that $\Z$ is a tangential vector field.
In the following, we write
\begin{align}\label{DEF:JKJG}
    \J = \J_1 + \J_2\quad\text{in~}Q,
    \qquad\quad
    \K = \K_1 + \K_2,
    \quad 
    \J_\Ga = \J_{\Ga,1} + \J_{\Ga,2}\quad\text{on~}\Sigma.
\end{align}
By means of \eqref{BL:MASS:B**}, the balance law \eqref{BL:CLM:B} can be reformulated as
\begin{align}\label{BL:CLM:B*}
    \begin{split}
        \rho(\delt\bv + \bv\cdot\Grad\bv) &= \Div\;\tilde{\T} + (\Div\;\J)\bv \\
        &= \Div(\tilde{\T} + \bv\otimes\J) - \Grad\bv \; \J \quad\text{in~}Q.
    \end{split}
\end{align}
Analogously, using \eqref{BL:MASS:S**} to reformulate \eqref{BL:CLM:S}, we get
\begin{align}\label{BL:CLM:S*}
    \begin{split}
        \sigma(\delt\bw + \bw\cdot\Gradg\bw) = \Divgt(\tilde{\T}_\Ga + \bw\otimes\K) - \big[\Gradg\bw\; \K\big]_\tau - (\J_\Ga\cdot\n)\bw + \Z\quad\text{on~}\Sigma.
    \end{split}
\end{align}
Thus, defining the frame indifferent objective tensors
\begin{align}\label{DEF:TTG}
    \T = \tilde{\T} + \bv\otimes\J \quad\text{in~}Q,
    \qquad
    \T_\Ga = \tilde{\T}_\Ga + \bw\otimes\K\quad\text{on~}\Sigma,
\end{align}
we can rewrite \eqref{BL:CLM:B*} and \eqref{BL:CLM:S*} as
\begin{alignat}{2}
    \rho\delt\bv + \Grad\bv\, (\rho\bv + \J) &= \Div\;\T \quad&&\text{in~}Q, \label{BL:CLM:B**} \\
    \sigma\delt\bw + \big[\Gradg\bw\,(\sigma\bw + \K)\big]_\tau 
    &= \Divgt\T_\Ga - (\J_\Ga \cdot\n)\bw + \Z\quad&&\text{on~}\Sigma\label{BL:CLM:S**}.
\end{alignat}
Next, we multiply the equations \eqref{BL:MASS:B'} each by $1/\tilde{\rho}_j$, $j=1,2$, and add the resulting equations. Recalling \eqref{DEF:PHIPSI*}, \eqref{DEF:PHIPSI} and \eqref{DEF:VAVG:B}, we obtain
\begin{align}\label{PDE:DIV:B}
    \Div\;\bv = \Div\left(\frac{\rho_1}{\tilde{\rho}_1}\bv_1 + \frac{\rho_2}{\tilde{\rho}_2}\bv_2\right) = - \delt\left(\frac{\rho_1}{\tilde{\rho}_1} + \frac{\rho_2}{\tilde{\rho}_2}\right) = - \delt 1 = 0 \quad\text{in~}Q,
\end{align}
which means that the volume-averaged velocity field is divergence-free. Similarly, multiplying \eqref{BL:MASS:S'} by $1/\tilde{\sigma}_j$, $j=1,2$, adding the resulting equations, and recalling \eqref{CONS:VOL}, \eqref{DEF:PHIPSI*}, \eqref{DEF:PHIPSI} and \eqref{DEF:VAVG:S}, we deduce
\begin{align}\label{PDE:DIV:S}
    \begin{split}
        \Divg\bw &= \Divg\left(\frac{\sigma_1}{\tilde{\sigma}_1}\bw_1 + \frac{\sigma_2}{\tilde{\sigma}_2}\bw_2\right) \\
        &= - \Divg\left(\frac{\sigma_1}{\tilde{\sigma}_1} + \frac{\sigma_2}{\tilde{\sigma}_2}\right) + \left(\frac{\hJ_{\Ga,1}}{\tilde{\sigma}_1} + \frac{\hJ_{\Ga,2}}{\tilde{\sigma}_2}\right)\cdot\n = 0 \quad\text{on~}\Sigma.
    \end{split}
\end{align}
We now define the quantities
\begin{alignat}{4}
    \tilde{\J}_j &= \frac{\J_j}{\tilde{\rho}_j}, \quad&&j=1,2, \qquad \J_\phi &&= \tilde{\J}_2 - \tilde{\J}_1 \qquad&&\text{in~}Q, \label{DEF:TJ}\\
    \tilde{\J}_{\Ga,j} &= \frac{\J_{\Ga,j}}{\tilde{\sigma}_j}, \quad&&j=1,2, \qquad \J_\psi &&= \tilde{\J}_{\Ga,2} - \tilde{\J}_{\Ga,1} \qquad&&\text{on~}\Sigma, \label{DEF:TJG}\\
    \tilde{\K}_j &= \frac{\K_j}{\tilde{\sigma}_j}, \quad&&j=1,2, \qquad \K_\psi &&= \tilde{\K}_2 - \tilde{\K}_1 \qquad&&\text{on~}\Sigma.\label{DEF:TK}
\end{alignat}
Then, multiplying the equations \eqref{BL:MASS:B*} each by $1/\tilde{\rho}_j$, $j=1,2$, and subtracting the resulting equations, we deduce that
\begin{align}\label{EQ:PHI}
    \delt\phi + \Div(\phi\bv) + \Div\;\J_\phi = 0 \quad\text{in~}Q.
\end{align}
Furthermore, multiplying the equations \eqref{BL:MASS:B*} each by $1/\tilde{\rho}_j$, $j=1,2$, adding the resulting equations, and recalling \eqref{CONS:PHIPSI} and \eqref{PDE:DIV:B}, we conclude that
\begin{align}\label{ID:DIVJ}
    \Div(\tilde{\J}_1 + \tilde{\J}_2) = -\delt(\phi_1 + \phi_2) - \Div((\phi_1 + \phi_2)\bv) = 0 \quad\text{in~}Q.
\end{align}
Arguing similarly for \eqref{BL:MASS:S*}, and recalling \eqref{CONS:PHIPSI} and \eqref{PDE:DIV:S}, we obtain
\begin{align}\label{EQ:PSI}
    \delt\psi + \Divg(\psi\bw) + \Divg\;\K_\psi &= \J_\psi\cdot\n\quad\text{on~}\Sigma,
\end{align}
and
\begin{align}\label{ID:DIVK}
    \Divg(\tilde{\K}_1 + \tilde{\K}_2) = - \delt(\psi_1 + \psi_2) - \Divg((\psi_1 + \psi_2)\bw) + (\hJ_{\Ga,1} + \hJ_{\Ga,2})\cdot\n = 0
\end{align}
on $\Sigma$. Using \eqref{DEF:TJ} and \eqref{ID:DIVJ}, we infer that
\begin{align}\label{ID:DIVJ*}
    \begin{split}
        \Div\;\J &= \frac12\left[\tilde{\rho}_2\Div\;\tilde{\J}_2 + \tilde{\rho}_1\Div\;\tilde{\J}_1\right] + \frac12\left[\tilde{\rho}_2\Div\;\tilde{\J}_2 + \tilde{\rho}_1\Div\;\tilde{\J}_1\right] \\
        &= \frac{\tilde{\rho}_2 - \tilde{\rho}_1}{2}\Div\;\tilde{\J}_2 - \frac{\tilde{\rho}_2 - \tilde{\rho}_1}{2}\Div\;\tilde{\J}_2 = \frac{\tilde{\rho}_2 - \tilde{\rho}_1}{2}\Div\;\J_\phi \qquad\text{in~}Q.
    \end{split}
\end{align}
Therefore, the fluxes $\J$ and $\tfrac12(\tilde{\rho}_2 - \tilde{\rho}_1)\J_\phi$ can merely differ by an additive divergence-free function. As in \cite{Abels2012, Giorgini2023}, we assume that
\begin{align}\label{ID:J}
    \J = \frac{\tilde{\rho}_2 - \tilde{\rho}_1}{2}\J_\phi\quad\text{in~}Q.
\end{align}
Proceeding analogously, we derive the identity
\begin{align*}
    \Divg\;\K = \frac{\tilde{\sigma}_2 - \tilde{\sigma}_1}{2}\Divg\;\K_\psi\quad\text{on~}\Sigma,
\end{align*}
and assume that
\begin{align}\label{ID:K}
    \K = \frac{\tilde{\sigma}_2 - \tilde{\sigma}_1}{2}\K_\psi\quad\text{on~}\Sigma.
\end{align}
Moreover, assuming that the phase-fields $\phi$ and $\psi$ only attain values in the interval $[-1,1]$, we use \eqref{DEF:PHIPSI*}, \eqref{CONS:PHIPSI}, \eqref{DEF:PHIPSI}, \eqref{DEF:RHO} and \eqref{DEF:SIGMA} to conclude that the densities $\rho = \rho(\phi)$ and $\sigma = \sigma(\psi)$ are given by
\begin{alignat}{2}
    \rho(\phi) &= \frac{\tilde{\rho}_2 - \tilde{\rho}_1}{2}\phi + \frac{\tilde{\rho}_2 + \tilde{\rho}_1}{2}\quad&&\text{in~}Q,\label{DEF:RHO*} \\
    \sigma(\psi) &= \frac{\tilde{\sigma}_2 - \tilde{\sigma}_1}{2}\psi + \frac{\tilde{\sigma}_2 + \tilde{\sigma}_1}{2}\quad&&\text{on~}\Sigma.\label{DEF:SIGMA*}
\end{alignat}
Finally, differentiating formula \eqref{DEF:SIGMA} with respect to time and employing \eqref{EQ:PSI} and \eqref{ID:K}, we deduce that $\sigma = \sigma(\psi)$ fulfills the identity
\begin{align*}
    \delt\sigma + \Divg(\sigma\bw) + \Divg\;\K 
    &= \frac{\tilde{\sigma}_2 - \tilde{\sigma}_1}{2} \J_\psi \cdot\n \quad \text{on~}\Sigma.
\end{align*}
Comparing this equation with \eqref{BL:MASS:S**} (see also \eqref{DEF:JKJG}), we finally conclude that
\begin{align}
    \label{ID:JPSIJGA}
    \J_\Gamma \cdot \n
    = \frac{\tilde{\sigma}_2 - \tilde{\sigma}_1}{2} \J_\psi \cdot\n \quad \text{on~}\Sigma.
\end{align}

\subsection{Local energy dissipation laws}
Proceeding as in \cite{Abels2012} and \cite{Giorgini2023}, we consider the following energy density in the bulk:
\begin{align}\label{DEF:END:B}
    e_\Om(\bv,\phi,\Grad\phi) = \frac{\rho(\phi)}{2}\abs{\bv}^2 + f(\phi,\Grad\phi)\quad\text{in~}Q.
\end{align}
Here, the first summand on the right-hand side represents the kinetic energy density, while the second summand denotes the free energy density in the bulk. It is assumed to be of Ginzburg--Landau type, that is, 
\begin{align}\label{DEF:GL:B}
    f(\phi,\Grad\phi) = \frac\ep2\abs{\Grad\phi}^2 + \frac1\ep F(\phi) \qquad\text{in~}Q.
\end{align}
Moreover, we additionally introduce the following surface energy density
\begin{align}\label{DEF:END:S}
    e_\Ga(\bw,\phi,\psi,\Gradg\psi) = \frac{\sigma(\psi)}{2}\abs{\bw}^2 + g(\phi,\psi,\Gradg\psi)\quad\text{on~}\Sigma.
\end{align}
As in the bulk, the first summand on the right-hand side represents the kinetic energy density on the surface, whereas the second summand stands for the free energy density on the surface.  It is also assumed to be of Ginzburg--Landau type (cf.~\cite{Knopf2025a,Giorgini2023}) with an additional term accounting for the relation between $\phi$ and $\psi$, namely,
\begin{align}\label{DEF:GL:S}
    g(\phi,\psi,\Gradg\psi) = \frac{\delta\kappa}{2}\abs{\Gradg\psi}^2 + \frac1\delta G(\psi) + \frac12\chi(K)(\alpha\psi - \phi)^2 \qquad\text{on~}\Sigma,
\end{align}
for $K\in[0,\infty]$. Here, $\kappa\ge 0$ acts as a weight for the surface Dirichlet energy density and $\chi$ is the function defined in \eqref{DEF:H}.

We now consider an arbitrary test volume $V(t) \subset\Om$, $t\in[0,T]$, which is transported by the flow associated with $\bv$. As we assume our physical system to be isothermal, the second law of thermodynamics leads to the dissipation inequality
\begin{align}\label{IEQ:DISS}
    \begin{split}
        0 &\geq \left[\int_{V(t)}e_\Om(\bv,\psi,\Grad\phi)\dx + \int_{\del V(t)\cap\Ga}e_\Ga(\bw,\phi,\psi,\Gradg\psi)\dG\right] \\
        &\quad + \int_{\del V(t)\cap\Om}\J_e\cdot\n_{\del V(t)}\dG + \int_{\del(\del V(t)\cap\Ga)}\K_e\cdot\n_{\del(\del V(t)\cap\Ga)}\dG_\Ga.
    \end{split}
\end{align}
In this inequality, $\J_e$ and $\K_e$ are energy fluxes that will be specified later, $\n_{\del V(t)}$ denotes the outer unit normal vector field on $\del V(t)$, whereas $\n_{\del(\del V(t)\cap\Om)}$ is the unit conormal vector field on $\del(\del V(t)\cap\Om)$. Since we consider a thermodynamically closed system, there is no exchange of energy over the boundary $\Ga$, and thus, the domain of the first integral in the second line is just $\del V(t)\cap\Om$ instead of $\del V(t)$. We further point out that $\del(\del V(t)\cap\Om)\subset\Ga$ is to be understood as the relative boundary of the set $\del V(t)\cap\Ga$ within the submanifold $\Ga$. Applying Gauß's divergence theorem on both integrals in the second line, and recalling that $\n_{\del V(t)} = \n$ on $\del V(t)\cap\Ga$, we reformulate \eqref{IEQ:DISS} as
\begin{align}\label{IEQ:DISS*}
    \begin{split}
        0 &\geq \left[\int_{V(t)}e_\Om(\bv,\psi,\Grad\phi)\dx + \int_{\del V(t)\cap\Ga}e_\Ga(\bw,\phi,\psi,\Gradg\psi)\dG\right] \\
        &\quad + \int_{V(t)}\Div\;\J_e\dx + \int_{\del V(t)\cap\Ga
        }\Divg\;\K_e - \J_e\cdot\n\dG.
    \end{split}
\end{align}
Using the Reynolds transport theorem (see, e.g., \cite[Theorem~2.11.1]{Baensch2023}), we obtain the relation
\begin{align}\label{TRANSP:B}
    \ddt\int_{V(t)}e_\Om(\bv,\phi,\Grad\phi)\dx = \int_{V(t)}\delt e_\Om(\bv,\phi,\Grad\phi) + \Div(e_\Om(\bv,\phi,\Grad\phi)\bv)\dx.
\end{align}
Similarly, applying the Reynolds transport theorem for evolving hypersurfaces (see, e.g., \cite[Theorem~2.10.1]{Baensch2023}), we obtain
\begin{align}\label{TRANSP:S}
    \begin{split}
        \ddt\int_{\del V(t)\cap\Ga}e_\Ga(\bw,\phi,\psi,\Gradg\psi)\dG &= \int_{V(t)\cap\Ga} \delt e_\Ga(\bw,\phi,\psi,\Gradg\psi) \\
        &\qquad\qquad + \Divg(e_\Ga(\bw,\phi,\psi,\Gradg\psi)\bv)\dG.
    \end{split}
\end{align}
Combining \eqref{IEQ:DISS*}, \eqref{TRANSP:B} and \eqref{TRANSP:S}, we thus get
\begin{align}
    0 &\geq \int_{V(t)}\delt e_\Om(\bv,\phi,\Grad\phi) + \Div(e_\Om(\bv,\phi,\Grad\phi)\bv) + \Div\;\J_e\dx \label{IEQ:DISS**}\\
    &\qquad + \int_{\del V(t)\cap\Ga}\delt e_\Ga(\bw,\phi,\psi,\Gradg\psi) + \Divg(e_\Ga(\bw,\phi,\psi,\Gradg\psi)\bv) + \Divg\;\K_e - \J_e\cdot\n\dG \nonumber.
\end{align}
In particular, this inequality holds for all test volumes $V(t)\subset\Om$ with $\del V(t)\cap\Ga = \emptyset$. Therefore, we infer the local dissipation law in the bulk, which reads as
\begin{align}\label{IEQ:DISS:LOC:B}
    0 \geq -\mathcal{D}_\Om\coloneqq \delt e_\Om(\bv,\phi,\Grad\phi) + \Div(e_\Om(\bv,\phi,\Grad\phi)\bv) + \Div\;\J_e \qquad\text{in~}Q.
\end{align}
Let now $c > 0$ be arbitrary, and consider a test volume $V(t)\subset\Om$ with $\abs{V(t)}$ being sufficiently small such that
\begin{align*}
    \int_{V(t)} \mathcal{D}_\Om\dx < c.
\end{align*}
Then, invoking \eqref{IEQ:DISS**}, we infer that
\begin{align*}
    \int_{\del V(t)\cap\Ga}\delt e_\Ga(\bw,\phi,\psi,\Gradg\psi) + \Divg(e_\Ga(\bw,\phi,\psi,\Gradg\psi)\bv) + \Divg\;\K_e - \J_e\cdot\n\dG < c.
\end{align*}
As $c > 0$ and the test volume $V(t)$ were arbitrary (except for the above smallness condition on $V(t)$), we conclude the following local dissipation law on the boundary:
\begin{align}\label{IEQ:DISS:LOC:S}
    0 \geq -\mathcal{D}_\Ga \coloneqq \delt e_\Ga(\bw,\phi,\psi,\Gradg\psi) + \Divg(e_\Ga(\bw,\phi,\psi,\Gradg\psi)\bv) + \Divg\;\K_e - \J_e\cdot\n \qquad\text{on~}\Sigma.
\end{align}

\subsection{Completion of the model derivation via the Lagrange multiplier approach}
We now want to identify the still undetermined flux terms using the \emph{Lagrange multiplier approach}. Therefore, we introduce Lagrange multipliers $\mu$ and $\theta$, which need to be determined in the course of this subsection. In the final model, $\mu$ will represent the \emph{bulk chemical potential}, while $\theta$ will represent the \emph{surface chemical potential}. Invoking the identities \eqref{EQ:PHI} and \eqref{EQ:PSI}, the local energy dissipation laws \eqref{IEQ:DISS:LOC:B} and \eqref{IEQ:DISS:LOC:S} can be expressed as
\begin{alignat}{2}
    0 \geq -\mathcal{D}_\Om &= \delt e_\Om(\bv,\phi,\Grad\phi) + \Div(e_\Om(\bv,\phi,\Grad\phi)\bv) + \Div\;\J_e \nonumber \\
    \qquad& - \mu(\delt\phi + \Div(\phi\bv) + \Div\;\J_\phi) \qquad&&\text{in~}Q, \label{IEQ:DISS:B:O}\\[0.5em]
    0 \geq -\mathcal{D}_\Ga &= \delt e_\Ga(\bw,\phi,\psi,\Gradg\psi) + \Divg(e_\Ga(\bw,\phi,\psi,\Gradg\psi)\bv) + \Divg\;\K_e  \nonumber \\
    \qquad& - \J_e\cdot\n - \theta(\delt\psi + \Divg(\psi\bw) + \Divg\;\K_\psi - \J_\psi\cdot\n) \qquad&&\text{on~}\Sigma. \label{IEQ:DISS:S:O}
\end{alignat}
In the following, to provide a cleaner presentation, we will simply write $\rho$, $\sigma$, $f$ and $g$ instead of $\rho(\phi)$, $\sigma(\psi)$, $f(\phi,\Grad\phi)$ and $g(\phi,\psi,\Gradg\psi)$, respectively. Invoking the definition of the energy densities $e_\Om$ and $e_\Ga$ (see \eqref{DEF:END:B} and \eqref{DEF:END:S}), we reformulate \eqref{IEQ:DISS:B:O} and \eqref{IEQ:DISS:S:O} as
\begin{alignat}{2}
    0 \geq -\mathcal{D}_\Om &=  \delt(\tfrac12\rho\abs{\bv}^2) + \Div(\tfrac12\rho\abs{\bv}^2\bv) + \delt f + \Div(f\bv) + \Div\;\J_e
    \label{IEQ:DISS:B:1} \nonumber\\
    &\quad  - \mu(\delt\phi + \Div(\phi\bv) + \Div\;\J_\phi) 
    &&\quad\text{in~}Q, 
    \\[0.5em]
    0 \geq -\mathcal{D}_\Ga &= \delt(\tfrac12\sigma\abs{\bw}^2) + \Divg(\tfrac12\sigma\abs{\bw}^2\bw) + \delt g + \Divg(g\bw) + \Divg\;\K_e
    \label{IEQ:DISS:S:1} \nonumber\\
    &\quad   - \J_e\cdot\n - \theta(\delt\psi + \Divg(\psi\bw) + \Divg\;\K_\psi - \J_\psi\cdot\n) 
    &&\quad\text{on~}\Sigma. 
\end{alignat}
In the following, for functions $h:Q\rightarrow\R$ and $k:\Sigma\rightarrow\R$, their material derivatives are given by
\begin{alignat*}{2}
    \deltb h &= \delt h + \bv\cdot\Grad h \qquad&&\text{in~}Q, \\
    \deltc k &= \delt k + \bw\cdot\Gradg k \qquad&&\text{on~}\Sigma.
\end{alignat*}
We point out that in the material derivative $\deltc$ on the boundary, only the tangential gradient, instead of the full gradient, is relevant. This is because we already know from \eqref{BC:V} that $\bw$ is a tangential vector field.
With these definitions at hand, we can express the derivatives $\deltb f$ and $\deltc g$ through the chain rule as
\begin{alignat}{2}
    \deltb f &= \delphi f \deltb\phi + \delgphi f\deltb\Grad\phi \qquad&&\text{in~}Q, \label{ID:DELTCF} \\
    \deltc g &= \delphi g \deltc\phi + \delpsi g\deltc\psi + \delgpsi g \deltc\Gradg\psi \qquad&&\text{on~}\Sigma. \label{ID:DELTPG}
\end{alignat}
\subsubsection{Computations in the bulk}
In the bulk we proceed exactly as in \cite{Abels2012} and \cite{Giorgini2023} to reformulate the inequality \eqref{IEQ:DISS:B:1} as
\begin{align}\label{IEQ:DISS:B:2}
    \begin{split}
        0 &\geq \Div[\J_e - \tfrac12\abs{\bv}^2\J + \T^\top\bv - \mu\J_\phi + \delgphi f\deltb\phi] \\
        &\quad + [\delphi f - \Div(\delgphi f) - \mu]\deltb\phi \\
        &\quad - [\T + \Grad\phi\otimes\delgphi f]:\Grad\bv + \Grad\mu\cdot\J_\phi \qquad\qquad\text{in~}Q.
    \end{split}
\end{align}
To ensure that \eqref{IEQ:DISS:B:2} holds, we now choose the chemical potential $\mu$ and the energy flux $\J_e$ as
\begin{alignat}{2}
    \mu &= \delphi f - \Div(\delgphi f) &&\qquad\text{in~}Q, \label{DEF:MU}\\
    \J_e &= \tfrac12\abs{\bv}^2\J - \T^\top\bv + \mu\J_\phi - \delgphi f\deltb\phi &&\qquad\text{in~}Q. \label{DEF:JE}
\end{alignat}
In view of these choices, the first two lines of the right-hand side of \eqref{IEQ:DISS:B:2} vanish. Furthermore, we assume the mass flux $\J_\phi$ to be of Fick's type, which means that
\begin{align}\label{DEF:JPHI}
    \J_\phi = -m_\Om(\phi)\Grad\mu\qquad\text{in~}Q.
\end{align}
Here, $m_\Om = m_\Om(\phi)$ is a non-negative function, which represents the mobility in the bulk. Therefore, \eqref{IEQ:DISS:B:2} reduces to
\begin{align}\label{IEQ:DISS:B:3}
    0 \geq -[\T + \Grad\phi\otimes\delgphi f]:\Grad\bv - m_\Om(\phi)\abs{\Grad\mu}^2 \qquad\text{in~}Q.
\end{align}
We next define the tensor
\begin{align*}
    \S = \T + p\I + \Grad\phi\otimes\delgphi f \qquad\text{in~}Q.
\end{align*}
Here, the scalar variable $p$ denotes the pressure in the bulk, and $\I$ stands for the identity. The tensor $\S$ is the \emph{viscous stress tensor} that corresponds to the irreversible changes of the energy caused by internal friction. For Newtonian fluids, $\S$ is usually assumed to be given by
\begin{align*}
    \S = 2\nu_\Om(\phi)\D\bv,
\end{align*}
where the non-negative function $\nu_\Om = \nu_\Om(\phi)$ represents the viscosity of the fluids in the bulk, and $\D\bv$ is the symmetric gradient of $\bv$. This choice, together with \eqref{IEQ:DISS:B:3} and the identity $p\I:\Grad\bv = p\,\Div\;\bv = 0$ in $Q$, implies that \eqref{IEQ:DISS:B:2} is fulfilled as it reduces to
\begin{align}\label{IEQ:DISS:B:4}
    0 \geq -\nu_\Om(\phi)\abs{\D\bv}^2 - m_\Om(\phi)\abs{\Grad\mu}^2 \qquad\text{in~}Q.
\end{align}
In view of \eqref{DEF:GL:B}, the chemical potential $\mu$ given by \eqref{DEF:MU} can be expressed as
\begin{align}\label{PDE:MU}
    \mu = - \ep\Lap\phi + \frac1\ep F^\prime(\phi)\qquad\text{in~}Q.
\end{align}
Moreover, the total stress tensor $\T$ is given by
\begin{align}\label{DEF:T}
    \T = \S - p\I - \Grad\phi\otimes\delgphi f = 2\nu_\Om(\phi)\D\bv - p\I - \ep\Grad\phi\otimes\Grad\phi\qquad\text{in~}Q.
\end{align}
Furthermore, plugging \eqref{DEF:JPHI} into \eqref{EQ:PHI}, we obtain the equation
\begin{align}\label{PDE:PHI}
    \delt\phi + \Div(\phi\bv) = \Div(m_\Om(\phi)\Grad\mu) \qquad\text{in~}Q.
\end{align}
Moreover, \eqref{BL:CLM:B} can be reformulated as
\begin{align}\label{PDE:NS:B}
    \delt(\rho(\phi)\bv) + \Div(\bv\otimes(\rho(\phi)\bv + \J)) = \Div\,\T
\end{align}
in $Q$, where $\T$ is given by \eqref{DEF:T} and, due to \eqref{ID:J} and \eqref{DEF:JPHI}, the flux term $\J$ is given by
\begin{align}\label{PDE:J}
    \J = - \frac{\tilde{\rho}_2 - \tilde{\rho}_1}{2}m_\Om(\phi)\Grad\mu \qquad\text{in~}Q.
\end{align}
\subsubsection{Computations on the surface}
Next, we consider the local energy dissipation law \eqref{IEQ:DISS:S:1} on the surface. Recalling the formulas for $f$ (see \eqref{DEF:GL:B}) and $\J_e$ (see \eqref{DEF:JE}), we deduce from \eqref{IEQ:DISS:S:1} that
\begin{align}\label{IEQ:DISS:S:2}
    \begin{split}
        0 &\geq \delt(\tfrac12\sigma\abs{\bw}^2) + \Divg(\tfrac12\sigma\abs{\bw}^2\bw) + \delphi g \deltc\phi + \delpsi g\deltc\psi + \delgpsi g \deltc\Gradg\psi \\
        &\quad + \Divg\K_e - \tfrac12(\J\cdot\n)\bv\cdot\bv - \T\n\cdot\bv - \mu\J_\phi\cdot\n + \ep(\Grad\phi\cdot\n)\deltc\phi \\
        &\quad - \theta\deltc\psi - \Divg(\K_\psi\theta) + \Gradg\theta\cdot\K_\psi + \theta\J_\psi\cdot\n
    \end{split}
\end{align}
on $\Sigma$. Here, we have used that
\begin{align*}
    \bv\cdot\Grad\phi = \bv\cdot[\Gradg\phi + \n(\Grad\phi\cdot\n)] = \bw\cdot\Gradg\phi\qquad\text{on~}\Sigma,
\end{align*}
which follows from the definition of the surface gradient and the boundary condition \eqref{BC:V}.
Next, a straightforward computation yields
\begin{align}
    \Divg(\delgpsi g \deltc\psi) 
    = \Divg(\delgpsi g)\deltc\psi 
    + \delgpsi g \cdot \deltc\Gradg\psi + [\Gradg\psi\otimes\delgpsi g]:\Gradg\bw
\end{align}
on $\Sigma$. 
We now use this identity along with \eqref{BL:CLM:S**}, \eqref{EQ:PSI}, \eqref{ID:K}, \eqref{DEF:SIGMA} and \eqref{ID:JPSIJGA} to reformulate \eqref{IEQ:DISS:S:2}. After a technical but straightforward computation, we arrive at
\begin{align}\label{IEQ:DISS:S:3}
    \begin{split}
        0 &\geq \Divg[\K_e - \tfrac12\abs{\bw}^2\K + \T_\Ga^\top \bw  - \theta\K_\psi + \delgpsi g\deltc\psi + \tfrac12\chi(K)(\alpha\psi-\phi)^2\bw] \\
        &\quad + [\delphi g + \ep\Grad\phi\cdot\n]\deltc\phi + [\delpsi g - \Divg(\delgpsi g) - \theta]\deltc\psi \\
        &\quad + [[\T\n]_\tau - \tfrac12(\J\cdot\n)\bv - \tfrac12(\J_\Ga\cdot\n)\bw - \tfrac12\chi(K)\Gradg(\alpha\psi-\phi)^2 + \Z]\cdot\bw \\
        &\quad - [\T_\Ga + \Gradg\psi\otimes\delgpsi g]:\Gradg\bw + \Gradg\theta\cdot\K_\psi + \theta\J_\psi\cdot\n - \mu\J_\phi\cdot\n
    \end{split}
\end{align}
on $\Sigma$. 
In order to ensure that \eqref{IEQ:DISS:S:3} is fulfilled, we choose the mass flux $\K_\psi$ and the energy flux $\K_e$ as follows:
\begin{alignat}{2}
    \K_\psi &= - m_\Ga(\psi)\Gradg\theta &&\qquad\text{on~}\Sigma, \label{DEF:KPSI}\\
    \K_e &= \frac12\abs{\bw}^2\K - \T_\Ga^\top\bw + \theta\K_\psi - \delgpsi g \deltc\psi 
    - \tfrac12\chi(K)(\alpha\psi-\phi)^2\bw
    &&\qquad\text{on~}\Sigma. \label{DEF:KE}
\end{alignat}
In \eqref{DEF:KPSI}, $m_\Ga = m_\Ga(\psi)$ is a non-negative function representing the mobility on the surface. This choice of $\K_e$ ensures that the first line of the right-hand side of \eqref{IEQ:DISS:S:3} vanishes. Due to \eqref{DEF:KPSI} and \eqref{DEF:KE}, inequality \eqref{IEQ:DISS:S:3} reduces to
\begin{align}\label{IEQ:DISS:S:4}
    \begin{split}
        0 &\geq [\delphi g + \ep\Grad\phi\cdot\n]\deltc\phi + [\delpsi g - \Divg(\delgpsi g) - \theta]\deltc\psi 
        \\
        &\quad + [[\T\n]_\tau - \tfrac12(\J\cdot\n)\bv - \tfrac12(\J_\Ga\cdot\n)\bw 
        - \tfrac12\chi(K)\Gradg(\alpha\psi-\phi)^2 + \Z]\cdot\bw 
        \\
        &\quad - [\T_\Ga + \Gradg\psi\otimes\delgpsi g]:\Gradg\bw - m_\Ga(\psi)\abs{\Gradg\theta}^2 + \theta\J_\psi\cdot\n - \mu\J_\phi\cdot\n
    \end{split}
\end{align}
on $\Sigma$. We further assume that the flux terms $\J_\phi$ and $\J_\psi$ are directly proportional. To be precise, we suppose that there exists a real number $\beta$ such that
\begin{align}\label{DEF:JPSI}
    \J_\psi\cdot\n = \beta\J_\phi\cdot\n = - \beta m_\Om(\phi)\Grad\mu\cdot\n \qquad\text{on~}\Sigma.
\end{align}
The parameter $\beta$ is related to the transfer of material between bulk and surface and indicates the proportion of bulk material that becomes surface material (or vice versa).
Due to \eqref{ID:JPSIJGA}, this further implies that
\begin{align}
    \label{PDE:JG}
    \J_\Gamma \cdot \n
    =
    -\beta\frac{\tilde{\sigma}_2 - \tilde{\sigma}_1}{2}m_\Om(\phi)\Grad\mu\cdot\n 
    \qquad\text{on~}\Sigma.
\end{align}
The second equality in \eqref{DEF:JPSI} follows from \eqref{DEF:JPHI}. We further make the constitutive assumption
\begin{align}\label{DEF:Z}
    \begin{split}
    \Z 
    &= - [\S\n]_\tau 
    + \tfrac12(\J\cdot\n)\bw + \tfrac12(\J_\Ga\cdot\n)\bw 
    - \gamma(\phi,\psi)\bw 
    + \alpha\ep (\Grad\phi\cdot\n) \Gradg\psi
    \\
    &= - [\T\n]_\tau 
    + \tfrac12(\J\cdot\n)\bw + \tfrac12(\J_\Ga\cdot\n)\bw 
    - \gamma(\phi,\psi)\bw  
    + \ep (\Grad\phi\cdot\n) \Gradg(\alpha\psi - \phi)
    \end{split}
\end{align}
on $\Sigma$. 
Here, $\gamma = \gamma(\phi,\psi)$ is a non-negative slip parameter that is related to tangential friction at the boundary. 
By this choice of $\Z$, inequality \eqref{IEQ:DISS:S:4} further reduces to
\begin{align}\label{IEQ:DISS:S:4*}
    \begin{split}
        0 &\geq [\delphi g + \ep\Grad\phi\cdot\n]\deltc\phi + [\delpsi g - \Divg(\delgpsi g) - \theta]\deltc\psi 
        \\
        &\quad + [\ep (\Grad\phi\cdot\n) 
            - \chi(K)(\alpha\psi-\phi)]\, \Gradg(\alpha\psi-\phi) \cdot\bw 
        \\
        &\quad - [\T_\Ga + \Gradg\psi\otimes\delgpsi g]:\Gradg\bw - m_\Ga(\psi)\abs{\Gradg\theta}^2 + \theta\J_\psi\cdot\n - \mu\J_\phi\cdot\n
    \end{split}
\end{align}
Analogously to the bulk, we introduce a tensor
\begin{align*}
    \S_\Ga = \T_\Ga + q\I + \Gradg\psi\otimes\delgpsi g \qquad\text{on~}\Sigma.
\end{align*}
Here, the scalar variable $q$ denotes the pressure on the surface, and the tensor $\S_\Ga$ is the \emph{viscous stress tensor} on the surface, which is assumed to be given by
\begin{align}
    \S_\Ga = 2\nu_\Ga(\psi)\Dg\bw.
\end{align}
Here, $\nu_\Ga = \nu_\Ga(\psi)$ is a non-negative function representing the viscosity of the fluids on the surface and $\Dg\bw$ is the symmetric surface gradient of $\bw$. Therefore, the total stress tensor on the surface reads as
\begin{align}\label{DEF:TG}
    \T_\Ga = \S_\Ga - q\I - \delta\kappa\Gradg\psi\otimes\Gradg\psi \qquad\text{on~}\Sigma.
\end{align}
This means that \eqref{BL:CLM:S} can be reformulated as
\begin{align}\label{PDE:NS:S}
    \delt(\sigma(\psi)\bw) + \Divgt(\bw\otimes(\sigma(\psi)\bw + \K)) = \Divgt\T_\Ga + \Z
    \quad\text{on $\Sigma$},
\end{align}
where, according to \eqref{ID:K} and \eqref{DEF:KPSI}, we have
\begin{align}
    \label{PDE:K}
    \K = -\frac{\tilde{\sigma}_2 - \tilde{\sigma}_1}{2}m_\Ga(\psi)\Gradg\theta,
\end{align}
$\T_\Gamma$ is given by \eqref{DEF:TG}, and $\Z$ is given by \eqref{DEF:Z}.
Moreover, by means of \eqref{DEF:TG} and the identity $q\I:\Gradg\bw = q\Divg\;\bw = 0$ on $\Sigma$, \eqref{IEQ:DISS:S:4} reduces to
\begin{align}\label{IEQ:DISS:S:5}
    \begin{split}
        0 &\geq [\delphi g + \ep\Grad\phi\cdot\n]\deltc\phi 
        + [\delpsi g - \Divg(\delgpsi g) - \theta]\deltc\psi 
        \\
        &\quad + [\ep (\Grad\phi\cdot\n) 
            - \chi(K)(\alpha\psi-\phi)]\, \Gradg(\alpha\psi-\phi) \cdot\bw
        \\
        &\quad 
        - \nu_\Gamma(\psi) \abs{\D_\Gamma \bw}^2
        - \gamma(\phi,\psi)\abs{\bw}^2 - m_\Ga(\psi)\abs{\Gradg\theta}^2 - (\beta\theta - \mu)m_\Om(\phi)\Grad\mu\cdot\n 
    \end{split}
\end{align}
on $\Sigma$. The first three terms in the third line of the right-hand side of \eqref{IEQ:DISS:S:5} are clearly non-positive. The fourth term is non-positive if, depending on a parameter $L\in [0,\infty]$, one of the following boundary conditions holds:%
\begin{subequations}
    \label{BC:MUTH}
    \begin{alignat}{2}
        \beta\theta - \mu &= 0 
        &&\qquad\text{on $\Sigma$ if $L=0$}, 
        \\
        m_\Om(\phi)\Grad\mu\cdot\n &= \tfrac1L(\beta\theta - \mu) &&\qquad\text{on $\Sigma$ if $L \in (0,\infty)$},  
        \\
        m_\Om(\phi)\Grad\mu\cdot\n &= 0 
        &&\qquad\text{on $\Sigma$ if $L=\infty$}.
    \end{alignat}
\end{subequations}
Similarly, to deal with the first and the second line of \eqref{IEQ:DISS:S:5}, we assume $\phi$ and $\psi$ to fulfill one of the boundary conditions
\begin{subequations}
    \label{BC:PHIPSI}
    \begin{alignat}{2}
        \label{BC:PHIPSI:1}
        \alpha\psi - \phi &= 0 
        &&\qquad\text{on $\Sigma$ if $K=0$}, 
        \\
        \label{BC:PHIPSI:2}
        \ep\Grad\phi\cdot\n &= \tfrac1K(\alpha\psi - \phi) &&\qquad\text{on $\Sigma$ if $K\in (0,\infty)$}, 
        \\
        \label{BC:PHIPSI:3}
        \Grad\phi\cdot\n &= 0 
        &&\qquad\text{on $\Sigma$ if $K=\infty$},
    \end{alignat}
\end{subequations}
where $K\in[0,\infty]$ is the number that was introduced in \eqref{DEF:GL:S}.
To show that \eqref{IEQ:DISS:S:5} is actually fulfilled, we need to distinguish the cases $K = 0$, $K\in(0,\infty)$ and $K=\infty$. 
\begin{itemize}[leftmargin=*]
\item \emph{The case} $K = 0$. Without loss of generality, we assume that $\alpha\neq 0$. The case $\alpha=0$ can be handled similarly, but the computations are even easier.
Then, we immediately see that the second line in \eqref{IEQ:DISS:S:5} vanishes. In the definition of $g$, we have $\chi(K) = \chi(0) = 0$ and thus $\delphi g = 0$ on $\Sigma$. Moreover, due to \eqref{BC:PHIPSI:1}, $\alpha^{-1}\phi$ can be interpreted as an extension of $\psi$ in some neighborhood of $\Sigma$, which entails that $\deltc\phi = \alpha\deltc\psi$ on $\Sigma$. Consequently, the first line of the right-hand side of \eqref{IEQ:DISS:S:5} can be reformulated as
\begin{align*}
    [\delpsi g - \Divg(\delgpsi g) + \alpha\ep\Grad\phi\cdot\n - \theta]\deltc\psi \qquad\text{on~}\Sigma.
\end{align*}
Choosing
\begin{align*}
    \theta = \delpsi g - \Divg(\delgpsi g) + \alpha\ep\Grad\phi\cdot\n \qquad\text{on~}\Sigma,
\end{align*}
we thus ensure that the first line of \eqref{IEQ:DISS:S:5} also vanishes. This shows that inequality \eqref{IEQ:DISS:S:5} is fulfilled. 

\item\emph{The case} $K\in(0,\infty)$. Boundary condition \eqref{BC:PHIPSI:2} entails that the second line of \eqref{IEQ:DISS:S:5} vanishes.
In view of the definition of $g$, we further have
\begin{alignat*}{2}
    \delphi g &= -\chi(K)(\alpha\psi - \phi) = -\ep\Grad\phi\cdot\n &&\qquad\text{on~}\Sigma, \\
    \delpsi g &=\frac1\delta G^\prime(\psi) + \alpha\chi(K)(\alpha\psi - \phi) = \frac1\delta G^\prime(\psi) + \alpha\ep\Grad\phi\cdot\n &&\qquad\text{on~}\Sigma.
\end{alignat*}
This implies the first line of \eqref{IEQ:DISS:S:5} can be rewritten as
\begin{align*}
    [\delpsi g - \Divg(\delgpsi g) - \theta]\deltc\psi\qquad\text{on~}\Sigma.
\end{align*}
Hence, the choice
\begin{align*}
    \theta = \delpsi g - \Divg(\delgpsi g) \qquad\text{on~}\Sigma,
\end{align*}
ensures that the first line also vanishes. Consequently, inequality \eqref{IEQ:DISS:S:5} is fulfilled.  

\item\emph{The case} $K = \infty$.
Since $\chi(K) = \chi(\infty) = 0$, it directly follows from \eqref{BC:PHIPSI:3} that the second line in \eqref{IEQ:DISS:S:5} vanishes. 
Moreover, since $\delphi g = 0$ on $\Sigma$, the first line of \eqref{IEQ:DISS:S:5} can be reformulated as
\begin{align*}
    [\delpsi g - \Divg(\delgpsi g) - \theta]\deltc\psi\qquad\text{on~}\Sigma.
\end{align*}
Thus, choosing
\begin{align*}
    \theta = \delpsi g - \Divg(\delgpsi g) \qquad\text{on~}\Sigma,
\end{align*}
ensures that the first line also vanishes. 
This shows that inequality \eqref{IEQ:DISS:S:5} is fulfilled.  
\end{itemize}
\noindent By the definition of $g$ (see \eqref{DEF:GL:S}), we conclude in all cases $K\in[0,\infty]$ that 
\begin{align}\label{PDE:THETA}
    \theta = - \delta\kappa\Lapg\psi + \frac1\delta G^\prime(\psi) + \alpha\ep\Grad\phi\cdot\n\qquad\text{on~}\Sigma.
\end{align}
Eventually, substituting \eqref{DEF:KPSI} and \eqref{DEF:JPSI} into \eqref{EQ:PSI}, we obtain
\begin{align}
    \label{PDE:PSI}
    \delt\psi + \Divg(\psi\bw) = \Divg(m_\Ga(\psi)\Gradg\theta) - \beta m_\Om(\phi)\Grad\mu\cdot\n\qquad\text{on~}\Sigma.
\end{align}

In summary, we have thus derived system \eqref{System} along with formulas \eqref{System:9} and \eqref{System:10}-\eqref{System:15} with the following correspondences:
\begin{center}
\makeatletter \setlength{\tabcolsep}{2pt} \makeatother
\begin{tabular}{rclcrclcrcl}
    \eqref{System:1} 
    &$\widehat{=}$ 
    &\eqref{PDE:NS:B}, \eqref{PDE:DIV:B}, 
    &$\phantom{xx}$ 
    &\eqref{System:2} 
    &$\widehat{=}$ 
    &\eqref{PDE:NS:S}, \eqref{PDE:DIV:S}, 
    &$\phantom{xx}$ 
    &\eqref{System:3} 
    &$\widehat{=}$ 
    &\eqref{BC:V},
    \\
    \eqref{System:4} 
    &$\widehat{=}$ 
    &\eqref{PDE:PHI}, 
    &$\phantom{xx}$ 
    &\eqref{System:5} 
    &$\widehat{=}$ 
    &\eqref{PDE:MU},
    &$\phantom{xx}$ 
    &\eqref{System:6} 
    &$\widehat{=}$ 
    &\eqref{PDE:PSI},
    \\
    \eqref{System:7} 
    &$\widehat{=}$ 
    &\eqref{PDE:THETA},
    &$\phantom{xx}$ 
    &\eqref{System:8} 
    &$\widehat{=}$ 
    &\eqref{BC:PHIPSI}, \eqref{BC:MUTH},
    &$\phantom{xx}$ 
    &\eqref{System:9} 
    &$\widehat{=}$ 
    &\eqref{DEF:RHO*},
    \\
    \eqref{System:10} 
    &$\widehat{=}$ 
    &\eqref{DEF:SIGMA*},
    &$\phantom{xx}$ 
    &\eqref{System:11} 
    &$\widehat{=}$ 
    &\eqref{DEF:T},
    &$\phantom{xx}$ 
    &\eqref{System:12} 
    &$\widehat{=}$ 
    &\eqref{DEF:TG},
    \\
    \eqref{System:13} 
    &$\widehat{=}$ 
    &\eqref{DEF:Z},
    &$\phantom{xx}$ 
    &\eqref{System:14} 
    &$\widehat{=}$ 
    &\eqref{PDE:J},
    &$\phantom{xx}$ 
    &\eqref{System:15} 
    &$\widehat{=}$ 
    &\eqref{PDE:K}, \eqref{PDE:JG}.
\end{tabular}
\end{center}

\section{Main results of the mathematical analysis}
\label{Section:MainResults}

\subsection{Assumptions}

Before stating the main results, we first fix some assumptions that are supposed to hold throughout the remainder of this paper.

\begin{enumerate}[label=\textnormal{\bfseries(A\arabic*)}]
    \item  \label{Assumption:Domain} We consider a bounded domain $\emptyset\neq \Om\subset\R^d$ with $d\in\{2,3\}$ with $C^2$-boundary $\Ga\coloneqq\del\Om$ and a final time $T>0$. We further use the notation
    \begin{align*}
        Q\coloneqq \Om\times(0,T), \quad\Sigma\coloneqq\Ga\times(0,T).
    \end{align*}
    
    \item \label{Assumption:Constants} The constants occurring in the system \eqref{System} satisfy $K,L\in[0,\infty]$, $\ep, \delta, \kappa > 0$, $\alpha\in[-1,1]$ and $\beta\in \R$ with $\alpha\beta\abs{\Omega} + \abs{\Gamma} \neq 0$. Since the choice of $\ep, \delta$ and $\kappa$ has no impact on the mathematical analysis, we will simply set $\ep = \delta = \kappa = 1$ without loss of generality.
    
    \item \label{Assumption:Density} The density functions $\rho$ and $\sigma$ are given by
    \begin{align*}
    	\rho(s) &= \frac{\tilde{\rho}_2 - \tilde{\rho}_1}{2}s + \frac{\tilde{\rho}_1 + \tilde{\rho}_2}{2}, \\
    	\sigma(s) &= \frac{\tilde{\sigma}_2 - \tilde{\sigma}_1}{2}s + \frac{\tilde{\sigma}_1 + \tilde{\sigma}_2}{2}
    \end{align*}
    for all $s\in[-1,1]$, respectively. Here, $\tilde{\rho}_1, \tilde{\rho}_2>0$ and $\tilde{\sigma}_1, \tilde{\sigma}_2>0$ are the specific densities of the two fluid components in the bulk and on the surface, respectively.
    As the derivatives of $\rho$ and $\sigma$ are constant, we will usually write
    \begin{align*}
        \rho^\prime \equiv \rho^\prime(s) 
        =\frac{\tilde{\rho}_2 - \tilde{\rho}_1}{2}
        \quad\text{and}\quad
        \sigma^\prime \equiv \sigma^\prime(s)
        = \frac{\tilde{\sigma}_2 - \tilde{\sigma}_1}{2}\,.
    \end{align*}
    This means that $\rho^\prime$ and $\sigma^\prime$ are interpreted as constants. Moreover, we use the notation
    \begin{alignat*}{2}
        &\rho_\ast \coloneqq \min\{\tilde\rho_1,\tilde\rho_2\}, \qquad 
        &&\rho^\ast \coloneqq \max\{\tilde\rho_1,\tilde\rho_2\},
        \\
        &\sigma_\ast \coloneqq \min\{\tilde\sigma_1,\tilde\sigma_2\}, \qquad 
        &&\sigma^\ast \coloneqq \max\{\tilde\sigma_1,\tilde\sigma_2\},
    \end{alignat*}
    By this definition, we clearly have
    \begin{align*}
        0 < \rho_\ast \leq \rho(s) \leq \rho^\ast \quad\text{and}\quad 0 < \sigma_\ast \leq \sigma(s) \leq \sigma^\ast \qquad\text{for all~}s\in[-1,1].
    \end{align*}
    
    \item \label{Assumption:Coefficients} The mobility functions $m_\Om:\R\rightarrow\R$ and $m_\Ga:\R\rightarrow\R$ are assumed to be constant. Thus, without loss of generality, we assume that $m_\Om = m_\Ga = 1$. Furthermore, the coefficients $\nu_\Om,\nu_\Ga:\R\rightarrow\R$ and $\gamma:\R^2\rightarrow\R$ are assumed to be continuous, bounded, and uniformly positive. This means that there exist positive constants $\nu_\ast,\nu^\ast,\gamma_\ast,\gamma^\ast$ such that for all $s\in\R$,
    \begin{align*}
        0 < \nu_\ast \leq \nu_\Om(s), \nu_\Ga(s) \leq \nu^\ast,
    \end{align*}
    and, for all $(s,r)\in\R^2$,
    \begin{align*}
    	0 < \gamma_\ast \leq \gamma(s,r) \leq \gamma^\ast.
    \end{align*}
    
    \item \label{Assumption:Potential} We assume that the potentials $F,G:\R\rightarrow\R$ are of the form
    \begin{align*}
        F(s) = F_1(s) + F_2(s), \qquad G(s) = G_1(s) + G_2(s) \qquad\text{for~}s\in\R,
    \end{align*}
    where $F_1,G_1\in C([-1,1])\cap C^2(-1,1)$, $F_1(s) = G_1(s) = + \infty$ for all $s\in\R\setminus[-1,1]$,
    \begin{align*}
        \lim_{s\searrow -1} F_1^\prime(s) = \lim_{s\searrow -1} G_1^\prime(s) = - \infty \quad\text{and}\quad \lim_{s\nearrow 1} F_1^\prime(s) = \lim_{s\nearrow 1} G_1^\prime(s) = + \infty,
    \end{align*}
    and there exists a constant $\Theta > 0$ such that
    \begin{align}
        F_1^{\prime\prime}(s) \geq \Theta \quad\text{and}\quad G_1^{\prime\prime}(s) \geq \Theta \qquad\text{for all~}s\in(-1,1).
    \end{align}
    Without loss of generality, we assume that $F_1(0) = G_1(0) = 0$ and $F_1^\prime(0) = G_1^\prime(0) = 0$. For $F_2$ and $G_2$, we assume that $F_2, G_2\in C^1(\R)$ with globally Lipschitz derivatives, respectively. Lastly, we require the singular part of the boundary potential to dominate the singular part of the bulk potential in the sense that there exist constants $\kappa_1, \kappa_2 > 0$ such that
    \begin{align}\label{DominationProperty}
        \abs{F_1^\prime(\alpha s)} \leq \kappa_1\abs{G_1^\prime(s)} + \kappa_2 \qquad\text{for all~}s\in(-1,1).
    \end{align}
\end{enumerate}

\subsection{Main results}

As pointed out in \ref{Assumption:Constants}, we set $\ep = \kappa = \delta = 1$, as their particular choice has no impact on the mathematical analysis. Furthermore, as stated in \ref{Assumption:Coefficients}, we will only consider the case of constant mobility functions $m_\Om$ and $m_\Ga$, and therefore, we assume without loss of generality that $m_\Om = m_\Ga = 1$. With these choices, and under consideration of \eqref{eqs:NSCH:pressure}, our system \eqref{System} reduces to
\begin{subequations}\label{eqs:NSCH}
    \begin{alignat}{2}
        \label{eqs:NSCH:1}
        &\delt(\rho(\phi)\bv) + \Div(\bv\otimes(\rho(\phi)\bv + \J)) - \Div(2\nu_\Om(\phi)\D\bv) + \Grad\overline{p} = \mu\Grad\phi &&\qquad\text{in~}Q\\
        \label{eqs:NSCH:2}
        &\Div\;\bv = 0 &&\qquad\text{in~}Q 
        \\
        \label{eqs:NSCH:3}
        &\delt(\sigma(\psi)\bw) + \Divgt(\bw\otimes(\sigma(\psi)\bw + \K)) - \Divgt(2\nu_\Ga(\psi)\Dg\bw) + \Gradg\overline{q} \nonumber 
        \\
        &\qquad = \theta\Gradg\psi - 2\nu_\Om(\phi)\big[ \D\bv\,\n \big]_\tau + \tfrac12(\J\cdot\n)\bw + \tfrac12(\J_\Ga\cdot\n)\bw - \gamma(\phi,\psi)\bw &&\qquad\text{on~}\Sigma, 
        \\
        \label{eqs:NSCH:4}
        &\Divg\;\bw = 0 &&\qquad\text{on~}\Sigma 
        \\
        \label{eqs:NSCH:5}
        &\bv\vert_\Ga = \bw,\qquad \bv\cdot\n = 0 &&\qquad\text{on~}\Sigma,
        \\[0.5em]
        \label{eqs:NSCH:6}
        &\delt\phi  + \Div(\phi\bv) = \Lap\mu &&\qquad\text{in~}Q, 
        \\
        \label{eqs:NSCH:7}
        &\mu = -\Lap\phi + F^\prime(\phi) &&\qquad\text{in~}Q, 
        \\
        \label{eqs:NSCH:8}
        &\delt\psi + \Divg(\psi\bw) = \Lapg\theta - \beta\deln\mu &&\qquad\text{on~}\Sigma, 
        \\
        \label{eqs:NSCH:9}
        &\theta = -\Lapg\psi + G^\prime(\psi) + \alpha\deln\phi &&\qquad\text{on~}\Sigma, 
        \\
        \label{eqs:NSCH:10}
        &\begin{cases}
        	K\deln\phi = \alpha\psi - \phi, &\text{if~} K\in[0,\infty), 
            \\
        	\deln\phi = 0, &\text{if~} K = \infty,
        \end{cases}
        \quad
        \begin{cases}
        	L\deln\mu = \beta\theta - \mu, &\text{if~} L\in[0,\infty), 
            \\
        	\deln\mu = 0, &\text{if~} L = \infty,
        \end{cases}
        &&\qquad\text{on~}\Sigma, 
        \\[0.5em]
        \label{eqs:NSCH:11}
        &\bv\vert_{t=0} = \bv_0, \qquad \phi\vert_{t=0} = \phi_0 &&\qquad\text{in~}\Om, 
        \\
        \label{eqs:NSCH:12}
        &\bw\vert_{t=0} = \bw_0, \qquad \psi\vert_{t=0} = \psi_0 &&\qquad\text{on~}\Ga.
    \end{alignat}
\end{subequations}

The total energy associated with this system reads as
\begin{align*}
	E_{\mathrm{tot}}(\bv,\bw,\phi,\psi) &= \intO \frac12 \rho(\phi)\abs{\bv}^2\dx + \intG \frac12 \sigma(\psi)\abs{\bw}^2\dG + \intO \frac12 \abs{\Grad\phi}^2 + F(\phi) \dx \\
	&\quad + \intG \frac12 \abs{\Gradg\psi}^2 + G(\psi)\dG + \chi(K) \intG \frac12 (\alpha\psi - \phi)^2\dG.
\end{align*}

The notion of weak solution to \eqref{eqs:NSCH} is defined as follows:

\begin{definition}\label{Definition:WeakSolution}
    Suppose the assumptions \ref{Assumption:Domain}-\ref{Assumption:Potential} hold. Let $T > 0$, $K,L\in[0,\infty]$, $(\bv_0,\bw_0)\in\mathbfcal{L}^2_\Div$, and $(\phi_0,\psi_0)\in\mathcal{H}^1_{K,\alpha}$ such that
    \begin{subequations}\label{Assumption:InitialCondition}
    	\begin{align}\label{Assumption:InitalCondition:Int}
        	\norm{\phi_0}_{L^\infty(\Om)} \leq 1, \qquad \norm{\psi_0}_{L^\infty(\Ga)} \leq 1.
    	\end{align}
    	Furthermore, we assume that
    	\begin{align}\label{Assumption:InitalCondition:L}
        	\beta\mean{\phi_0}{\psi_0}\in(-1,1), \qquad \mean{\phi_0}{\psi_0}\in(-1,1), \qquad\text{if~}L\in[0,\infty),
    	\end{align}
    	and
    	\begin{align}\label{Assumption:InitalCondition:Inf}
        	\meano{\phi_0}\in(-1,1), \qquad \meang{\psi_0}\in(-1,1), \qquad\text{if~} L = \infty.
    	\end{align}
    \end{subequations}
    The sextuplet $(\bv,\bw,\phi,\psi,\mu,\theta)$ is called a weak solution of the system \eqref{eqs:NSCH} on $[0,T]$ if the following properties hold:
    \begin{enumerate}[label=\textnormal{(\roman*)},leftmargin=*]
        \item\label{DEF:WF:1} The functions $\bv, \bw, \phi, \psi, \mu$ and $\theta$ have the regularities
        \begin{subequations}\label{REG}
            \begin{align}
                \scp{\bv}{\bw} &\in C_w([0,T];\mathbfcal{L}^2_\Div)\cap L^2(0,T;\mathbfcal{H}^1_{0,\Div}), \label{REG:VW}
                \\
                \scp{\phi}{\psi} &\in C([0,T];\mathcal{L}^2)\cap H^1(0,T;(\mathcal{H}^1_{L,\beta})^\prime)\cap L^\infty(0,T;\mathcal{H}^1_{K,\alpha}) 
                , \label{REG:PP}
                \\
                \scp{\mu}{\theta}&\in L^2(0,T;\mathcal{H}_{L,\beta}^1) \label{REG:MT}, 
                \\
                \scp{F^\prime(\phi)}{G^\prime(\psi)}&\in L^2(0,T;\mathcal{L}^2),\label{REG:POT}
            \end{align}
            and it holds $\abs{\phi} < 1$ a.e.~in $Q$ and $\abs{\psi} < 1$ a.e.~on $\Sigma$.
    \end{subequations}
    
    \item\label{DEF:WF:2} The initial conditions are satisfied in the following sense:
    \begin{align}
        \label{WF:INI}
        (\bv,\bw)\vert_{t=0} = (\bv_0,\bw_0) \quad\text{a.e.~in~}\Om\times\Ga,
        \qquad
        (\phi,\psi)\vert_{t=0} = (\phi_0,\psi_0) \quad\text{a.e.~in~}\Om\times\Ga.
    \end{align}
    
    \item\label{DEF:WF:3} The variational formulations
    \begin{subequations}
        \begin{align}\label{WF:VW}
            &-\bigscp{\rho(\phi)\bv}{\delt\wv}_{L^2(Q)} 
                - \bigscp{\sigma(\psi)\bw}{\delt\ww}_{L^2(\Sigma)} 
                + \bigscp{(\J\cdot\Grad)\bv}{\wv}_{L^2(Q)} 
            \nonumber \\
            &\qquad
                + \bigscp{(\K\cdot\Gradg)\bw}{\ww}_{L^2(\Sigma)} 
                + \bigscp{\rho(\phi)(\bv\cdot\Grad)\bv}{\wv}_{L^2(Q)} 
                + \bigscp{\sigma(\psi)(\bw\cdot\Gradg)\bw}{\ww}_{L^2(\Sigma)} 
            \nonumber \\
            &\qquad 
                + \bigscp{2\nu_\Om(\phi)\D\bv}{\D\wv}_{L^2(Q)} 
                + \bigscp{2\nu_\Ga(\psi)\Dg\bw}{\Dg\ww}_{L^2(\Sigma)}  
                + \bigscp{\gamma(\phi,\psi)\bw}{\ww}_{L^2(\Sigma)}
            \\
            &\quad= \bigscp{\mu\Grad\phi}{\wv}_{L^2(Q)} 
                + \bigscp{\theta\Gradg\psi}{\ww}_{L^2(\Sigma)} 
                + \tfrac{1}{2}\chi(L)(\beta\sigma^\prime + \rho^\prime)\bigscp{(\beta\theta - \mu)\bw}{\ww}_{L^2(\Sigma)}, 
                \nonumber
        \\[2ex]
            \label{WF:PP}
            &\bigscp{\phi}{\delt\zeta}_{L^2(Q)} 
                + \bigscp{\psi}{\delt\xi}_{L^2(\Sigma)} 
                + \bigscp{\phi\bv}{\Grad\zeta}_{L^2(Q)} 
                + \bigscp{\psi\bw}{\Gradg\xi}_{L^2(\Sigma)} 
            \\
            &\quad= 
                \bigscp{\Grad\mu}{\Grad\zeta}_{L^2(Q)} 
                + \bigscp{\Gradg\theta}{\Gradg\xi}_{L^2(\Sigma)} 
                + \chi(L)\bigscp{\beta\theta-\mu}
                    {\beta\xi-\zeta}_{L^2(\Sigma)} \nonumber                
        \\[2ex]
            \label{WF:MT}
            &\bigscp{\mu}{\eta}_{L^2(Q)} 
            + \bigscp{\theta}{\vartheta}_{L^2(\Sigma)} \nonumber 
            \\
            &\quad= \bigscp{\Grad\phi}{\Grad\eta}_{L^2(Q)}  
            + \bigscp{F^\prime(\phi)}{\eta}_{L^2(Q)}  
            + \bigscp{\Gradg\psi}{\Gradg\vartheta}_{L^2(\Sigma)} 
            + \bigscp{G^\prime(\psi)}{\vartheta}_{L^2(\Sigma)}  
            \\
            &\qquad + \chi(K) \bigscp{\alpha\psi - \phi}
                {\alpha\vartheta - \eta}_{L^2(\Sigma)} \nonumber
        \end{align}          
        where $\J = -\rho^\prime\,\Grad\mu$ and $\K = -\sigma^\prime\,\Gradg\theta$,
        hold for all test functions $(\wv,\ww)\in C_c^\infty(0,T;\mathbfcal{H}_{0,\Div}^2)$, $(\zeta,\xi)\in C_c^\infty(0,T;\mathcal{H}^1_{L,\beta})$ and $(\eta,\vartheta)\in L^2(0,T;\mathcal{H}^1_{K,\alpha})$.
    \end{subequations}
    
    \item\label{DEF:WF:4} The functions $\phi$ and $\psi$ satisfy the mass conservation law
    \begin{align}\label{WF:MCL}
        \begin{dcases}
            \beta\intO \phi(t)\dx + \intG \psi(t)\dG = \beta\intO \phi_0 \dx + \intG \psi_0\dG, &\textnormal{if~} L\in[0,\infty), \\
            \intO\phi(t)\dx = \intO\phi_0\dx \quad\textnormal{and}\quad \intG\psi(t)\dG = \intG\psi_0\dG, &\textnormal{if~} L = \infty,
        \end{dcases}
    \end{align}
    for all $t\in[0,T]$.
    
    \item\label{DEF:WF:5} The energy inequality
    \begin{align}\label{WF:DISS}
        &E(\bv(t),\bw(t),\phi(t),\psi(t)) + \int_0^t\intO 2\nu_\Om(\phi)\abs{\D\bv}^2\dxs + \int_0^t\intG 2\nu_\Ga(\psi)\abs{\Dg\bw}^2\dGs \nonumber \\
        &\qquad + \int_0^t\intG \gamma(\phi,\psi)\abs{\bw}^2 \dGs + \int_0^t\intO \abs{\Grad\mu}^2\dxs \\
        &\qquad + \int_0^t\intG \abs{\Gradg\theta}^2\dGs + \chi(L)\int_0^t\intG (\beta\theta - \mu)^2\dGs \leq E(\bv_0,\bw_0,\phi_0,\psi_0) \nonumber
    \end{align}
    holds for all $t\in[0,T)$.
    \end{enumerate}
\end{definition}

We are now in a position to state the main theorem of this paper.

\begin{theorem}\label{Theorem:Main}
    Suppose the assumptions \ref{Assumption:Domain}-\ref{Assumption:Potential} hold. Let $K,L\in[0,\infty]$, $(\bv_0,\bw_0)\in\mathbfcal{L}^2_\Div$ and let $(\phi_0,\psi_0)\in\mathcal{H}^1_{K,\alpha}$ satisfy \eqref{Assumption:InitialCondition}. If $L = 0$, additionally assume that
    \begin{align}\label{Assumption:Densities:0}
    	\beta(\tilde{\sigma}_2 - \tilde{\sigma}_1) = -(\tilde{\rho}_2 - \tilde{\rho}_1).
	\end{align}     
	Then, there exists at least one weak solution $(\bv,\bw,\phi,\psi,\mu,\theta)$ to system \eqref{eqs:NSCH} in the sense of Definition~\ref{Definition:WeakSolution}.
\end{theorem}

\pagebreak[2]

\begin{proof}
    The theorem is proved in two steps.
    
    \textbf{Step~1:} For $L\in(0,\infty]$, we prove the existence of a weak solution via a semi-Galerkin scheme combined with a fixed-point argument. In this approach, the Galerkin basis to discretize the bulk-surface Navier--Stokes subsystem \eqref{eqs:NSCH:1}-\eqref{eqs:NSCH:5} consists of eigenfunctions of a related \textit{bulk-surface Stokes system} (cf.~\eqref{App:BSS:system}). Therefore, the proof crucially relies on the well-posedness results, the regularity theory, and the spectral theory for the bulk-surface Stokes problem, which are developed in Section~\ref{Section:BulkSurfaceStokes}. 
    The statement of Theorem~\ref{Theorem:Main} in the case $L\in(0,\infty]$ is then shown in Theorem~\ref{Theorem:L>0}.
    
    \textbf{Step~2:} For $L = 0$, the procedure described in Step~1 does not work since admissible test functions in \eqref{WF:PP} need to satisfy the corresponding trace relation. Instead, we construct a weak solution via the asymptotic limit $L \to 0$ of weak solutions associated with $L>0$. This result is shown in Theorem~\ref{Theorem:L->0}.

    This means that Theorem~\ref{Theorem:Main} is established by combining Theorem~\ref{Theorem:L>0} and Theorem~\ref{Theorem:L->0}.
\end{proof}

\medskip

\begin{remark} \label{REM:REG:PP}
	We point out that, since $(\bv,\bw)\in L^2(0,T;\mathbfcal{L}^2_\Div)$, any weak solution $(\bv,\bw,\phi,\psi,\mu,\theta)$ automatically satisfies
	\begin{align*}
		(\phi,\psi)\in L^2(0,T;\mathcal{W}^{2,6}), \qquad (F^\prime(\phi),G^\prime(\psi))\in L^2(0,T;\mathcal{L}^6),
	\end{align*}
	according to \cite[Theorem~3.2]{Giorgini2025}. 
    Consequently, the equations
    \begin{subequations}
    \begin{alignat}{2}
        \mu &= -\Lap\phi + F^\prime(\phi) &&\qquad\text{a.e.~in~}Q, \label{EQ:MU:STRG}
        \\
        \theta &= -\Lapg\psi + G^\prime(\psi) + \alpha\deln\phi, &&\qquad\text{a.e.~on~}\Sigma, 
        \label{EQ:THETA:STRG}
        \\
        K \deln \phi &= \alpha \psi - \phi
        &&\qquad\text{a.e.~on~}\Sigma
        \label{BC:PP:STRG}
    \end{alignat}
    \end{subequations}
    are fulfilled in the strong sense.
    Moreover, if $K\in(0,\infty]$, we have
	\begin{align*}
		(\phi,\psi)\in L^4(0,T;\mathcal{H}^2),
	\end{align*}
	whereas in the case $K = 0$, we only have
	\begin{align*}
		(\phi,\psi)\in L^3(0,T;\mathcal{H}^2),
	\end{align*}
	see \cite[Theorem~3.3]{Giorgini2025}.
\end{remark}

\medskip

\begin{remark} \label{REM:COMP}
	The compatibility condition \eqref{Assumption:Densities:0} in the case $L=0$ means that $\beta\sigma^\prime + \rho^\prime = 0$. Therefore, it ensures that the term 
    \begin{align*}
        \tfrac{1}{2}\chi(L)(\beta\sigma^\prime + \rho^\prime)\bigscp{(\beta\theta - \mu)\bw}{\ww}_{L^2(\Sigma)}
    \end{align*}
    in the weak formulation \eqref{WF:VW} vanishes, which is crucial in the proof of Theorem~\ref{Theorem:L->0}. Condition \eqref{Assumption:Densities:0} is related to the compatibility condition $\tilde\rho_1 = \tilde\rho_2$ that was imposed in \cite{Gal2023a} for the analysis of system \eqref{eqs:GK} from \cite{Giorgini2023} with $L=0$. From the viewpoint of mathematical analysis, this is an advantage of the new model \eqref{System} compared to \eqref{eqs:GK} since even in the case $L=0$ we can allow for unmatched densities of the fluids.
\end{remark}

\medskip

Under additional assumptions on the initial data $(\phi_0,\psi_0)$, we immediately infer that the solution component $(\phi,\psi,\mu,\theta)$ satisfies the Cahn--Hilliard subsystem \eqref{eqs:NSCH:6}-\eqref{eqs:NSCH:10} in the strong sense.

\begin{theorem}\label{THM:CHSUB}
	Suppose the assumptions \ref{Assumption:Domain}-\ref{Assumption:Potential} hold. Let $K\in[0,\infty]$, $L\in(0,\infty]$, $(\bv_0,\bw_0)\in\mathbfcal{L}^2_\Div$ and let $(\phi_0,\psi_0)\in\mathcal{H}^1_{K,\alpha}$ satisfy \eqref{Assumption:InitialCondition}. We further assume that the following compatibility condition holds:
    \begin{enumerate}[label=\textnormal{\bfseries(C)},topsep=0ex]
        \item \label{cond:MT:0} There exists $\scp{\mu_0}{\theta_0}\in\mathcal{H}^1$ such that for all $\scp{\eta}{\vartheta}\in\mathcal{H}			^1_{K,\alpha}$, it holds
        \begin{align*}
        \begin{aligned}
            &\intO\mu_0\eta\dx + \intG\theta_0\eta_\Ga\dG 
            \\
            &= \intO\Grad\phi_0\cdot\Grad\eta + F^\prime(\phi_0)\eta\dx + \intG\Gradg\psi_0\cdot\Gradg\vartheta + G^\prime(\psi_0)\vartheta\dG 
            \\
            &\quad + \chi(K)\intG(\alpha\psi_0 - \phi_0)(\alpha\vartheta - \eta)\dG.
        \end{aligned}
        \end{align*}
    \end{enumerate}
    Let $(\bv,\bw,\phi,\psi,\mu,\theta)$ be a corresponding weak solution in the sense of Definition~\ref{Definition:WeakSolution}. Then $(\phi,\psi,\mu,\theta)$ has the following regularity:
    \begin{align*}
        (\phi,\psi) &\in L^\infty(0,T;\mathcal{W}^{2,6})\cap \big(C(\overline{Q})\times C(\overline{\Sigma})\big), \\
        (\delt\phi,\delt\psi)&\in L^\infty(0,T;(\mathcal{H}^1)^\prime)\cap L^2(0,T;\mathcal{H}^1), \\
        (\mu,\theta)&\in L^\infty(0,T;\mathcal{H}^1)\cap L^2(0,T;\mathcal{H}^2), \\
        (F^\prime(\phi), G^\prime(\psi))&\in L^2(0,T;\mathcal{L}^\infty)\cap L^\infty(0,T;\mathcal{L}^6).
    \end{align*}
\end{theorem}

\begin{proof}
    The assertion follows immediately from \cite[Theorem 4.1]{Giorgini2025} since $(\bv,\bw)\in L^\infty(0,T;\mathcal{L}^2)\cap L^2(0,T;\mathcal{H}^1_{0,\Div})$.
\end{proof}

\medskip

\begin{remark}
    If $d = 2$, the statement of Theorem~\ref{THM:CHSUB} also holds in the case $L = 0$, see \cite[Remark~3.6]{Giorgini2025}. However, it is not yet clear whether the statement remains true in the case $L = 0$ if $d = 3$.
\end{remark}

\section{Analysis of a bulk-surface Stokes system}
\label{Section:BulkSurfaceStokes}
In this section, we analyze a new bulk-surface Stokes system and prove the existence of a unique strong solution. Afterwards, we study the corresponding bulk-surface Stokes operator and its spectral properties. These results are essential for our proof of Theorem~\ref{Theorem:L>0} as we use the eigenfunctions of the bulk-surface Stokes operator as the basis functions in our semi-Galerkin scheme. 

To this end, we consider the system
\begin{subequations}
    \label{App:BSS:system}
    \begin{alignat}{2}
        \label{App:BSS:system:1}
        -\Div(2\nu_\Om(\phi)\D \bv) + \Grad p &= \f &&\qquad\text{in~}\Om, 
        \\
        \label{App:BSS:system:2}
        \Div\, \bv &= g &&\qquad\text{in~}\Om, 
        \\
        \label{App:BSS:system:3}
        -\Divgt(2\nu_\Ga(\psi)\Dg\bw) + 2\nu_\Om(\phi)\big[ \D\bv\,\n \big]_\tau + \Gradg q + \gamma(\phi,\psi)\bw &= \f_\Ga &&\qquad\text{on~}\Ga, 
        \\
        \label{App:BSS:system:4}
        \Divg\,\bw &= g_\Ga &&\qquad\text{on~}\Ga, 
        \\
        \label{App:BSS:system:5}
        \bv\vert_\Ga = \bw, \quad \bv\cdot \n &= 0 &&\qquad\text{on~}\Ga,
    \end{alignat}
\end{subequations}
for given source terms $(\f,\f_\Gamma)\in(\mathbfcal{H}^1_0)^\prime$, $(g,g_\Ga)\in\mathcal{L}^2_{(0)}$ and measurable functions $\phi:\Om\rightarrow\R$ and $\psi:\Ga\rightarrow\R$. Furthermore, the viscosity coefficients $\nu_\Om,\nu_\Ga:\R\rightarrow\R$ and the friction coefficient $\gamma:\R^2\rightarrow\R$ satisfy the assumptions stated in \ref{Assumption:Coefficients}.

\subsection{Existence and uniqueness of weak solutions}
We start our analysis by considering the case of incompressible fluids, i.e., $g = 0$ in $\Om$ and $g_\Ga = 0$ on $\Ga$. In this case, the weak formulation of \eqref{App:BSS:system} reads as
\begin{align}\label{Stokes:WF}
	\begin{split}
    	&\intO 2\nu_\Om(\phi)\D\bv:\D\wv\dx + \intG 2\nu_\Ga(\psi)\Dg\bw:\Dg\ww\dG + \intG \gamma(\phi,\psi)\bw\cdot\ww\dG \\
    	&\qquad = \bigang{(\f,\f_\Ga)}{(\wv,\ww)}_{\mathbfcal{H}^1_0}
    \end{split}
\end{align}
for all $(\wv,\ww)\in \mathbfcal{H}^1_{0,\Div}$.

To prove the existence of a unique weak solution to \eqref{App:BSS:system}, we start by establishing the following bulk-surface Korn-type inequality, which yields the coercivity of the bilinear form corresponding to the problem.

\begin{lemma}\label{Lemma:Korn}
    There exists a constant $C_K > 0$ such that
    \begin{align}
        \norm{(\bv,\bw)}_{\mathbfcal{H}^1} \leq C_K\big(\norm{\D\bv}_{\mathbf{L}^2(\Om)} + \norm{\Dg\bw}_{\mathbf{L}^2(\Ga)} + \norm{\bw}_{\mathbf{L}^2(\Ga)}\big)
    \end{align}
    for all $(\bv,\bw)\in \mathbfcal{H}^1_{0}$.
\end{lemma}

\begin{proof}
    First, we recall the surface Korn inequality: There exists a constant $C_\Ga> 0$ such that
    \begin{align}
        \label{KORN:S}
        \norm{\tw}_{\mathbf{H}^1(\Ga)} \leq C_\Ga\big(\norm{\Dg\tw}_{\mathbf{L}^2(\Ga)} + \norm{\tw}_{\mathbf{L}^2(\Ga)}\big)
    \end{align}
    for all $\tw\in \mathbf{H}^1(\Ga)$, see, e.g. \cite[Formula~(4.7)]{Jankuhn2018}. Furthermore, the standard Korn inequality (see \cite[Remark~IV.7.3.]{Boyer}) states that there exists a constant $C_\Om > 0$ such that
    \begin{align}
        \label{KORN:B}
        \norm{\tv}_{\mathbf{H}^1(\Om)} \leq C_\Om \norm{\D\tv}_{\mathbf{L}^2(\Om)}
    \end{align}
    for all $\tv\in \mathbf{H}^1(\Om)$ with $\tv\vert_\Ga = \mathbf{0}$ a.e.~on $\Gamma$. 
    
    Let now $(\bv,\bw)\in \mathbfcal{H}^1_{0}$ be arbitrary.
    According to the inverse trace theorem (see \cite[Theorem~4.2.3]{Hsiao2008}), there exists a function $\ww \in H^{3/2}(\Omega)$ such that $\ww = \bw$ a.e.~on $\Gamma$ and there is a constant $C_\mathrm{tr}>0$ independent of $\ww$ and $\bw$ such that
    \begin{align}
        \label{EST:IT}
        \norm{\ww}_{H^{3/2}(\Omega)} \le C_\mathrm{tr}\, \norm{\bw}_{H^{1}(\Gamma)}.
    \end{align}
    In particular, we have $(\bv - \ww)\vert_\Ga = 0$ a.e.~on $\Gamma$. Hence, using Korn's inequality \eqref{KORN:B} and the inverse trace estimate \eqref{EST:IT}, we get
    \begin{align*}
        \norm{\bv - \ww}_{\mathbf{H}^1(\Om)}
        &\le C_\Omega \big( \norm{\D\bv}_{\mathbf{L}^2(\Om)}
            + \norm{\D\ww}_{\mathbf{L}^2(\Om)}
            \big)
        \\
        &\le C_\Omega \big( \norm{\D\bv}_{\mathbf{L}^2(\Om)}
            + C_\mathrm{tr} \norm{\ww}_{\mathbf{H}^1(\Ga)}
            \big).
    \end{align*}
    Finally, using the surface Korn inequality \eqref{KORN:S}, we conclude that
    \begin{align*}
        \norm{\bv}_{\mathbf{H}^1(\Om)}
        &\le \norm{\ww}_{\mathbf{H}^1(\Om)}
            + \norm{\bv - \ww}_{\mathbf{H}^1(\Om)}
        \\
        &\le C_\Omega \norm{\D\bv}_{\mathbf{L}^2(\Om)}
            + (1 + C_\Omega C_\mathrm{tr}) \norm{\ww}_{\mathbf{H}^1(\Ga)}
        \\
        &\le C_K \big( \norm{\D\bv}_{\mathbf{L}^2(\Om)}
            + \norm{\Dg\ww}_{\mathbf{L}^2(\Ga)}
            + \norm{\ww}_{\mathbf{L}^2(\Ga)} \big) 
    \end{align*}
    for some constant $C_K>0$ that is independent of $(\bv,\bw)$. This proves the claim.
\end{proof}

Thus, considering the bilinear form
\begin{align*}
    \mathbfcal{B}:\mathbfcal{H}^1_{0,\Div}\times \mathbfcal{H}^1_{0,\Div}&\rightarrow\R, \\
    ((\bv,\bw),(\wv,\ww))&\mapsto \intO 2\nu_\Om(\phi)\D\bv:\D\wv\dx + \intG 2\nu_\Ga(\psi)\Dg\bw:\Dg\ww\dG \\
    &\qquad\qquad + \intG \gamma(\phi,\psi)\bw\cdot\ww\dG,
\end{align*}
we infer from Lemma~\ref{Lemma:Korn} and the assumptions on the coefficients (see \ref{Assumption:Coefficients}) that $\mathbfcal{B}$ is coercive. Consequently, employing the Lax--Milgram theorem, for every $(\f,\f_\Ga)\in (\mathbfcal{H}^1_0)^\prime$, there exists a unique pair $(\bv,\bw)\in \mathbfcal{H}^1_{0,\Div}$, which fulfills \eqref{Stokes:WF} for all $(\wv,\ww)\in \mathbfcal{H}^1_{0,\Div}$. In particular, we have
\begin{align*}
	\norm{(\bv,\bw)}_{\mathbfcal{H}^1} \leq \frac{C_K}{\min\{2\nu_\ast,\gamma_\ast\}} \norm{(\f,\f_\Ga)}_{(\mathbfcal{H}^1_0)^\prime}.
\end{align*}

Next, to reconstruct the pressure pair $(p,q)$, we need the following lemma, which is an adaptation of \cite[Lemma 3.2]{Lengeler2015}.

\begin{lemma}\label{App:Lemma:Div}
    The operator
    \begin{align}
        \underline{\Div}:\mathbfcal{H}^1_0/\mathbfcal{H}^1_{0,\Div}\rightarrow \mathcal{L}^2_{(0)},
        \quad(\bv,\bw)\mapsto (\Div\,\bv,\Divg\,\bw),
    \end{align}
    is an isomorphism. In particular, there exists a constant $C > 0$ such that
    \begin{align}
        \norm{(\bv,\bw)}_{\mathbfcal{H}^1_0/\mathbfcal{H}^1_{0,\Div}} \leq C\norm{(\Div\,\bv,\Divg\,\bw)}_{\mathcal{L}^2_{(0)}}.
    \end{align}
\end{lemma}

\begin{proof}
    We start by considering the operator
    \begin{align*}
        \widetilde{\underline{\Div}}:\mathbfcal{H}^1_0\rightarrow \mathcal{L}^2_{(0)},
        \quad
        (\bv,\bw) \mapsto (\Div\,\bv,\Divg\,\bw).
    \end{align*}
    This operator is clearly
    well-defined with $\mathrm{ker}\,\widetilde{\underline{\Div}} = \mathbfcal{H}^1_{0,\Div}$. Now, we show that $\widetilde{\underline{\Div}}$ is surjective. To this end, let $(g,g_\Ga)\in \mathcal{L}^2_{(0)}$ be arbitrary. Then, let $\bw\in \mathbf{H}^1(\Ga)$ with $\bw\cdot\n = 0$ be a solution of
    \begin{align*}
        \Divg\,\bw = g_\Ga \quad\text{on~}\Ga.
    \end{align*}
    The existence of such a solution $\bw$ can be shown by
    employing $L^2$-theory for the Laplace--Beltrami operator on $\Ga$, noticing that the right-hand side is mean-value free. 
    Moreover, utilizing $L^2$-theory of the divergence operator in $\Om$, we find $\bv\in \mathbf{H}^1(\Om)$ such that
    \begin{alignat*}{2}
        \Div\,\bv &= g &&\quad\text{in~}\Om, \\
        \bv &= \bw &&\quad\text{on~}\Ga.
    \end{alignat*}
    We have therefore shown the existence of $(\bv,\bw)\in\mathbfcal{H}^1_0$ satisfying
    \begin{align*}
        \widetilde{\underline{\Div}}(\bv,\bw) = (g,g_\Ga),
    \end{align*}
    which proves that $\widetilde{\underline{\Div}}:\mathbfcal{H}^1_0\rightarrow \mathcal{L}^2_{(0)}$ is a surjective homomorphism. Consequently, by the fundamental theorem of homomorphisms, $\underline{\Div}:\mathbfcal{H}^1_0/\mathbfcal{H}^1_{0,\Div}\rightarrow \mathcal{L}^2_{(0)}$ is an isomorphism.
\end{proof}

As a corollary, we obtain the following (see also \cite[Corollary 3.3]{Lengeler2015}).

\begin{corollary}\label{App:Corollary:Grad}
    The operator $\underline{\Grad}:\mathcal{L}^2_{(0)}\mapsto (\mathbfcal{H}^1_{0,\Div})^\perp\subset (\mathbfcal{H}^1_0)^\prime$, defined by
    \begin{align}
        \underline{\Grad}(p,q)(\bv,\bw) = - \intO p\,\Div\,\bv\dx - \intG q\,\Divg\,\bw\dG,
    \end{align}
    for $(\bv,\bw)\in \mathbfcal{H}^1_0$, is an isomorphism. We have
    \begin{align}\label{Est:(p,q):Cor:App}
        \norm{(p,q)}_{\mathcal{L}^2_{(0)}} \leq C \norm{\underline{\Grad}(p,q)}_{(\mathbfcal{H}^1_0)^\prime}
    \end{align}
    with the same constant $C$ as in Lemma~\ref{App:Lemma:Div}.
\end{corollary}

\begin{proof}
    This follows from Lemma~\ref{App:Lemma:Div}, the fact that
    \begin{align*}
        \underline{\Grad} = - \underline{\Div}^\prime:(\mathcal{L}^2_{(0)})^\prime \rightarrow (\mathbfcal{H}^1_0/\mathbfcal{H}^1_{0,\Div})^\prime
    \end{align*}
    is an isomorphism, from the isometry $(\mathbfcal{H}^1_0/\mathbfcal{H}^1_{0,\Div})^\prime\simeq (\mathbfcal{H}^1_{0,\Div})^\perp$, and by Riesz' isometry $\mathcal{L}^2_{(0)}\simeq (\mathcal{L}^2_{(0)})^\prime$. Recall also that forming the adjoint preserves the operator norm.
\end{proof}

Now, Corollary~\ref{App:Corollary:Grad} allows us to reconstruct the pressure pair $(p,q)$. Indeed, consider the unique weak solution $(\bv,\bw)$ that was constructed above by the Lax--Milgram theorem. Then, by Corollary~\ref{App:Corollary:Grad}, we find a unique pair $(p,q)\in \mathcal{L}^2_{(0)}$ such that
\begin{align*}
	\underline{\Grad}(p,q)(\wv,\ww) &= - \intO 2\nu_\Om(\phi)\D\bv:\D\wv\dx - \intG 2\nu_\Ga(\psi)\Dg\bw:\Dg\ww\dG
	- \intG \gamma(\phi,\psi)\bw\cdot\ww\dG \\
	&\qquad + \bigang{(\f,\f_\Ga)}{(\wv,\ww)}_{\mathbfcal{H}^1_0}
\end{align*}
for all $(\wv,\ww)\in \mathbfcal{H}^1_0$, from which we deduce the weak formulation
\begin{align}
    \label{WF:PQ:0}
    \begin{split}
        &\intO 2\nu_\Om(\phi)\D\bv:\D\wv\dx + \intG 2\nu_\Ga(\psi)\Dg\bw:\Dg\ww\dG + \intG \gamma(\phi,\psi)\bw\cdot\ww\dG \\
        &\qquad - \intO p\,\Div\,\wv\dx - \intG q\,\Divg\,\ww\dG \\
        &\quad = \bigang{(\f,\f_\Ga)}{(\wv,\ww)}_{\mathbfcal{H}^1_0}
    \end{split}
\end{align}
for all $(\wv,\ww)\in \mathbfcal{H}^1_0$.

\medskip \pagebreak[2]

We summarize our considerations above in the following theorem.
\begin{theorem} \label{THM:WWP:BSS}
    Let $(\f,\f_\Gamma)\in(\mathbfcal{H}^1_0)^\prime$ and $(g,g_\Ga) = (0,0)$. Then there exists a unique weak solution $(\bv,\bw)\in \mathbfcal{H}^1_0$ of the bulk-surface stokes system \eqref{App:BSS:system} and an associated pressure pair $(p,q)\in \mathcal{L}^2_{(0)}$ such that \eqref{WF:PQ:0} is fulfilled for all $(\wv,\ww)\in \mathbfcal{H}^1_0$. Additionally, there exists a constant $C > 0$, that depends only on the data of the system, such that
    \begin{align}\label{Est:BulkSurface:H^1+pressure}
        \norm{(\bv,\bw)}_{\mathbfcal{H}^1} + \norm{(p,q)}_{\mathcal{L}^2} \leq C\norm{(\f,\f_\Ga)}_{(\mathbfcal{H}^1_0)^\prime}.
    \end{align}
\end{theorem}

Now, we want to study the system \eqref{App:BSS:system} with general divergences. To this end, we consider $(g,g_\Ga)\in \mathcal{L}^2_{(0)}$. Then, by Lemma~\ref{App:Lemma:Div}, we find a unique pair $(\bv_0,\bw_0)\in\mathbfcal{H}^1_0$ such that
\begin{align}
    \underline{\Div}(\bv_0,\bw_0) = (g,g_\Ga).
\end{align}
Next, we define $(\tilde{\f},\tilde{\f}_\Ga)\in (\mathbfcal{H}^1_0)^\prime$ via
\begin{align}
    \begin{split}
        \bigang{(\tilde{\f},\tilde{\f}_\Ga)}{(\wv,\ww)}_{\mathbfcal{H}^1_0} &\coloneqq \bigang{(\f,\f_\Ga)}{(\wv,\ww)}_{\mathbfcal{H}^1_0} - \intO 2\nu_\Om(\phi)\D\bv_0:\D\wv\dx \\
        &\qquad - \intG 2\nu_\Ga(\psi)\Dg\bw_0:\Dg\ww\dG - \intG \gamma(\phi,\psi)\bw_0\cdot\ww\dG
    \end{split}
\end{align}
for $(\wv,\ww)\in \mathbfcal{H}^1_0$. Employing again the bulk-surface Korn inequality from Lemma~\ref{Lemma:Korn}, the Lax--Milgram theorem provides the existence of a unique pair $(\bv_1,\bw_1)\in \mathbfcal{H}^1_{0,\Div}$, which solves
\begin{align}
	\begin{split}
    	&\intO 2\nu_\Om(\phi)\D\bv_1:\D\wv\dx + \intG 2\nu_\Ga(\psi)\Dg\bw_1:\Dg\ww\dG + \intG \gamma(\phi,\psi)\bw_1\cdot\ww\dG \\
    	&\qquad = \bigang{(\tilde{\f},\tilde{\f}_\Ga)}{(\wv,\ww)}_{\mathbfcal{H}^1_0}
    \end{split}
\end{align}
for all $(\wv,\ww)\in \mathbfcal{H}^1_{0,\Div}$. Then, by Corollary~\ref{App:Corollary:Grad}, there exists a unique pair $(p,q)\in \mathcal{L}^2_{(0)}$ such that
\begin{align*}
    \underline{\Grad}(p,q)(\wv,\ww) &= - \intO 2\nu_\Om(\phi)\D\bv_1:\D\wv\dx - \intG 2\nu_\Ga(\psi)\Dg\bw_1:\Dg\ww\dG \\
    &\quad - \intG \gamma(\phi,\psi)\bw_1\cdot\ww\dG + \bigang{(\tilde{\f},\tilde{\f}_\Ga)}{(\wv,\ww)}_{\mathbfcal{H}^1_0}
\end{align*}
for all $(\wv,\ww)\in \mathbfcal{H}^1_0$. Now, we define $(\bv,\bw) \coloneqq (\bv_0 + \bv_1, \bw_0 + \bw_1)$. Then
\begin{alignat*}{2}
    \Div\,\bv &= \Div\,\bv_0 + \Div\,\bv_1 = g &&\quad\text{in~}\Om, \\
    \Divg\,\bw &= \Divg\,\bw_0 + \Divg\,\bw_1 = g_\Ga &&\quad\text{on~}\Ga,
\end{alignat*}
and, by construction of $(p,q)$ and the definition of $(\tilde{\f},\tilde{\f}_\Ga)$, we readily notice that $(\bv,\bw)\in \mathbfcal{H}^1_0$ solves
\begin{align}
    \label{WF:PQ}
    \begin{split}
    &\intO \D\bv:\D\wv\dx + \intG \Dg\bw:\Dg\ww\dG + \intG \gamma(\phi,\psi)\bw\cdot\ww\dG 
    \\
    &\qquad - \intO p\,\Div\,\wv\dx - \intG q\,\Divg\,\ww\dG 
    \\
    &\quad = \bigang{(\f,\f_\Ga)}{(\wv,\ww)}_{\mathbfcal{H}^1_0}
    \end{split}
\end{align}
for all $(\wv,\ww)\in \mathbfcal{H}^1_0$, together with
\begin{align*}
    \Div\,\bv = g \ \text{in~}\Om, \qquad \Divg\,\bw = g_\Ga \ \text{on~}\Ga.
\end{align*}

\medskip

This means that the following result is established.

\begin{theorem}
    Let $(\f,\f_\Ga)\in (\mathbfcal{H}^1_0)^\prime$ and $(g,g_\Ga)\in\mathcal{L}^2_{(0)}$. Then there exists a unique weak solution $(\bv,\bw)\in \mathbfcal{H}^1_0$ of the bulk-surface stokes system \eqref{App:BSS:system} and an associated pressure pair $(p,q)\in \mathcal{L}^2_{(0)}$ such that \eqref{WF:PQ} is fulfilled for all $(\wv,\ww)\in \mathbfcal{H}^1_0$.
\end{theorem}

\subsection{Existence and uniqueness of strong solutions}
In the next step, assuming that the coefficients $\nu_\Om,\nu_\Ga$ and $\gamma$ are positive constants, we prove that the system \eqref{App:BSS:system} admits a unique strong solution. In the following, we discuss only the case of incompressible fluids, as it is the more relevant case for our subsequent analysis of system~\eqref{System}. The proof of the existence of a unique strong solution relies on a fixed-point argument. To this end, we start by considering the surface Stokes equation and the corresponding surface Stokes operator.

In this subsection, let $r\in[2,\infty)$ be arbitrary. We consider the surface Helmholtz projection 
\begin{align}
    \label{App:Def:Helmholtz:Surface}
    \mathbf{P}_\Div^\Ga: \mathbf{L}^r_{\tau}(\Gamma) \to \mathbf{L}^r_{\Div}(\Gamma),
    \quad
    \mathbf{P}_\Div^\Ga(\bw) = \bw - \Gradg q, 
\end{align}
where $q \in \dot{W}^{1,r}(\Ga)$ is (up to an additive constant) uniquely determined by
\begin{align}
    \label{App:Helmholtz:WF}
    \intG \Gradg q\cdot\Gradg \xi \dG = \intG \bw\cdot\Gradg\xi\dG \qquad\text{for all~}\xi\in \dot{W}^{1,r^\prime}(\Ga),
\end{align}
cf.~\cite[Lemma A.1]{Simonett2022}.
Here, $r'$ denotes the dual Sobolev exponent of $r$, i.e., $r^\prime = \frac{r}{r-1}$.
The existence of this projection is also related to the surface Helmholtz decomposition, which can be found, for example, in \cite[Theorem~4.2]{Reusken2020}.
According to \cite[Sect.~2]{Simonett2022}, we further have
\begin{align}
    \intG \mathbf{P}_\Div^\Ga(\bw) \cdot \bu\dG = \intG \bw\cdot \mathbf{P}_\Div^\Ga(\bu)\dG
\end{align}
for all $\bw\in \mathbf{L}^r_{\tau}(\Ga)$ and $\bu\in \mathbf{L}^{r^\prime}_\tau(\Ga)$. Therefore, we define the surface Stokes operator by
\begin{align}
    \mathbf{A}^\Ga(\bw) = - \mathbf{P}_\Div^\Ga\Divgt(2\nu_\Ga\Dg\bw) \qquad\text{for all~}\bw\in D(\mathbf{A}^\Ga) = \mathbf{W}^{2,r}_\tau(\Ga)\cap \mathbf{L}^r_\Div(\Ga).
\end{align}
We further refer to \cite{Pruess2021, Simonett2022} for the surface Stokes operator on more general surfaces, to \cite{Abels2024a} for the surface Stokes operator on evolving surfaces, and to the survey article \cite{Hieber2018} for the Stokes operator in various other geometric settings.

\pagebreak[3]

The following theorem ensures the existence of a unique strong solution to the surface Stokes problem.

\begin{theorem}
    \label{App:Theorem:SurfaceStokes}
    Let $\f_\Ga\in \mathbf{L}^r_\tau(\Ga)$. Then there exists a unique strong solution $(\bw,q)$ with $\bw\in \mathbf{W}^{2,r}_\tau(\Ga)\cap \mathbf{L}^r_{\Div}(\Ga)$ and $q\in W^{1,r}_{(0)}(\Gamma)$ of the surface Stokes problem
    \begin{alignat}{2}\label{App:SurfaceStokes}
        -\Divgt(2\nu_\Ga\Dg \bw) + \Gradg q + \gamma\bw &= \f_\Ga &&\qquad\text{on~}\Ga, \\
        \Divg\,\bw &= 0 &&\qquad\text{on~}\Ga.
    \end{alignat}
    Moreover, there exists a constant $C>0$, independent of $\f_\Ga$, such that
    \begin{align}\label{App:Est:SurfaceStokes}
        \norm{\bw}_{\mathbf{W}^{2,r}(\Ga)} + \norm{q}_{W^{1,r}(\Ga)} \leq C \norm{\f_\Ga}_{\mathbf{L}^r(\Ga)}.
    \end{align}
\end{theorem}

\begin{proof}
    The definition of the projection $\mathbf{P}_\Div^\Ga$ implies that there exists a function $q_1\in \dot{W}^{1,r}(\Gamma)$ such that 
    \begin{align}
        \label{SurfaceStokes:q1}
        \f_\Ga = \mathbf{P}_\Div^\Ga \f_\Ga + \Gradg q_1
        \quad\text{a.e.~on $\Gamma$.}
    \end{align}    
    We further know from \cite[Corollary 3.4]{Pruess2021} that the operator $\mathbf{A}^\Ga$ has maximal $L^p$-regularity. Hence, $\{z\in\mathbb{C}: \mathrm{Re}\,z < 0\}\subset \rho(\mathbf{A}^\Ga)$,
    and there exists a constant $M\geq 1$ such that for all $z\in\mathbb C$ with $\mathrm{Re}\, z > 0$, $z + \mathbf{A}^\Ga$ is invertible with
    \begin{align}\label{Est:Resolvent}
        \norm{(z + \mathbf{A}^\Ga)^{-1}}_{\mathrm{op}} \leq \frac{M}{1 + \abs{z}}, 
    \end{align}
    see, e.g., \cite[Proposition 3.5.2]{Pruess2016}. Consequently, we find $\bw\in D(\mathbf{A}^\Ga)$ such that
    \begin{align}
        \label{SurfaceStokes:w}
        \gamma\bw + \mathbf{A}^\Ga(\bw) = \mathbf{P}_\Div^\Ga \f_\Ga \qquad\text{a.e.~on $\Ga$}.
    \end{align}
    In view of \eqref{Est:Resolvent}, it holds that
    \begin{align*}
        \norm{\bw}_{\mathbf{W}^{2,r}(\Ga)} \leq \frac{M}{1 + \gamma}\norm{\f_\Ga}_{\mathbf{L}^r(\Ga)}.
    \end{align*}
    As $\bw\in \mathbf{W}^{2,r}_\tau(\Ga)\cap\mathbf{L}^r_\Div(\Ga)$, we further have
    \begin{align*}
        - \Divgt(2\nu_\Ga\Dg\bw), \; \mathbf{A}^\Ga(\bw) \in \mathbf{L}^r(\Ga)
        \qquad\text{and}\qquad
        \Divg\,\bw = 0 \quad\text{a.e.~on $\Ga$}.
    \end{align*}
    Hence, by the definition of the projection $\mathbf{P}_\Div^\Ga$, we find a function $q_2 \in \dot{W}^{1,r}(\Gamma)$ such that 
    \begin{align}
        \label{SurfaceStokes:q2}
        - \Divgt(2\nu_\Ga\Dg\bw) = \mathbf{A}^\Ga(\bw) - \Gradg q_2
        \quad\text{a.e.~on $\Gamma$.}
    \end{align}
    Defining $q\coloneqq q_1+q_2$ and combining \eqref{SurfaceStokes:q1}, \eqref{SurfaceStokes:w} and \eqref{SurfaceStokes:q2}, we conclude that 
    \begin{align}\label{Pressure:q:constr}
         - \Divgt(2\nu_\Ga\Dg\bw) + \gamma\bw + \Gradg q 
        = \mathbf{A}^\Ga(\bw) + \gamma\bw + \Gradg q_1
        = \f_\Ga
        \quad\text{a.e.~on $\Gamma$.}
    \end{align}
    Moreover, the choice of $q$ is unique if we additionally demand $\meang{q}=0$. Then, the Poincar\'{e}--Wirtinger inequality yields $q\in W^{1,r}(\Ga)$. In this way, $(\bw,q)$ is the unique strong solution of \eqref{App:SurfaceStokes} with the desired regularities. 
\end{proof}

\medskip

In order to analyze the bulk-surface Stokes system \eqref{App:BSS:system}, we also need the Helmholtz projection in the bulk. It is defined as 
\begin{align}
    \label{App:Def:Helmholtz:Bulk}
    \mathbf{P}_\Div^\Om:\mathbf{L}^r(\Om)\rightarrow\mathbf{L}^r_\Div(\Om),
    \quad
    \mathbf{P}_\Div^\Om(\bv) = \bv - \Grad p,
\end{align}
where $p\in \dot{W}^{1,r}(\Om)$ is (up to an additive constant) uniquely determined by
\begin{align*}
	\intO \Grad p \cdot \Grad\zeta\dx = \intO \bv\cdot\Grad\zeta\dx
    \quad\text{for all $\zeta\in \dot{W}^{1,r'}(\Om)$},
\end{align*}
cf.~\cite[Theorem~1.4]{Simader1992}. Here, $r'$ denotes again the dual Sobolev exponent of $r$.

\pagebreak[3]

We are now in a position to prove the existence of a strong solution to \eqref{App:BSS:system}.
\begin{theorem}\label{App:Theorem:BSStokes:Reg}
    Let $\f \in\mathbf{L}^r(\Omega)$ and $\f_\Ga \in\mathbf{L}^r_\tau(\Gamma)$. Then there exists a unique solution $(\bv,\bw)\in\mathbfcal{W}^{2,r}_{0,\Div}$ and $(p,q)\in\mathcal{W}^{1,r}\cap\mathcal{L}^r_{(0)}$ to the bulk-surface Stokes system \eqref{App:BSS:system}. Moreover, there exists a constant $C > 0$ such that
    \begin{align}\label{Est:Stokes:Full}
        \norm{(\bv,\bw)}_{\mathbfcal{W}^{2,r}} + \norm{(p,q)}_{\mathcal{W}^{1,r}} \leq C\norm{(\f,\f_\Ga)}_{\mathbfcal{L}^r}.
    \end{align}
\end{theorem}

\begin{proof}
Let $\tw \in \mathbf{W}^{2-1/r,r}_\Div(\Ga)$ be arbitrary.
Then, according to \cite[Proposition 2.2 and Proposition 2.3]{Temam1984}, there exists a unique strong solution $\bv \in \mathbf{W}^{2,r}_\Div(\Om)$ and $p\in W^{1,r}(\Om)\cap L^r_{(0)}(\Om)$ of the bulk Stokes problem
\begin{subequations}
\label{App:BulkStokes}
\begin{alignat}{2}
    -\Div(2\nu_\Om\D\bv) + \Grad p &= \f &&\qquad\text{in~}\Om, \\
    \Div\,\bv &= 0 &&\qquad\text{in~}\Om, \\
    \bv\vert_{\Ga} &= \tw &&\qquad\text{on~}\Ga.
\end{alignat}
\end{subequations}
Furthermore, \cite[Proposition 2.2]{Temam1984} provides the estimate
\begin{align}
    \label{App:Est:BulkStokes}
    \norm{\bv}_{\mathbf{W}^{2,r}(\Om)} + \norm{p}_{W^{1,r}(\Om)} 
    \leq C\big( \norm{\f}_{\mathbf{L}^r(\Om)} + \norm{\tw}_{\mathbf{W}^{2-1/r,r}(\Ga)}\big).
\end{align}
By the trace theorem, we have $\D\bv \in L^r(\Ga;\R^{d\times d})$ and thus,
\begin{align*}
    \f_\Ga - 2\nu_\Om\big[ \D\bv\,\n \big]_\tau \in \mathbf{L}^r_\tau(\Ga).
\end{align*}
Hence, according to Theorem~\ref{App:Theorem:SurfaceStokes}, there exists a unique strong solution $\bw \in \mathbf{W}^{2,r}_\Div(\Ga)$ and $q\in W^{1,r}(\Ga)$ of the surface Stokes problem 
\begin{subequations}
\begin{alignat}{2}
    -\Divgt(2\nu_\Ga\Dg\bw) + \Gradg q + \gamma\bw &= \f_\Ga - 2\nu_\Om\big[ \D\bv\,\n \big]_\tau &&\qquad\text{on~}\Ga, \\
    \Divg\,\bw &= 0 &&\qquad\text{on~}\Ga.
\end{alignat}
\end{subequations}
This defines an operator $T:\mathbf{W}^{2-1/r,r}_\Div(\Ga)\rightarrow \mathbf{W}^{2,r}_\Div(\Ga)$, $\tw\mapsto T\tw = \bw$. 
Using the estimates \eqref{App:Est:SurfaceStokes} and \eqref{App:Est:BulkStokes}, we deduce that
\begin{align}
    \begin{split}
        \norm{T\tw_1 - T\tw_2}_{\mathbf{W}^{2,r}(\Ga)} &\leq C\norm{\D\bv_1\,\n - \D\bv_2\,\n}_{\mathbf{L}^r(\Ga)} \leq C\norm{\bv_1 - \bv_2}_{\mathbf{W}^{2,r}(\Om)} \\
        &\leq C\norm{\tw_1 - \tw_2}_{\mathbf{W}^{2-1/r,r}(\Ga)}.
    \end{split}
\end{align}
This means that $T$ is continuous.
Furthermore, since $\mathbf{W}^{2,r}(\Ga)$ is compactly embedded in $\mathbf{W}^{2-1/r,r}(\Ga)$, Schauder's fixed-point theorem implies that $T$ has a fixed-point $\bw_\ast\in \mathbf{W}^{2,r}_\Div(\Ga)$ such that $T\bw_\ast = \bw_\ast$.
Finally, we write $q_*$ to denote the pressure associated with $\bw_\ast$, and we define $(\bv_\ast,p_\ast)$ as the corresponding unique strong solution to the bulk Stokes problem \eqref{App:BulkStokes}. In this way, we have constructed functions
\begin{align*}
    (\bv_\ast,\bw_\ast)\in\mathbfcal{W}^{2,r}_{0,\Div},
    \quad
    (p_\ast,q_\ast)\in\mathcal{W}^{1,r}\cap\mathcal{L}^r_{(0)},
\end{align*}
which are a strong solution to the bulk-surface Stokes system \eqref{App:BSS:system}. As every strong solution is also a weak solution, we conclude from Theorem~\ref{THM:WWP:BSS} that the strong solution is unique. As a last step, we have to verify the inequality \eqref{Est:Stokes:Full}. 
To this end, we recall the trace embedding $W^{3/4,r}(\Omega) \emb L^r(\Gamma)$ (see, e.g., \cite[Part~I, Chapter~2, Theorem~2.24]{Brezzi1987}), which holds since $\tfrac 34 > \tfrac12 \ge \tfrac 1r$. 
Using this embedding along with Sobolev embeddings, Ehrling's lemma and estimate \eqref{Est:BulkSurface:H^1+pressure}, we obtain 
\begin{align}
    \label{EST:TREHR}
    \begin{split}
    \norm{\D\bv_\ast}_{\mathbf{L}^r(\Ga)} 
    & \leq C\, \norm{\D\bv_\ast}_{\mathbf{W}^{3/4,r}(\Om)} 
    \leq C\, \norm{\bv_\ast}_{\mathbf{W}^{7/4,r}(\Om)} 
    \\
    &\leq
    \varepsilon\norm{\bv_\ast}_{\mathbf{W}^{2,r}(\Om)} 
    + C_\varepsilon \norm{\bv_\ast}_{\mathbf{H}^{1}(\Om)}
    \\
    &\leq \varepsilon\norm{\bv_\ast}_{\mathbf{W}^{2,r}(\Om)} 
    + C_\varepsilon\, C\, \norm{(\f,\f_\Ga)}_{\mathbfcal{L}^r}
    \end{split}
\end{align}
for any $\varepsilon>0$ and some constant $C_\varepsilon > 0$ that may depend on $\varepsilon$ and the same quantities as constants denoted by $C$. Fixing a sufficiently small $\varepsilon>0$, 
the desired inequality  \eqref{Est:Stokes:Full} follows by combining \eqref{EST:TREHR} with the estimates \eqref{App:Est:SurfaceStokes} and \eqref{App:Est:BulkStokes}. This completes the proof.
\end{proof}

\subsection{The bulk-surface Stokes operator: Regularity and spectral theory}

We further consider the situation where $\nu_\Om,\nu_\Ga$ and $\gamma$ are positive constants.
Using the bulk and surface Helmholtz projections, we define the bulk-surface Stokes operator as follows:
\begin{align*}
    \mathbfcal{A}:D(\mathbfcal{A})\subset \mathbfcal{L}^2_\Div&\rightarrow\mathbfcal{L}^2_\Div, \\
    (\bv,\bw)&\mapsto \big(-\mathbf{P}_\Div^\Om(\Div(2\nu_\Om\D\bv)), -\mathbf{P}_\Div^\Ga(\Divgt(2\nu_\Ga\Dg\bw) + 2\nu_\Om\big[ \D\bv\,\n \big]_\tau + \gamma\bw)\big),
\end{align*}
where $D(\mathbfcal{A}) = \mathbfcal{H}^2_{0,\Div}$.

\pagebreak[3]

\begin{lemma} \label{LEM:PROP:BSS}
    The operator $\mathbfcal{A}$ is invertible and its inverse 
    \begin{align*}
        \mathbfcal{S}: \mathbfcal{L}^2_\Div \to \mathbfcal{L}^2_\Div,
        \quad
        \mathbfcal{S}(\f,\f_\Ga) = \mathbfcal{A}^{-1}(\f,\f_\Ga), 
    \end{align*}
    which maps $(\f,\f_\Ga)$ onto the corresponding strong solution of the bulk-surface Stokes problem \eqref{App:BSS:system}, is a linear and bounded operator that is injective, compact, and self-adjoint with respect to the inner product of $\mathbfcal{L}^2$.
\end{lemma}

\begin{proof}
    The invertibility of $\mathbfcal{A}$ and the existence of the inverse Stokes operator $\mathbfcal{S}$ are a direct consequence of Theorem~\ref{App:Theorem:BSStokes:Reg}. As the Stokes operator $\mathbfcal{A}$ is linear and invertible, it is clear that $\mathbfcal{S}$ is linear and injective.
    Combining the estimate \eqref{Est:Stokes:Full} with the trace embedding $\mathbf{H}^2(\Omega) \emb \mathbf{L}^2(\Gamma)$, we deduce that $\mathbfcal{S}$ is bounded. Moreover, as $\mathbfcal{S}(\mathbfcal{L}^2_\Div) = D(\mathbfcal{A}) = \mathbfcal{H}^2_{0,\Div}$, we conclude from the compact embedding $\mathbfcal{H}^2 \emb \mathbfcal{L}^2$ that $\mathbfcal{S}$ is compact. It remains to prove that $\mathbfcal{S}$ is self-adjoint. To this end, let $(\f,\f_\Ga), (\widehat\f,\widehat\f_\Ga) \in \mathbfcal{L}^2_\Div$ be arbitrary. Choosing $(\bv,\bw) = \mathbfcal{S}(\f,\f_\Ga) \in \mathbfcal{H}^2_{0,\Div}$ and $(\wv,\ww) = \mathbfcal{S}(\widehat\f,\widehat\f_\Ga) \in \mathbfcal{H}^2_{0,\Div}$ in the weak formulation \eqref{WF:PQ:0}, we finally conclude that
    \begin{align}
    \label{EQ:SAD}
        \begin{split}
        &\bigscp{\mathbfcal{S}(\f,\f_\Ga)}{(\widehat\f,\widehat\f_\Ga)}_{\mathbfcal{L}^2}
        = \bigang{(\widehat\f,\widehat\f_\Ga)}{(\bv,\bw)}_{\mathbfcal{H}^1_0}
        \\
        &\quad = \intO 2\nu_\Om\,\D\bv:\D\wv\dx 
            + \intG 2\nu_\Ga\,\Dg\bw:\Dg\ww\dG 
            + \intG \gamma\,\bw\cdot\ww\dG 
        \\
        &\quad = \bigang{(\f,\f_\Ga)}{(\wv,\ww)}_{\mathbfcal{H}^1_0}
        = \bigscp{(\f,\f_\Ga)}{\mathbfcal{S}(\widehat\f,\widehat\f_\Ga)}_{\mathbfcal{L}^2}.
        \end{split}
    \end{align}
    This proves that $\mathbfcal{S}$ is self-adjoint with respect to the inner product of $\mathbfcal{L}^2$, and thus, the proof is complete.
\end{proof}

\begin{theorem} \label{THM:SPEC:BSS}
    The operator $\mathbfcal{A}$ has countably many eigenvalues, and each of them has a finite-dimensional eigenspace. Repeating each eigenvalue according to its multiplicity, we can interpret them as a sequence $\{\lambda_k\}_{k\in\N} \subset \R$ with
    \begin{align*}
        0 < \lambda_1 \le \lambda_2 \le \lambda_3 \le \ldots
        \qquad\text{and}\qquad
        \lambda_k \to \infty \;\;\text{as $k\to\infty$}.
    \end{align*}
    Moreover, there exists an orthonormal basis $\{(\tv_k,\tw_k) \}_{k\in\N} \subset \mathbfcal{H}^2_{0,\Div}$ of $\mathbfcal{L}^2_\Div$ such that for every $k\in\N$, $(\tv_k,\tw_k)$ is an eigenfunction to the eigenvalue $\lambda_k$. This means that $(\tv_k,\tw_k)$ is non-trivial and satisfies
    \begin{align*}
        \mathbfcal{A}(\tv_k,\tw_k) = \lambda_k (\tv_k,\tw_k)
        \quad\Leftrightarrow\quad
        (\tv_k,\tw_k) = \lambda_k \mathbfcal{S}(\tv_k,\tw_k).
    \end{align*}
    In particular, any pair $(\bv,\bw)\in\mathbfcal{L}^2_\Div$ can be expressed as
	\begin{align*}
		(\bv,\bw) = \sum_{k=1}^\infty \bigscp{(\bv,\bw)}{(\tv_k,\tw_k)}_{\mathbfcal{L}^2}(\tv_k,\tw_k)
	\end{align*}
    in the sense that the series converges in $\mathbfcal{L}^2$.
\end{theorem}

\begin{proof}
    Due to the properties established in Lemma~\ref{LEM:PROP:BSS}, the assertions directly follow by applying the spectral theorem for compact normal operators to the solution operator $\mathbfcal{S}$ (see, e.g., \cite[Sect.~2.12]{Alt2016} and also \cite[Secs.~D.5--D.6]{Evans2010}).
\end{proof}

\medskip

\begin{corollary}
    \label{Cor:Proj:Stokes}
    The bulk-surface Helmholtz projection
    \begin{align*}
        \mathbfcal{P}_\Div: \mathbfcal{L}^2 \to \mathbfcal{L}^2_\Div,
        \quad
        \mathbfcal{P}_\Div(\bv,\bw) 
        = \big( \mathbf{P}^\Omega_\Div(\bv) , \mathbf{P}^\Gamma_\Div(\bw) \big)
    \end{align*}
    can be represented as
    \begin{align*}
        \mathbfcal{P}_\Div(\bv,\bw) = \sum_{k=1}^\infty \bigscp{(\bv,\bw)}{(\tv_k,\tw_k)}_{\mathbfcal{L}^2}(\tv_k,\tw_k),
    \end{align*}
    for all $(\bv,\bw) \in \mathbfcal{L}^2$, in the sense that the series converges in $\mathbfcal{L}^2$.
\end{corollary}

\begin{proof}
    As $\mathbfcal{P}_\Div(\bv,\bw) \in \mathbfcal{L}^2_\Div$, the claim directly follows from the fact that, according to Theorem~\ref{THM:SPEC:BSS}, the eigenfunctions $\{ (\tv_k,\tw_k) \}_{k\in\N}$ form an orthonormal basis of $\mathbfcal{L}^2_\Div$.
\end{proof}

\medskip

\begin{corollary} \label{COR:ONS}
    The eigenfunctions $\{(\tv_k,\tw_k)\}_{k\in\N}\subset\mathbfcal{H}^2_{0,\Div}$ form a complete orthogonal system of $D(\mathbfcal{A}) = \mathbfcal{H}^2_{0,\Div}$.
\end{corollary}

\begin{proof}
    First note that $\mathbfcal{A}$ is self-adjoint, which can be shown by performing similar computations as for $\mathbfcal{S}$. The claim then follows as in the proof of \cite[Theorem~II.6.6]{Boyer}.
\end{proof}

\medskip

\begin{corollary} \label{COR:DENS}
	The inclusion $\mathbfcal{W}^{2,r}_{0,\Div}\emb\mathbfcal{L}^2_\Div$ is dense for every $r\in(2,\infty)$.
\end{corollary}

\begin{proof}
    Let $(\bv,\bw) \in \mathbfcal{L}^2_\Div$ be arbitrary.
    Since $\mathbfcal{H}^2_{0,\Div} \emb \mathbfcal{L}^r_\Div$, Theorem~\ref{App:Theorem:BSStokes:Reg} implies that
    \begin{align*}
        (\tv_k,\tw_k) = \lambda_k \mathbfcal{S}(\tv_k,\tw_k)
        \in \mathbfcal{W}^{2,r}_{0,\Div}
    \end{align*}
    for all $k\in\N$. Hence, for $N\in\N$, we have
    \begin{align*}
        \mathbfcal{W}^{2,r}_{0,\Div} \ni
        \sum_{k=1}^{N} \bigscp{(\bv,\bw)}{(\tv_k,\tw_k)}_{\mathbfcal{L}^2}(\tv_k,\tw_k)
        \to (\bv,\bw) 
        \quad\text{in $\mathbfcal{L}^2$}
    \end{align*}
    as $N\to \infty$, which proves the claim.
\end{proof}

\section{Construction of a weak solution for \texorpdfstring{$L > 0$}{L>0}}
\label{Section:Proof:L>0}

In this section, we prove the following theorem, which ensures the existence of weak solutions to \eqref{eqs:NSCH} in the sense of Definition~\ref{Definition:WeakSolution} in the case $L\in(0,\infty]$.

\begin{theorem}\label{Theorem:L>0}
Suppose that the assumptions \ref{Assumption:Domain}-\ref{Assumption:Potential} hold. Let $K\in[0,\infty]$, $L\in (0,\infty]$, 
let $(\bv_0,\bw_0)\in\mathbfcal{L}^2_\Div$, and let $(\phi_0,\psi_0)\in\mathcal{H}^1_{K,\alpha}$ satisfy \eqref{Assumption:InitialCondition}. 
Then there exists at least one weak solution $(\bv,\bw,\phi,\psi,\mu,\theta)$ to system \eqref{eqs:NSCH} in the sense of Definition~\ref{Definition:WeakSolution}.
\end{theorem}

\begin{proof}
The proof is based on a semi-Galerkin discretization of the Navier--Stokes equations combined with Schauder's fixed-point theorem. To this end, we start by introducing an approximate system. In its construction, the analysis of the previous chapter on the bulk-surface Stokes equations will play a crucial role.

In this proof, the letter $C$ will denote generic positive constants, which may depend only on $\Omega$, $T$, the initial data, and the constants introduced in \ref{Assumption:Domain}-\ref{Assumption:Potential}. The exact value of $C$ may vary throughout the proof.

\subsection{Construction of an approximate system}
Let us consider the family of eigenfunctions $\{(\tv_j,\tw_j)\}_{j\in\N}$ of the bulk-surface Stokes operator $\mathbfcal{A}$ that was obtained in Theorem~\ref{THM:SPEC:BSS}. For any $m\in\N$, we introduce the finite-dimensional subspace $\mathbfcal{V}_m \coloneqq \textup{span}\{\scp{\tv_1}{\tw_1},\ldots,\scp{\tv_m}{\tw_m}\}$ of $\mathbfcal{L}^2_\Div$. These finite-dimensional spaces $\mathbfcal{V}_m$ are endowed with the norm of $\mathbfcal{L}^2$. The orthogonal projection of $\mathbfcal{L}^2$ onto $\mathbfcal{V}_m$ with respect to the inner product on $\mathbfcal{L}^2_\Div$ is denoted by $\mathbfcal{P}_{\mathbfcal{V}_m} = (\mathbf{P}_{\mathbfcal{V}_m}^\Om,\mathbf{P}_{\mathbfcal{V}_m}^\Ga)$. Recalling that $\Om$ is of class $C^4$, the regularity theory of the bulk-surface Stokes operator $\mathbfcal{A}$ (see Theorem~\ref{App:Theorem:BSStokes:Reg}) yields that $(\tv_j,\tw_j)\in \mathbfcal{W}^{2,r}_{0,\Div}$ for all $r\in[2,\infty)$ and $j\in\N$. Moreover, invoking the estimate \eqref{Est:Stokes:Full}, it is straightforward to check that the following inverse Sobolev embedding inequalities hold for all $m\in\N$, $r\in[2,\infty)$ and $(\wv,\ww)\in \mathbfcal{V}_m$:
\begin{align}
    \label{Est:InverseSobolev}
    \begin{split}
        \norm{(\wv,\ww)}_{\mathbfcal{H}^1} &\leq C_m \norm{(\wv,\ww)}_{\mathbfcal{L}^2}, \quad \norm{(\wv,\ww)}_{\mathbfcal{H}^2} \leq C_m \norm{(\wv,\ww)}_{\mathbfcal{L}^2}, \\
        \norm{(\wv,\ww)}_{\mathbfcal{W}^{2,r}} &\leq C_{m,r}\norm{(\wv,\ww)}_{\mathbfcal{L}^2}.
    \end{split}
\end{align}
Here, $C_m$ and $C_{m,r}$ are positive constants that may depend on the same quantities as $C$ and additionally on $m$ and $(m,r)$, respectively.
Now, we claim that, for any $m\in\N$, there exists an approximate solution $(\bv_m,\bw_m,\phi_m,\psi_m,\mu_m,\theta_m)$ to system \eqref{eqs:NSCH} having the regularities
\begin{align}
    \label{Reg:ApproxSolution}
    \left\{\;
    \begin{aligned}
    (\bv_m,\bw_m)&\in H^1(0,T;\mathbfcal{V}_m), \\
    (\phi_m,\psi_m)&\in H^1(0,T;(\mathcal{H}^1)^\prime)\cap L^\infty(0,T;\mathcal{H}^1_{K,\alpha})\cap L^2(0,T;\mathcal{W}^{2,6}),\\
    (\mu_m,\theta_m)&\in L^2(0,T;\mathcal{H}^1), \\
    (F^\prime(\phi_m), G^\prime(\psi_m))&\in L^2(0,T;\mathcal{L}^6),
    \end{aligned}
    \right.
\end{align}
with $\abs{\phi_m} < 1$ a.e.~in $Q$ and $\abs{\psi_m} < 1$ a.e.~on $\Sigma$, in the sense that
\begin{align}
    &\intO\rho(\phi_m)\delt\bv_m\cdot\wv\dx + \intG\sigma(\psi_m)\delt\bw_m\cdot\ww\dG \nonumber \\
    &\qquad + \intO\rho(\phi_m)(\bv_m\cdot\Grad)\bv_m\cdot\wv\dx + \intG\sigma(\psi_m)(\bw_m\cdot\Gradg)\bw_m\cdot\ww\dG \nonumber \\
    &\qquad + \intO\rho^\prime\,(\Grad\mu_m\cdot\Grad)\bv_m\cdot\wv\dx + \intG\sigma^\prime\,(\Gradg\theta_m\cdot\Gradg)\bw_m\cdot\ww\dG \label{WF:VW:DISCR} \\ 
    &\qquad  + \intO2\nu_\Om(\phi_m)\D\bv_m:\D\wv\dx + \intG2\nu_\Ga(\psi_m)\Dg\bw_m:\Dg\ww\dG + \intG \gamma(\phi_m,\psi_m)\bw_m\cdot\ww\dG \nonumber \\
    &\quad = \intO\mu_m\Grad\phi_m\cdot\wv\dx + \intG\theta_m\Gradg\psi_m\cdot\ww\dG \nonumber \\
    &\qquad + \frac{\chi(L)}{2}\intG(\beta\sigma^\prime -\rho^\prime)(\beta\theta_m - \mu_m)\bw_m\cdot\ww\dG \nonumber
\end{align}
a.e.~on $[0,T]$ for all $(\wv,\ww)\in\mathbfcal{V}_m$, 
\begin{align}\label{WF:PP:DISCR}
    \begin{split}
        &\bigang{(\delt\phi_m,\delt\psi_m)}{(\zeta,\xi)}_{\mathcal{H}^1} - \intO \phi_m\bv_m\cdot\Grad\zeta\dx - \intG 									\psi_m\bw_m\cdot\Gradg\xi\dG \\
        &\quad = - \intO \Grad\mu_m\cdot\Grad\zeta\dx - \intG \Gradg\theta_m\cdot\Gradg\xi \dG - \chi(L)\intG (\beta\theta_m - \mu_m)				(\beta\xi - \zeta)\dG
    \end{split}
\end{align}
a.e.~on $[0,T]$ for all $(\zeta,\xi)\in\mathcal{H}^1$, 
\begin{subequations}\label{WF:MT:DISCR:STR}
    \begin{alignat}{2}
        &\mu_m = -\Lap\phi_m + F^\prime(\phi_m) &&\qquad\text{a.e.~in~}Q, \label{WF:MT:DEF:MU}\\
        &\theta_m = -\Lapg\psi_m + G^\prime(\psi_m) + \alpha\deln\phi_m &&\qquad\text{a.e.~on~}\Sigma, \label{WF:MT:DEF:THETA}\\
        &\begin{cases}
        K\deln\phi_m = \alpha\psi_m - \phi_m, &\text{if~}K\in[0,\infty), \\
        \deln\phi_m = 0, &\text{if~}K = \infty,
        \end{cases} &&\qquad\text{a.e.~on~}\Sigma,
    \end{alignat}
\end{subequations}
and
\begin{align}\label{Approx:Problem:IC}
    (\bv_m,\bw_m)\vert_{t=0} = \mathbfcal{P}_{\mathbfcal{V}_m}(\bv_0,\bw_0), \quad (\phi_m,\psi_m)\vert_{t=0} = (\phi_0, \psi_0)  \qquad \text{in~}\Om\times\Ga. 
\end{align}
In particular, because of \eqref{WF:MT:DISCR:STR}, the pair $(\mu_m,\theta_m)$ satisfies the weak formulation
\begin{align}
    &\intO\mu_m\eta\dx + \intG \theta_m\vartheta\dG \nonumber \\
    &\quad = \intO \Grad\phi_m\cdot\Grad\eta + F^\prime(\phi_m)\eta\dx + \intG \Gradg\psi_m\cdot\Gradg\vartheta + G^\prime(\psi_m)\vartheta\dG \label{WF:MT:DISCR} \\
    &\qquad + \chi(K) \intG (\alpha\psi_m - \phi_m)(\alpha\vartheta - \eta)\dG \nonumber
\end{align}
a.e.~on $[0,T]$ for all $(\eta,\vartheta)\in\mathcal{H}^1_{K,\alpha}$. The existence of such an approximate solution will be established in the next subsection.

\subsection{Existence of approximate solutions} \label{SUBSEC:EXAPP}
Let $m\in\N$ be arbitrary.
We employ a fixed-point argument to prove the existence of an approximate solution $(\bv_m,\bw_m,\phi_m,\psi_m,\mu_m,\theta_m)$ satisfying 
\eqref{Reg:ApproxSolution}-\eqref{Approx:Problem:IC}. To this end, we fix velocity fields $(\bv_\ast,\bw_\ast)\in H^1(0,T;\mathbfcal{V}_m)$ and consider the convective Cahn--Hilliard system
\begin{subequations}\label{CCH:Fixed}
    \begin{align}
        &\delt\phi_m + \Div(\phi_m\bv_\ast) = \Lap\mu_m && \text{in} \ Q, \label{CCH:Fixed:1}\\
        &\mu_m = -\Lap\phi_m + F'(\phi_m)   && \text{in} \ Q, \label{CCH:Fixed:2} \\
        &\delt\psi_m + \Divg(\psi_m\bw_\ast) = \Lapg\theta_m - \beta \deln\mu_m && \text{on} \ \Sigma, \label{CCH:Fixed:3} \\
        &\theta_m = - \Lapg\psi_m + G'(\psi_m) + \alpha\deln\phi_m && \text{on} \ \Sigma, \label{CCH:Fixed:4} \\
        &\begin{cases} K\deln\phi_m = \alpha\psi_m - \phi_m &\text{if} \ K\in [0,\infty),  \\
        \deln\phi_m = 0 &\text{if} \ K = \infty
        \end{cases} && \text{on} \ \Sigma, \label{CCH:Fixed:5} \\
        &\begin{cases} 
        L \deln\mu_m = \beta\theta_m - \mu_m &\text{if} \  L\in(0,\infty), \\
        \deln\mu_m = 0 &\text{if} \ L=\infty
        \end{cases} &&\text{on} \ \Sigma, \label{CCH:Fixed:6} \\
        &\phi_m\vert_{t=0} = \phi_0 &&\text{in~}\Om, \label{CCH:Fixed:7} \\
        &\psi_m\vert_{t=0} = \psi_0 &&\text{on~}\Ga. \label{CCH:Fixed:8}
    \end{align}
\end{subequations}
Thanks to \cite[Theorem~3.4]{Knopf2025} (see also \cite[Theorem~3.2]{Giorgini2025}), there exists a unique solution $(\phi_m,\psi_m,\mu_m,\theta_m)$ to \eqref{CCH:Fixed} with
\begin{align}\label{CCH:Fixed:Regularity}
    \left\{\;
    \begin{aligned}
        (\phi_m,\psi_m)&\in H^1(0,T;(\mathcal{H}^1)^\prime) \cap L^\infty(0,T;\mathcal{H}^1_{K,\alpha})\cap L^2(0,T;\mathcal{W}^{2,6}), \\
        (\mu_m,\theta_m)&\in L^2(0,T;\mathcal{H}^1), \\
        (F^\prime(\phi_m),G^\prime(\psi_m))&\in L^2(0,T;\mathcal{L}^6),
    \end{aligned}
    \right.
\end{align}
in the sense that the variational formulations \eqref{WF:PP:DISCR} and \eqref{WF:MT:DISCR} hold with $(\bv_\ast,\bw_\ast)$ instead of $(\bv_m,\bw_m)$. Furthermore, the phase fields $\phi_m$ and $\psi_m$ satisfy 
\begin{align}\label{CCH:Fixed:<1}
    \abs{\phi_m} < 1 \ \text{a.e.~in~}Q \quad\text{and}\quad \abs{\psi_m} < 1 \ \text{a.e.~on~}\Sigma.
\end{align}
We also know from \cite[Theorem~3.4]{Knopf2025} that $(\phi_m,\psi_m,\mu_m,\theta_m)$ satisfies the energy inequality
\begin{align}
    \label{CCH:Fixed:EnergyInequality}
        \begin{split}
            &E_{\mathrm{free}}(\phi_m(t),\psi_m(t)) + \int_0^t\intO \abs{\Grad\mu_m}^2\dxs + \int_0^t\intG \abs{\Gradg\theta_m}^2\dGs \\
            &\qquad + \chi(L) \int_0^t\intG (\beta\theta_m - \mu_m)^2\dGs \\
            &\quad \leq E_{\mathrm{free}}(\phi_0,\psi_0) + \int_0^t\intO \phi_m\Grad\mu_m\cdot\bv_\ast\dxs + \int_0^t\intG \psi_m\Gradg\theta_m\cdot\bw_\ast\dGs
        \end{split}
\end{align}
for all $t\in[0,T]$, where the bulk-surface free energy is given by
\begin{align*}
    E_{\mathrm{free}}(\zeta,\xi) &= \intO \frac12\abs{\Grad\zeta}^2 + F(\zeta) \dx + \intG \frac12\abs{\Gradg\xi}^2 + G(\xi) \dG \\
    &\quad + \chi(K) \intG \frac12(\alpha\xi - \zeta)^2\dG.
\end{align*}
Using Hölder's and Young's inequality combined with \eqref{CCH:Fixed:<1}, we immediately infer from \eqref{CCH:Fixed:EnergyInequality} that
\begin{align}
    \label{CCH:Fixed:EnergyEstimate}
        \begin{split}
            &E_{\mathrm{free}}(\phi_m(t),\psi_m(t)) + \frac12\int_0^t \norm{(\mu_m,\theta_m)}_{L,\beta}^2\ds 
            \leq E_{\mathrm{free}}(\phi_0,\psi_0) + \frac12\int_0^t\norm{(\bv_\ast,\bw_\ast)}_{\mathbfcal{L}^2}^2\ds
        \end{split}
\end{align}
for all $t\in[0,T]$.

Next, for any $t\in[0,T]$, we make the ansatz
\begin{align*}
    \scp{\bv_m(t)}{\bw_m(t)} = \sum_{j=1}^m a_j^k(t)\scp{\tv_j}{\tw_j} \qquad\text{in~} \Om\times\Ga
\end{align*}
as a solution to the Galerkin approximation of \eqref{WF:VW:DISCR} that reads as
\begin{align}
    &\intO\rho(\phi_m)\delt\bv_m\cdot\tv_l\dx + \intG\sigma(\psi_m)\delt\bw_m\cdot\tw_l\dG \nonumber \\
    &\qquad + \intO\rho(\phi_m)(\bv_\ast\cdot\Grad)\bv_m\cdot\tv_l\dx + \intG\sigma(\psi_m)(\bw_\ast\cdot\Gradg)\bw_m\cdot\tw_l\dG \nonumber \\
    &\qquad + \intO\rho^\prime\,(\Grad\mu_m\cdot\Grad)\bv_m\cdot\tv_l\dx + \intG\sigma^\prime\,(\Gradg\theta_m\cdot\Gradg)\bw_m\cdot\tw_l\dG \nonumber \label{GalerkinApprox} \\
    &\qquad  + \intO2\nu_\Om(\phi_m)\D\bv_m:\D\tv_l\dx + \intG2\nu_\Ga(\psi_m)\Dg\bw_m:\Dg\tw_l\dG \\
    &\qquad + \intG \gamma(\phi_m,\psi_m)\bw_m\cdot\tw_l\dG \nonumber \\
    &\quad = \intO\mu_m\Grad\phi_m\cdot\tv_l\dx + \intG\theta_m\Gradg\psi_m\cdot\tw_l\dG \nonumber \\
    &\qquad + \frac{\chi(L)}{2}\intG(\beta\sigma^\prime - \rho^\prime)(\beta\theta_m - \mu_m)\bw_m\cdot\tw_l\dG, \nonumber
\end{align}
a.e.~on $[0,T]$ for all $l=1,\ldots,m$, supplemented with the initial condition 
\begin{align}
    \label{GalerkinApprox:IC}
    \scp{\bv_m}{\bw_m}\vert_{t=0} = \mathbfcal{P}_{\mathbfcal{V}_m}(\bv_0,\bw_0) \qquad\text{a.e.~in~}\Om\times\Ga.
\end{align}
Setting $\A_m(t) = (a_1^m(t),\ldots,a_m^m(t))^T$, the system \eqref{GalerkinApprox} is equivalent to the system of differential equations
\begin{align}
    \label{ODE:Galerkin}
    \M^m(t)\ddt \A^m(t) + \L^m(t)\A^m(t) = \G^m(t),
\end{align}
where
\begin{align*}
    (\M^m(t))_{l,j} &\coloneqq \intO \rho(\phi_m)\tv_j\cdot\tv_l \dx + \intG \sigma(\psi_m)\tw_j\cdot\tw_l\dG, \\
    (\L^m(t))_{l,j} &\coloneqq \intO \rho(\phi_m)(\bv_\ast\cdot\Grad)\tv_j\cdot\tv_l\dx + \intG \sigma(\psi_m)(\bw_\ast\cdot\Gradg)\tw_j\cdot\tw_l\dG \\
    &\quad + \intO \rho^\prime\,(\Grad\mu_m\cdot\Grad)\tv_j\cdot\tv_l\dx + \intG \sigma^\prime\,(\Gradg\theta_m\cdot\Gradg)\tw_j\cdot\tw_l\dG \\
    &\quad + \intO 2\nu_\Om(\phi_m)\D\tv_j:\D\tv_l\dx + \intG 2\nu_\Ga(\psi_m)\Dg\tw_j:\Dg\tw_l\dG \\
    &\quad + \intG \gamma(\phi_m,\psi_m)\tw_j\cdot\tw_l\dG \\
    &\quad - \frac{\chi(L)}{2}\intG (\beta\sigma^\prime -\rho^\prime)(\beta\theta_m - \mu_m)\tw_j\cdot\tw_l\dG, \\
    (\G^m(t))_l &\coloneqq \intO \mu_m\Grad\phi_m\cdot\tv_l\dx + \intG \theta_m\Gradg\psi_m\cdot\tw_l\dG,
\end{align*}
and the initial condition is
\begin{align*}
    \A^m(0)_j = (\mathbfcal{P}_{\mathbfcal{V}_m}(\bv_0,\bw_0),(\tv_j,\tw_j))_{\mathbfcal{L}^2}, \qquad j=1,\ldots,m.
\end{align*}
Proceeding as in \cite{Giorgini2020}, it is straightforward to check that $\M^m(t)$ is positive definite for all $t\in [0,T]$.
Defining
\begin{align*}
    f_m(t,\Y) \coloneqq (\M^m)^{-1}(t)\L^m(t)\Y + (\M^m)^{-1}(t)\G^m(t)
\end{align*}
for $\Y\in\R^m$ and $t\in[0,T]$, \eqref{ODE:Galerkin} can be equivalently restated as
\begin{align}\label{ODE:f_m}
    \ddt \A^m(t) = f_m\big(t,\A^m(t)\big) \qquad\text{for~}t\in[0,T].
\end{align}
As $(\phi_m,\psi_m)\in C([0,T];\mathcal{L}^2)$, we infer that $\M^m\in C([0,T];\R^{m\times m})$. Since $\M^m(t)$ is positive definite for all $t\in [0,T]$, we deduce that $\det\big(\M^m(\cdot)\big)$ is continuous and uniformly positive on $[0,T]$.
Hence, according to Cramer's rule for the inverse, we also have $(\M^m)^{-1}\in C([0,T];\R^{m\times m})$. Next, we find functions $h,l\in L^2(0,T)$ such that
\begin{align*}
    \abs{f_m(t,\Y)} \leq h(t)
\end{align*}
for $\Y\in\R^m$, and
\begin{align*}
    \abs{f_m(t,\Y) - f_m(t,\overline{\Y})} \leq l(t) \abs{\Y - \overline{\Y}}
\end{align*}
for $\Y,\overline{\Y}\in\R^m$. Thus, we are in a position to apply Carath\'{e}odory's existence theorem (see, e.g.,  \cite[Theorem XVIII]{Walter1998}) to find a unique solution $\A^m\in H^1(0,T;\R^m)$ of \eqref{ODE:f_m}. This, in turn, implies the existence of a unique pair $(\bv_m,\bw_m)\in H^1(0,T;\mathbfcal{V}_m)$ that solves \eqref{GalerkinApprox}-\eqref{GalerkinApprox:IC}.

Now, we multiply \eqref{GalerkinApprox} by $a_l^k$ and sum over $l=1,\ldots,m$ to find that
\begin{align}\label{Galerkin:Test:a}
    &\intO \frac12\rho(\phi_m)\delt\abs{\bv_m}^2\dx + \intG \frac12\sigma(\psi_m)\delt\abs{\bw_m}^2\dG \nonumber \\
    &\qquad + \intO \frac12\rho(\phi_m)\bv_\ast\cdot\Grad\abs{\bv_m}^2\dx + \intG \frac12\sigma(\psi_m)\bw_\ast\cdot\Gradg\abs{\bw_m}^2\dG \nonumber \\
    &\qquad + \intO 2\nu_\Om(\phi_m)\abs{\D\bv_m}^2\dx + \intG 2\nu_\Ga(\psi_m)\abs{\Dg\bw_m}^2\dG + \intG \gamma(\phi_m,\psi_m)\abs{\bw_m}^2\dG \nonumber \\
    &\qquad - \intO \frac12\rho^\prime\,\Grad\mu_m\cdot\Grad\abs{\bv_m}^2 \dx - \intG \frac12\sigma^\prime\,\Gradg\theta_m\cdot\Gradg\abs{\bw_m}^2\dG\\
    &\quad = \intO \mu_m\Grad\phi_m\cdot\bv_m\dx + \intG \theta_m\Gradg\psi_m\cdot\bw_m\dG \nonumber \\
    &\qquad + \frac{\chi(L)}{2} \intG (\beta\sigma^\prime -\rho^\prime)(\beta\theta_m - \mu_m)\abs{\bw_m}^2\dG. \nonumber
\end{align}
First, we notice that
\begin{align*}
    (\delt\rho(\phi_m),\delt\sigma(\psi_m)) = (\rho^\prime\,\delt\phi_m,\sigma^\prime\,\delt\psi_m) \qquad\text{in~} L^2(0,T;(\mathcal{H}^1)^\prime)
\end{align*}
which can be easily seen using difference quotients. 
In particular, we have
\begin{align*}
    &\bigang{(\rho^\prime\,\delt\phi_m,\sigma^\prime\,\delt\psi_m)}{(\abs{\bv_m}^2,\abs{\bw_m}^2)}_{\mathcal{H}^1}
    \\
    &= \rho^\prime\,\bigang{\delt\phi_m}{\abs{\bv_m}^2}_{H^1(\Omega)}
    + \sigma^\prime\,\bigang{\delt\psi_m}{\abs{\bw_m}^2}_{H^1(\Gamma)} 
    \\
    &= \bigang{(\delt\phi_m,\delt\psi_m)}
    {(\rho^\prime\abs{\bv_m}^2,\sigma^\prime\abs{\bw_m}^2)}_{\mathcal{H}^1}.
\end{align*}
Hence, using the weak formulation \eqref{WF:PP:DISCR}, we obtain
\begin{align*}
    &\intO \frac12\rho(\phi_m)\delt\abs{\bv_m}^2\dx + \intG \frac12\sigma(\psi_m)\delt\abs{\bw_m}^2\dG \\
    &\quad = \ddt \intO \frac12 \rho(\phi_m)\abs{\bv_m}^2\dx + \ddt \intG \frac12 \sigma(\psi_m)\abs{\bw_m}^2\dG \\
    &\qquad - \frac12 \bigang{(\delt\phi_m,\delt\psi_m)}
    {(\rho^\prime\abs{\bv_m}^2,\sigma^\prime\abs{\bw_m}^2)}_{\mathcal{H}^1} \\
    &\quad = \ddt \intO \frac12 \rho(\phi_m)\abs{\bv_m}^2\dx + \ddt \intG \frac12 \sigma(\psi_m)\abs{\bw_m}^2\dG \\
    &\qquad - \intO \frac12\rho^\prime\,\phi_m\bv_\ast\cdot\Grad\abs{\bv_m}^2\dx - \intG \frac12\sigma^\prime\,\psi_m\bw_\ast\cdot\Gradg\abs{\bw_m}^2\dG \\
    &\qquad + \intO \frac12 \rho^\prime\,\Grad\mu_m\cdot\Grad\abs{\bv_m}^2\dx + \intG \frac12 \sigma^\prime\,\Gradg\theta_m\cdot\Gradg\abs{\bw_m}^2\dG \\
    &\qquad + \chi(L) \intG \frac12 (\beta\sigma^\prime -\rho^\prime)(\beta\theta_m - \mu_m)\abs{\bw_m}^2\dG.
\end{align*}
Then, since
\begin{align*}
    \intO \bv_\ast\cdot\Grad\abs{\bv_m}^2\dx = \intG \bw_\ast\cdot\Gradg\abs{\bw_m}^2\dG = 0,
\end{align*}
and, as $\rho(\phi_m)$ and $\rho^\prime\,\phi_m$ as well as $\sigma(\psi_m)$ and $\sigma^\prime\,\psi_m$ differ only by an additive constant, respectively, we infer
by means of \eqref{Galerkin:Test:a} that
\begin{align}\label{IntParts:Final}
    &\ddt \Big(\intO \frac12 \rho(\phi_m)\abs{\bv_m}^2\dx + \intG \frac12 \sigma(\psi_m)\abs{\bw_m}^2 \dG \Big)
    + \intO 2\nu_\Om(\phi_m)\abs{\D\bv_m}^2\dx
    \nonumber \\
    &\qquad  + \intG 2\nu_\Ga(\psi_m)\abs{\Dg\bw_m}^2\dG + \intG \gamma(\phi_m,\psi_m)\abs{\bw_m}^2\dG 
    \\ 
    &\quad = - \intO \phi_m\Grad\mu_m\cdot\bv_m\dx - \intG \psi_m\Gradg\theta_m\cdot\bw_m\dG. \nonumber
\end{align}
Using \eqref{CCH:Fixed:<1} and the bulk-surface Korn inequality (see Lemma~\ref{Lemma:Korn}), we can estimate the right-hand side of \eqref{IntParts:Final} as
\begin{align*}
    &\abs{\intO \phi_m\Grad\mu_m\cdot\bv_m\dx + \intG \psi_m\Gradg\theta_m\cdot\bw_m\dG} \\
    &\quad\leq \norm{\Grad\mu_m}_{\mathbf{L}^2(\Om)}\norm{\bv_m}_{\mathbf{L}^2(\Om)} + \norm{\Gradg\theta_m}_{\mathbf{L}^2(\Ga)}\norm{\bw_m}_{\mathbf{L}^2(\Ga)} \\
    &\quad\leq \frac{\min\{\gamma_\ast,2\nu_\ast\}}{2}\big(\norm{\D\bv_m}_{\mathbf{L}^2(\Om)}^2 + \norm{\Dg\bw_m}_{\mathbf{L}^2(\Ga)}^2 + \norm{\bw_m}_{\mathbf{L}^2(\Ga)}^2\big) \\
    &\qquad + \frac{C_K^2}{2\min\{\gamma_\ast,2\nu_\ast\}}\big(\norm{\Grad\mu_m}_{\mathbf{L}^2(\Om)}^2 + \norm{\Gradg\theta_m}_{\mathbf{L}^2(\Ga)}^2\big).
\end{align*}
In this way, we derive the differential inequality
\begin{align*}
    \begin{split}
    &\ddt \Big(\intO \frac12 \rho(\phi_m)\abs{\bv_m}^2\dx + \intG \frac12 \sigma(\psi_m)\abs{\bw_m}^2 \dG \Big) 
    \\
    &\qquad + \frac{\min\{\gamma_\ast,2\nu_\ast\}}{2}\Big(\norm{\D\bv_m}_{\mathbf{L}^2(\Om)}^2 + \norm{\Dg\bw}_{\mathbf{L}^2(\Ga)}^2 + \norm{\bw_m}_{\mathbf{L}^2(\Ga)}^2\Big) 
    \end{split}
    \\ 
    &\quad\leq\frac{C_K^2}{2\min\{\gamma_\ast,2\nu_\ast\}}\big(\norm{\Grad\mu_m}_{\mathbf{L}^2(\Om)}^2 + \norm{\Gradg\theta_m}_{\mathbf{L}^2(\Ga)}^2\big).
\end{align*}
Integrating over $[0,s]$ with $s\in[0,T]$, and using \eqref{CCH:Fixed:EnergyEstimate}, it follows that
\begin{align*}
    &\intO \frac{\rho_\ast}{2}\abs{\bv_m(s)}^2\dx + \intG \frac{\sigma_\ast}{2}\abs{\bw_m(s)}^2\dG \\
    &\quad\leq \intO \frac{\rho^\ast}{2}\abs{\mathbf{P}_{\mathbfcal{V}_m}^\Om(\bv_0)}^2\dx + \intG \frac{\sigma^\ast}{2}\abs{\mathbf{P}_{\mathbfcal{V}_m}^\Ga(\bw_0)}^2\dG + \frac{C_K^2}{\min\{\gamma_\ast,2\nu_\ast\}}E_{\mathrm{free}}(\phi_0,\psi_0) \\
    &\qquad + \frac{C_K^2}{2\min\{\gamma_\ast,2\nu_\ast\}}\int_0^s \norm{(\bv_\ast(\tau),\bw_\ast(\tau))}_{\mathbfcal{L}^2}^2\dtau.
\end{align*}
This, in turn, entails that
\begin{align}\label{Est:vw:m:int:-1}
    \begin{split}
        &\norm{\bv_m(s)}_{\mathbf{L}^2(\Om)}^2 + \norm{\bw_m(s)}_{\mathbf{L}^2(\Ga)}^2 \\
        &\quad\leq \rho^\ast c_*\norm{\bv_0}_{L^2(\Om)}^2 + \sigma^\ast c_*\norm{\bw_0}_{L^2(\Ga)}^2 + \frac{2C_K^2}{\min\{\gamma_\ast,2\nu_\ast\}} c_*E_{\mathrm{free}}(\phi_0,\psi_0) \\
        &\qquad + \frac{C_K^2}{\min\{\gamma_\ast,2\nu_\ast\}} c_*\int_0^s\norm{(\bv_\ast(\tau),\bw_\ast(\tau))}_{\mathbfcal{L}^2}^2\dtau,
    \end{split}
\end{align}
where $c_* \coloneqq \frac{1}{\rho_\ast} + \frac{1}{\sigma_\ast}$.
To simplify the presentation, we now define
\begin{align*}
    C_1 &\coloneqq \rho^\ast c_*\norm{\bv_0}_{\mathbf{L}^2(\Om)}^2 + \sigma^\ast c_*\norm{\bw_0}_{\mathbf{L}^2(\Ga)}^2 + \frac{2C_K^2}{\min\{\gamma_\ast,2\nu_\ast\}} c_*E_{\mathrm{free}}(\phi_0,\psi_0), 
    \\
    C_2 &\coloneqq \frac{C_K^2}{\min\{\gamma_\ast,2\nu_\ast\}} c_*,
    \qquad
    C_3 \coloneqq C_1 T.
\end{align*}
Assuming
\begin{align}\label{Assump:vw:_star}
    \int_0^t \norm{(\bv_\ast(\tau),\bw_\ast(\tau))}_{\mathbfcal{L}^2}^2\dtau \leq C_3 \e^{C_2t}, \qquad t\in[0,T],
\end{align}
we deduce from \eqref{Est:vw:m:int:-1} that
\begin{align}\label{Est:vw:m:int}
    \begin{split}
        \int_0^t \norm{(\bv_m(s),\bw_m(s))}_{\mathbfcal{L}^2}^2\ds \leq C_3 + C_2\int_0^t\int_0^s\norm{(\bv_\ast(\tau),\bw_\ast(\tau))}_{\mathbfcal{L}^2}^2\dtau\ds \leq C_3\e^{C_2t}
    \end{split}
\end{align}
for all $t\in[0,T]$. Furthermore, we also infer from \eqref{Est:vw:m:int:-1} that
\begin{align}\label{Est:vw:m:sup}
    \sup_{t\in[0,T]} \norm{(\bv_m(t),\bw_m(t))}_{\mathbfcal{L}^2} \leq \big(C_1 + C_3C_2\e^{C_2T}\big)^{1/2} \eqqcolon M_0.
\end{align}
Next, we aim to control the time derivative of $(\bv_m,\bw_m)$. To this end, we multiply \eqref{GalerkinApprox} by $\ddt a_l^k$ and sum over $l=1,\ldots,m$. This yields
\begin{align}
    &\intO \rho(\phi_m)\abs{\delt\bv_m}^2\dx + \intG \sigma(\psi_m)\abs{\delt\bw_m}^2\dG 
    \nonumber \\
    &\quad = - \intO \rho(\phi_m)(\bv_\ast\cdot\Grad)\bv_m\cdot\delt\bv_m\dx - \intG \sigma(\psi_m)(\bw_\ast\cdot\Gradg)\bw_m\cdot\delt\bw_m\dG 
    \nonumber \\
    &\qquad - \intO 2\nu_\Om(\phi_m)\D\bv_m:\D\delt\bv_m\dx - \intG 2\nu_\Ga(\psi_m)\Dg\bw_m:\Dg\delt\bw_m\dG 
    \nonumber \\
    &\qquad - \intG \gamma(\phi_m,\psi_m)\bw_m\cdot\delt\bw_m\dG 
    \\
    &\qquad + \intO \rho^\prime\,(\Grad\mu_m\cdot\Grad)\bv_m\cdot\delt\bv_m\dx + \intG \sigma^\prime\,(\Gradg\theta_m\cdot\Gradg)\bw_m\cdot\delt\bw_m\dG 
    \nonumber \\
    &\qquad - \intO \phi_m\Grad\mu_m\cdot\delt\bv_m\dx - \intG \psi_m\Gradg\theta_m\cdot\delt\bw_m\dG 
    \nonumber \\
    &\qquad + \frac{\chi(L)}{2}\intG (\beta\sigma^\prime -\rho^\prime)(\beta\theta_m - \mu_m)\bw_m\cdot\delt\bw_m\dG.
    \nonumber 
\end{align}
Taking advantage of \eqref{Est:InverseSobolev}, we find that
\begin{align}\label{Est:Galerkin:Test:ddta}
    \begin{split}
        &\rho_\ast\norm{\delt\bv_m}_{\mathbf{L}^2(\Om)}^2 
        + \sigma_\ast\norm{\delt\bw_m}_{\mathbf{L}^2(\Ga)}^2 
        \\[1ex]
        &\quad\leq\rho^\ast\norm{\bv_\ast}_{\mathbf{L}^2(\Om)}\norm{\Grad\bv_m}_{\mathbf{L}^\infty(\Om)}\norm{\delt\bv_m}_{\mathbf{L}^2(\Om)} 
        + \sigma^\ast\norm{\bw_\ast}_{\mathbf{L}^2(\Om)}\norm{\Gradg\bw_m}_{\mathbf{L}^\infty(\Ga)}\norm{\delt\bw_m}_{\mathbf{L}^2(\Ga)} 
        \\
        &\qquad  + 2\nu^\ast\norm{\D\bv_m}_{\mathbf{L}^2(\Om)}\norm{\D\delt\bv_m}_{\mathbf{L}^2(\Om)} 
        + 2\nu^\ast\norm{\Dg\bw_m}_{\mathbf{L}^2(\Ga)}\norm{\Dg\delt\bw_m}_{\mathbf{L}^2(\Ga)} 
        \\
        &\qquad 
        + \gamma^\ast\norm{\bw_m}_{\mathbf{L}^2(\Ga)}\norm{\delt\bw_m}_{\mathbf{L}^2(\Ga)}
        + \abs{\rho^\prime\,}\norm{\Grad\mu_m}_{\mathbf{L}^2(\Om)}\norm{\Grad\bv_m}_{\mathbf{L}^\infty(\Om)}\norm{\delt\bv_m}_{\mathbf{L}^2(\Om)}  
        \\
        &\qquad 
        + \abs{\sigma^\prime\,}\norm{\Gradg\theta_m}_{\mathbf{L}^2(\Ga)}\norm{\Gradg\bw_m}_{\mathbf{L}^\infty(\Ga)}\norm{\delt\bw_m}_{\mathbf{L}^2(\Ga)}
        + \norm{\Grad\mu_m}_{\mathbf{L}^2(\Om)}\norm{\delt\bv_m}_{\mathbf{L}^2(\Om)} 
        \\
        &\qquad 
        + \norm{\Gradg\theta_m}_{\mathbf{L}^2(\Ga)}\norm{\delt\bw_m}_{\mathbf{L}^2(\Ga)}  
        \\
        &\qquad 
        + \tfrac12\chi(L)\abs{\beta\sigma^\prime -\rho^\prime\,}\norm{\beta\theta_m - \mu_m}_{L^2(\Ga)}\norm{\bw_m}_{\mathbf{L}^\infty(\Ga)}\norm{\delt\bw_m}_{\mathbf{L}^2(\Ga)}
        \\[1ex]
        &\quad\leq C(\rho^\ast + \sigma^\ast) \norm{(\bv_\ast,\bw_\ast)}_{\mathbfcal{L}^2}\norm{(\bv_m,\bw_m)}_{\mathbfcal{W}^{2,4}}\norm{(\delt\bv_m,\delt\bw_m)}_{\mathbfcal{L}^2} 
        \\
        &\qquad + \nu^\ast C_m^2\norm{(\bv_m,\bw_m)}_{\mathbfcal{L}^2}\norm{(\delt\bv_m,\delt\bw_m)}_{\mathbfcal{L}^2} + \gamma^\ast\norm{\bw_m}_{L^2(\Ga)}\norm{\delt\bw_m}_{L^2(\Ga)} 
        \\
        &\qquad + C\big(\abs{\rho^\prime\,} + \abs{\sigma^\prime\,}\big)\norm{(\Grad\mu_m, \Gradg\theta_m)}_{\mathbfcal{L}^2}\norm{(\bv_m,\bw_m)}_{\mathbfcal{W}^{2,4}}\norm{(\delt\bv_m,\delt\bw_m)}_{\mathbfcal{L}^2} 
        \\
        &\qquad + C\norm{(\Grad\mu_m,\Gradg\theta_m)}_{\mathbfcal{L}^2}\norm{(\delt\bv_m,\delt\bw_m)}_{\mathbfcal{L}^2} 
        \\
        &\qquad + \chi(L)C\abs{\beta\sigma^\prime -\rho^\prime\,}\norm{\beta\theta - \mu_m}_{L^2(\Ga)}\norm{\bw_m}_{\mathbf{H}^2(\Ga)}\norm{\delt\bw_m}_{\mathbf{L}^2(\Ga)} 
        \\[1ex]
        &\quad\leq CC_m(\rho^\ast + \sigma^\ast)\norm{(\bv_\ast,\bw_\ast)}_{\mathbfcal{L}^2}\norm{(\bv_m,\bw_m)}_{\mathbfcal{L}^2}\norm{(\delt\bv_m,\delt\bw_m)}_{\mathbfcal{L}^2} 
        \\
        &\qquad + \nu^\ast C_m^2\norm{(\bv_m,\bw_m)}_{\mathbfcal{L}^2}\norm{(\delt\bv_m,\delt\bw_m)}_{\mathbfcal{L}^2} + \gamma^\ast\norm{\bw_m}_{L^2(\Ga)}\norm{\delt\bw_m}_{L^2(\Ga)} 
        \\
        &\qquad + CC_m\big(\abs{\rho^\prime\,} + \abs{\sigma^\prime\,}\big)\norm{(\Grad\mu_m, \Gradg\theta_m)}_{\mathbfcal{L}^2}\norm{(\bv_m,\bw_m)}_{\mathbfcal{L}^2}\norm{(\delt\bv_m,\delt\bw_m)}_{\mathbfcal{L}^2} 
        \\
        &\qquad + C\norm{(\Grad\mu_m,\Gradg\theta_m)}_{\mathbfcal{L}^2}\norm{(\delt\bv_m,\delt\bw_m)}_{\mathbfcal{L}^2} 
        \\
        &\qquad + \chi(L)CC_m\abs{\beta\sigma^\prime -\rho^\prime\,}\norm{\beta\theta_m - \mu_m}_{L^2(\Ga)}\norm{(\bv_m,\bw_m)}_{\mathbfcal{L}^2}\norm{\delt\bw_m}_{\mathbf{L}^2(\Ga)}. 
        \\
    \end{split}
\end{align}
Eventually, in view of \eqref{CCH:Fixed:EnergyInequality} and \eqref{Assump:vw:_star}-\eqref{Est:vw:m:sup}, we conclude that
\begin{align}\label{Est:delt:vw:m:int}
    &\int_0^T \norm{(\delt\bv_m(\tau),\delt\bw_m(\tau))}_{\mathbfcal{L}^2}^2\dtau \nonumber\\
    &\quad\leq 3\Big(CC_m(\rho^\ast + \sigma^\ast) c_*M_0\Big)^2\int_0^T\norm{(\bv_\ast,\bw_\ast)}_{\mathbfcal{L}^2}^2\dtau  \nonumber\\
    &\qquad + 3\Big(\nu^\ast C_m^2 c_*\Big)^2C_3\e^{C_2T} + 3\Big(\gamma^\ast c_*\Big)^2C_3\e^{C_2T} \nonumber \\
    &\qquad + 3\Big(CC_m\Big(\abs{\rho^\prime\,} + \abs{\sigma^\prime\,}\Big) c_*M_0 + M_0\Big)^2\int_0^T \norm{(\Grad\mu_m(\tau),\Gradg\theta_m(\tau))}_{\mathbfcal{L}^2}^2\dtau \nonumber \\
    &\qquad + \Big(CC_m\abs{\beta\sigma^\prime -\rho^\prime\,} c_*M_0\Big)^2\int_0^T\norm{\beta\theta_m(\tau) - \mu_m(\tau)}_{L^2(\Ga)}^2\dtau \nonumber \\
    &\quad\leq \Bigg(3\Big(CC_m(\rho^\ast + \sigma^\ast) c_*M_0\Big)^2 +  3\Big(\nu^\ast C_m^2 c_*\Big)^2 + 3\Big(\gamma^\ast c_*\Big)\Bigg)C_3\e^{C_2T} \nonumber \\
    &\qquad + 3\Big(CC_m\Big(\abs{\rho^\prime\,} + \abs{\sigma^\prime\,}\Big) c_*M_0 + M_0\Big)^2\Big(2E_{\mathrm{free}}(\phi_0,\psi_0) + C_3\e^{C_2T}\Big) \nonumber \\
    &\qquad + \Big(CC_m\abs{\beta\sigma^\prime -\rho^\prime\,} c_*M_0\Big)^2\Big(2E_{\mathrm{free}}(\phi_0,\psi_0) + C_3\e^{C_2T}\Big) \nonumber \\
    &\quad\eqqcolon M_m^2,
\end{align}
where $M_m$ depends only on $\rho_\ast,\rho^\ast,\sigma_\ast,\sigma^\ast,\gamma^\ast, E_{\mathrm{free}}(\phi_0,\psi_0), L,T,\Om$ and $m$.

Now, we introduce the setting of the fixed-point argument. To this end, we define the set
\begin{align*}
    \mathbfcal{M} = 
    \left\{(\bv,\bw)\in H^1(0,T;\mathbfcal{V}_m) \;\middle|\; 
    \begin{aligned}
        &\int_0^t\norm{(\bv(\tau),\bw(\tau))}_{\mathbfcal{L}^2}^2\dtau \leq C_3\e^{C_2t}, \ t\in[0,T],
        \\ 
        &\norm{(\delt\bv,\delt\bw)}_{L^2(0,T;\mathbfcal{V}_m)}\leq M_m
    \end{aligned}
    \right\},
\end{align*}
which is a subset of $L^2(0,T;\mathbfcal{V}_m)$. Next, we define the map
\begin{align*}
    \Lambda:\mathbfcal{M}\rightarrow L^2(0,T;\mathbfcal{V}_m), \qquad \Lambda(\bv_\ast,\bw_\ast) = (\bv_m,\bw_m),
\end{align*}
where $(\bv_m,\bw_m)$ is the solution to system \eqref{GalerkinApprox} in which the functions $(\phi_m,\psi_m,\mu_m,\theta_m)$ are the solution of \eqref{CCH:Fixed}. In light of \eqref{Est:vw:m:int} and \eqref{Est:delt:vw:m:int}, we infer that $\Lambda(\mathbfcal{M}) \subseteq \mathbfcal{M}$. It is easily seen that $\mathbfcal{M}$ is convex and closed. Furthermore, $\mathbfcal{M}$ is a compact set in $L^2(0,T;\mathbfcal{V}_m)$. 

We are left to prove that the map $\Lambda$ is continuous. To this end, let $(\wv_\ast,\ww_\ast)\in \mathbfcal{M}$ be arbitrary and let $\{(\wv_n,\ww_n)\}_{n\in\N}\subset\mathbfcal{M}$ be a sequence with $(\wv_n,\ww_n)\rightarrow (\wv_\ast,\ww_\ast)$ in $L^2(0,T;\mathbfcal{V}_m)$. Arguing as above, we find a corresponding sequence $\{(\wphi_n,\wpsi_n,\wmu_n,\wtheta_n)\}_{n\in\N}$ and a quadruplet $(\wphi_\ast,\wpsi_\ast,\wmu_\ast,\wtheta_\ast)$, which solve the convective bulk-surface Cahn--Hilliard equation \eqref{CCH:Fixed} with $(\wv_n,\ww_n)$ and $(\wv_\ast,\ww_\ast)$ instead of $(\bv_\ast,\bw_\ast)$, respectively. Then, since $(\wv_\ast,\ww_\ast)$ belongs to $\mathbfcal{M}$, \cite[Theorem~3.5]{Giorgini2025} provides the continuous dependence estimate
\begin{align*}
    &\ddt\norm{((\wphi_n(t)- \wphi(t), \wpsi_n(t) - \wpsi_\ast(t))}_{L,\beta,\ast}^2 + \norm{(\wphi_n - \wphi_\ast, \wpsi_n - \wpsi_\ast)}_{K,\alpha}^2 
    \\
    &\qquad\leq C\norm{(\wv_n - \wv_\ast, \ww_n - \ww_\ast)}_{\mathbfcal{L}^2}^2 + C\norm{(\wphi_n(t)- \wphi(t), \wpsi_n(t) - \wpsi_\ast(t))}_{L,\beta,\ast}^2.
\end{align*}
Also, by this construction, we have $(\wphi_n(0), \wpsi_n(0)) = (\wphi_\ast(0), \wpsi_\ast(0)) = (\phi_0,\psi_0)$.
Hence, if $K\in[0,\infty)$, the bulk-surface Poincar\'{e} inequality and Gronwall's lemma imply that
\begin{align}\label{ContDep:Conv}
    \begin{split}
        &\norm{(\wphi_n - \wphi_\ast, \wpsi_n - \wpsi_\ast)}_{L^\infty(0,T;(\mathcal{H}^1)^\prime)\cap L^2(0,T;\mathcal{H}^1)}^2 \\
        &\qquad \leq \e^{CT}\int_0^T \norm{(\wv_n(\tau) - \wv_\ast(\tau), \ww_n(\tau) - \ww_\ast(\tau))}_{\mathbfcal{L}^2}^2\dtau \rightarrow 0 \qquad\text{as~}n\rightarrow\infty.
    \end{split}
\end{align}
If $K = \infty$, we employ Ehrling's lemma to deduce that
\begin{align*}
    \norm{(\wphi_n - \wphi_\ast,\wpsi_n - \wpsi_\ast)}_{\mathcal{L}^2}^2 \leq \frac 12 \norm{(\wphi_n - \wphi_\ast, \wpsi_n - \wpsi_\ast)}_{\mathcal{H}^1}^2 
    + C \norm{(\wphi_n - \wphi_\ast, \wpsi_n - \wpsi_\ast)}_{L,\beta,\ast}^2.
\end{align*}
Consequently,
\begin{align*}
    \begin{split}
    \norm{(\wphi_n - \wphi_\ast, \wpsi_n - \wpsi_\ast)}_{\mathcal{L}^2}^2 &\leq \norm{(\Grad\wphi_n - \Grad\wphi_\ast,\Gradg\wpsi_n - \Gradg\wpsi_\ast)}_{\mathcal{L}^2}^2 \\
    &\qquad + C \norm{(\wphi_n - \wphi_\ast, \wpsi_n - \wpsi_\ast)}_{L,\beta,\ast}^2 \,.
    \end{split}
\end{align*}
Arguing as in the case $K\in[0,\infty)$, we conclude that \eqref{ContDep:Conv} remains true for $K = \infty$. 
Hence, for any choice of $K\in [0,\infty]$, we conclude that
\begin{align}
    \label{CONV:P*P*}
    (\wphi_n, \wpsi_n) \to (\wphi_\ast,\wpsi_\ast)
    \quad\text{strongly in $L^\infty(0,T;(\mathcal{H}^1)^\prime)\cap L^2(0,T;\mathcal{H}^1)$ and a.e.~in $Q\times\Sigma$,}
\end{align}
as $n\to\infty$, up to extraction of a subsequence.

As $\{(\wv_n,\ww_n)\}_{n\in\N}$ lies in $\mathbfcal{M}$, it is uniformly bounded in $L^2(0,T;\mathbfcal{L}^2)$.
Therefore, proceeding as in the proof of \cite[Theorem~3.4]{Knopf2025} (see also \cite{Giorgini2025}), we deduce the uniform estimate
\begin{align}
    \norm{(F_1^\prime(\wphi_n), G_1^\prime(\wpsi_n))}_{L^2(0,T;\mathcal{L}^2)}
    + \norm{(\wmu_n,\wtheta_n)}_{L^2(0,T;\mathcal{H}^1)} 
    \leq C. \label{Bound:PhiPsi:n}
\end{align}
Consequently, there exists a pair $(F_\ast,G_\ast)\in L^2(0,T;\mathcal{L}^2)$ such that, up to subsequence extraction,
\begin{align}
    \label{CONV:F*G*}
    (F_1^\prime(\wphi_n),G_1^\prime(\wpsi_n)) \rightarrow (F_\ast,G_\ast) \qquad\text{weakly in~} L^2(0,T;\mathcal{L}^2)
\end{align}
as $n\rightarrow\infty$. To identify the limit functions $(F_\ast, G_\ast)$ with $(F_1^\prime(\wphi_\ast), G_1^\prime(\wpsi_\ast))$, we employ Minty's trick. To be precise, from the weak-strong convergence principle, we infer
\begin{align}
    \label{MINTY:1}
    &\lim_{n\rightarrow\infty} \int_0^T\intO F_1^\prime(\wphi_n)\wphi_n\dxt = \int_0^T \intO F_\ast\wphi_\ast\dxt, 
    \\
    \label{MINTY:2}
    &\lim_{n\rightarrow\infty} \int_0^T\intG G_1^\prime(\wpsi_n)\wpsi_n\dGt = \int_0^T \intG G_\ast\wpsi_\ast\dGt,
\end{align}
from which the claim readily follows due to the assumptions on $F_1$ and $G_1$ (see, e.g., \cite[Lemma~1.3, p.~42]{Barbu1976}).
Combing \eqref{CONV:P*P*} and \eqref{CONV:F*G*}, we conclude from the weak formulation \eqref{WF:MT:DISCR} that
\begin{align}
    &\int_0^T\intO (\wmu_n - \wmu_\ast)\eta\dxt + \int_0^T\intG (\wtheta_n - \wtheta_\ast)\vartheta\dGt 
    \nonumber\\
    \begin{split}
        &\quad = \int_0^T\intO (\Grad\wphi_n - \Grad\wphi_\ast)\cdot\Grad\eta + (F^\prime(\wphi_n) - F^\prime(\wphi_\ast))\eta \dxt \\
        &\qquad + \int_0^T\intG (\Gradg\wpsi_n - \Gradg\wpsi_\ast)\cdot\Gradg\vartheta + (G^\prime(\wpsi_n) - G^\prime(\wpsi_\ast))\vartheta \dGt 
    \end{split}
    \\
    &\qquad + \chi(K) \int_0^T\intG (\alpha(\wpsi_n - \wpsi_\ast) - (\wphi_n - \wphi_\ast))(\alpha\vartheta - \eta)\dGt 
    \nonumber\\
    &\quad\longrightarrow 0
    \nonumber
\end{align}
as $n\rightarrow\infty$ for all $(\eta,\vartheta)\in L^2(0,T;\mathcal{H}^1_{K,\alpha})$. This proves that
\begin{align}
    (\wmu_n,\wtheta_n) \rightarrow (\wmu_\ast, \wtheta_\ast) \qquad\text{weakly in~} L^2(0,T;(\mathcal{H}^1_{K,\alpha})^\prime).
\end{align}
In combination with \eqref{Bound:PhiPsi:n} we infer that, up to subsequence extraction, 
\begin{align}
    \label{CONV:M*T*}
    (\wmu_n,\wtheta_n) \rightarrow (\wmu_\ast, \wtheta_\ast) \qquad\text{weakly~} L^2(0,T;\mathcal{H}^1).
\end{align}
As the limit functions in \eqref{CONV:P*P*}, \eqref{CONV:F*G*}, and \eqref{CONV:M*T*} do not depend on the choice of the extracted subsequence, we conclude that these convergences remain true for the whole sequence.

Now, to simplify the notation, we set $(\bu^\Om_n,\bu^\Ga_n) = \Lambda(\wv_n,\ww_n)$. Then, as $(\bu^\Om_n,\bu^\Ga_n)\in\mathbfcal{M}$ for all $n\in\N$, we infer the existence of $(\bu^\Om_\ast,\bu^\Ga_\ast)\in H^1(0,T;\mathbfcal{V}_m)$ such that, up to subsequence extraction,
\begin{alignat}{2}
    (\bu^\Om_n,\bu^\Ga_n) &\rightarrow (\bu^\Om_\ast,\bu^\Ga_\ast) &&\qquad\text{weakly in $H^1(0,T;\mathbfcal{V}_m)$}. \label{Conv:u_n}
\end{alignat}
As $\mathbfcal{V}_m$ is finite-dimensional, the embedding $H^1(0,T;\mathbfcal{V}_m) \emb C([0,T];\mathbfcal{V}_m)$ is compact. Together with the first inverse Sobolev embedding from \eqref{Est:InverseSobolev}, we infer that
\begin{alignat}{2}
    (\bu^\Om_n,\bu^\Ga_n) 
    &\rightarrow (\bu^\Om_\ast,\bu^\Ga_\ast),
    &&\qquad\text{strongly in~} 
    C([0,T];\mathbfcal{V}_m),
    \label{Conv:u_n:strong}
    \\
    (\Grad\bu^\Om_n,\Gradg\bu^\Ga_n) 
    &\rightarrow (\Grad\bu^\Om_\ast,\Gradg\bu^\Ga_\ast) 
    &&\qquad\text{strongly in~} C([0,T];\mathbfcal{L}^2),
    \label{Conv:Du_n:strong}
\end{alignat}
up to subsequence extraction.
Moreover, as the set $\mathbfcal{M} \subset L^2(0,T;\mathbfcal{V}_m)$ is bounded, closed, and convex, it is weakly sequentially compact. This readily yields $(\bu^\Om_\ast,\bu^\Ga_\ast)\in \mathbfcal{M}$.
By the definition of $\Lambda$, we know that
\begin{align}
    &\intO\rho(\wphi_n)\delt\bu^\Om_n\cdot\tv_l\dx + \intG\sigma(\wpsi_n)\delt\bu^\Ga_n\cdot\tw_l\dG \nonumber \\
    &\qquad + \intO\rho(\wphi_n)(\wv_n\cdot\Grad)\bu^\Om_n\cdot\tv_l\dx + \intG\sigma(\wpsi_n)(\ww_n\cdot\Gradg)\bu^\Ga_n\cdot\tw_l\dG \nonumber \\
    &\qquad + \intO\rho^\prime\,(\Grad\wmu_n\cdot\Grad)\bu^\Om_n\cdot\tv_l\dx + \intG\sigma^\prime\,(\Gradg\wtheta_n\cdot\Gradg)\bu^\Ga_n\cdot\tw_l\dG \label{GalerkinApprox:Continuity} \\
    &\qquad + \intO2\nu_\Om(\wphi_n)\D\bu^\Om_n:\D\tv_l\dx + \intG2\nu_\Ga(\wpsi_n)\Dg\bu^\Ga_n:\Dg\tw_l\dG + \intG \gamma(\wphi_n,\wpsi_n)\bu^\Ga_n\cdot\tw_l\dG  \nonumber \\
    &\quad = \intO\wmu_n\Grad\wphi_n\cdot\tv_l\dx + \intG\wtheta_n\Gradg\wpsi_n\cdot\tw_l\dG \nonumber \\
    &\qquad + \frac{\chi(L)}{2}\intG(\beta\sigma^\prime - \rho^\prime)(\beta\wtheta_n - \wmu_n)\bu^\Ga_n\cdot\tw_l\dG, \nonumber
\end{align}
for all $l=1,\ldots,m$. Due to the convergences \eqref{CONV:P*P*}, \eqref{CONV:M*T*}, \eqref{Conv:u_n:strong} and \eqref{Conv:Du_n:strong}, it is straightforward to pass to the limit $n\to\infty$ in all terms.
Furthermore, convergence \eqref{Conv:u_n:strong} implies that the initial condition $(\bu^\Om_\ast,\bu^\Ga_\ast)\vert_{t=0} = \mathbfcal{P}_{\mathbfcal{V}_m}(\bv_0,\bw_0)$ a.e.~in $\Om\times\Ga$ is fulfilled.
In this way, we infer that $(\bu^\Om_\ast,\bu^\Ga_\ast)$ is a solution to system \eqref{GalerkinApprox}-\eqref{GalerkinApprox:IC}. As the solution of system \eqref{GalerkinApprox}-\eqref{GalerkinApprox:IC} is unique, this readily implies $(\bu^\Om_\ast, \bu^\Ga_\ast) = \Lambda(\wv_\ast,\ww_\ast)$. Consequently, using convergence \eqref{Conv:u_n:strong}, we conclude that
\begin{align}
    \Lambda(\wv_n,\ww_n) = (\bu^\Om_n,\bu^\Ga_n)
    \rightarrow 
    (\bu^\Om_\ast, \bu^\Ga_\ast) = \Lambda(\wv_\ast,\ww_\ast)
    \qquad\text{strongly in~} L^2(0,T;\mathbfcal{V}_m).
\end{align}
This proves that the map $\Lambda$ is continuous. 

Finally, we are in a position to apply Schauder's fixed-point theorem, which implies that the map $\Lambda$ has a fixed-point in $\mathbfcal{M}$. Hence, for any $m\in\N$, this ensures the existence of an approximate solution $(\bv_m,\bw_m,\phi_m,\psi_m,\mu_m,\theta_m)$ to the discrete problem \eqref{WF:VW:DISCR}-\eqref{Approx:Problem:IC}.

\subsection{Uniform estimates and passage to the limit} \label{Subsection:Uniform}
Now, we establish bounds on the approximate solutions $(\bv_m,\bw_m,\phi_m,\psi_m,\mu_m,\theta_m)$ which are uniform with respect to $m\in\N$. We start by taking $(\wv,\ww) = (\bv_m,\bw_m)$ in \eqref{WF:VW:DISCR}. Integrating by parts yields 
\begin{align}\label{Test:vw:m}
    &\ddt\intO \Big(\frac12\rho(\phi_m)\abs{\bv_m}^2\dx + \intG \frac12\sigma(\psi_m)\abs{\bw_m}^2 \dG\Big) \nonumber \\
    &\qquad + \intO 2\nu_\Om(\phi_m)\abs{\D\bv_m}^2\dx + \intG 2\nu_\Ga(\psi_m)\abs{\Dg\bw_m}^2\dG + \intG \gamma(\phi_m,\psi_m)\abs{\bw_m}^2\dG \nonumber \\
    &\quad = - \intO \phi_m\Grad\mu_m\cdot\bv_m\dx - \intG \psi_m\Gradg\theta_m\cdot\bw_m\dG
\end{align}
(cf. \eqref{IntParts:Final}).
Next, from \eqref{CCH:Fixed:EnergyInequality}, we already know that
\begin{align}\label{Est:Energy:Free}
    \begin{split}
        &E_{\mathrm{free}}(\phi_m(t),\psi_m(t)) + \int_0^t\intO \abs{\Grad\mu_m}^2\dxs + \int_0^t\intG \abs{\Gradg\theta_m}^2\dGs \\
        &\qquad + \chi(L)\int_0^t\intG (\beta\theta_m - \mu_m)^2\dGs \\
        &\quad \leq E_{\mathrm{free}}(\phi_0,\psi_0) + \int_0^t\intO \phi_m\Grad\mu_m\cdot\bv_m\dxs + \int_0^t\intG \psi_m\Gradg\theta_m\cdot\bw_m\dGs.
    \end{split}
\end{align}
Thus, as $E_{\mathrm{tot}} = E_{\mathrm{kin}} + E_{\mathrm{free}}$, integrating \eqref{Test:vw:m} in time over $[0,t]$ for $0 < t \leq T$, and summing with \eqref{Est:Energy:Free}, we obtain
\begin{align}\label{EnergyDiss:m}
    &E_{\mathrm{tot}}(\bv_m(t),\bw_m(t),\phi_m(t),\psi_m(t)) \nonumber\\
    &\qquad + \int_0^t\intO 2\nu_\Om(\phi_m)\abs{\D\bv_m}^2\dxs + \int_0^t\intG 2\nu_\Ga(\psi_m)\abs{\Dg\bw_m}^2\dGs \nonumber\\
    &\qquad + \int_0^t\intG \gamma(\phi_m,\psi_m)\abs{\bw_m}^2\dGs + \int_0^t\intO \abs{\Grad\mu_m}^2\dxs + \int_0^t\intG \abs{\Gradg\theta_m}^2\dGs \\
    &\qquad + \chi(L)\int_0^t\intG (\beta\theta_m - \mu_m)^2\dGs 
    \nonumber\\[1ex]
    &\quad \leq E_{\mathrm{tot}}(\mathbfcal{P}_{\mathbfcal{V}_m}(\bv_0,\bw_0),\phi_0,\psi_0)
    \nonumber
\end{align}
for all $t\in[0,T]$. By the assumptions on the initial data $(\phi_0,\psi_0)$ and the properties of the projection $\mathbfcal{P}_{\mathbfcal{V}_m}$, we find that
\begin{align}
    \begin{split}
        \label{Est:InitialEnergy:m}
        E_{\mathrm{tot}}(\mathbfcal{P}_{\mathbfcal{V}_m}(\bv_0,\bw_0),\phi_0,\psi_0) &\leq \frac{\rho^\ast + \sigma^\ast}{2}\norm{(\bv_0,\bw_0)}_{\mathbfcal{L}^2}^2 + \frac12\norm{(\phi_0,\psi_0)}_{K,\alpha}^2 \\
        &\qquad + \intO F(\phi_0)\dx + \intG G(\psi_0)\dG \\
        &\leq C.
    \end{split}
\end{align}
Thus, using Lemma~\ref{Lemma:Korn} and the bounds $\abs{\phi_m} < 1$ a.e.~in $Q$ and $\abs{\psi_m} < 1$ a.e.~on $\Sigma$, we readily deduce from \eqref{EnergyDiss:m} that
\begin{align}
    &\norm{(\bv_m,\bw_m)}_{L^\infty(0,T;\mathbfcal{L}^2)} + \norm{(\bv_m,\bw_m)}_{L^2(0,T;\mathbfcal{H}^1)} \leq C, \label{Est:Uniform:vw}\\
    &\norm{(\phi_m,\psi_m)}_{L^\infty(0,T;\mathcal{H}^1)} \leq C, \label{Est:Uniform:pp}\\
    &\norm{(\Grad\mu_m,\Gradg\theta_m)}_{L^2(0,T;\mathbfcal{L}^2)} + \chi(L)^{1/2}\norm{\beta\theta_m - \mu_m}_{L^2(0,T;L^2(\Ga))} \leq C. \label{Est:Uniform:mt}
\end{align}
This allows us to proceed as in the proof of \cite[Theorem~3.4]{Knopf2025} (see also \cite{Giorgini2025}) to derive the uniform estimate
\begin{align}
    \norm{(\phi_m,\psi_m)}_{H^1(0,T;(\mathcal{H}^1)^\prime)}
    + \norm{(F_1^\prime(\phi_m),G_1^\prime(\psi_m))}_{L^2(0,T;\mathcal{L}^2)}
    + \norm{(\mu_m,\theta_m)}_{L^2(0,T;\mathcal{H}^1)}
    \le C.
    \label{Est:Uniform:KS}
\end{align}
Consequently, based on the uniform estimates \eqref{Est:Uniform:vw}-\eqref{Est:Uniform:KS}, the Banach--Alaoglu theorem and the Aubin--Lions--Simon lemma imply the existence of a sextuplet $(\bv,\bw,\phi,\psi,\mu,\theta)$ such that
\begin{alignat}{2}
    (\bv_m,\bw_m) &\rightarrow (\bv,\bw) 
    &&\qquad\text{weakly in~} L^2(0,T;\mathbfcal{H}^1_{0,\Div}), \label{Conv:vw:H^1}
    \nonumber \\
    & 
    &&\qquad\text{weakly-star in~} L^\infty(0,T;\mathbfcal{L}^2), 
    \\
    (\delt\phi_m,\delt\psi_m) 
    &\rightarrow (\delt\phi,\delt\psi) &&\qquad\text{weakly in~} L^2(0,T;(\mathcal{H}^1)^\prime), \label{Conv:delt:pp} 
    \\
    (\phi_m,\psi_m) &\rightarrow (\phi,\psi) 
    &&\qquad\text{weakly-star in~} L^\infty(0,T;\mathcal{H}^1_{K,\alpha}), 
    \nonumber \\
    & &&\qquad\text{strongly in~} C([0,T]; \mathcal{H}^s) \ \text{for any~} s\in[0,1), \label{Conv:PP:Strong:C}
    \\
    (\mu_m,\theta_m) &\rightarrow (\mu,\theta) 
    &&\qquad\text{weakly in~} L^2(0,T;\mathcal{H}^1), \label{Conv:mt:H^1}
    \\
    \beta\theta_m - \mu_m &\rightarrow \beta\theta - \mu
    &&\qquad\text{weakly in~} L^2(0,T;L^2(\Gamma))
    \label{Conv:mt:L^2G}
\end{alignat}
along a subsequence $m\rightarrow\infty$. Additionally, in view of \eqref{Est:Uniform:KS}, we obtain
\begin{alignat}{2}
    (F^\prime(\phi_m), G^\prime(\psi_m)) &\rightarrow (F^\prime(\phi), G^\prime(\psi)) &&\qquad\text{weakly in~} L^2(0,T;\mathcal{L}^2), \label{Conv:Pot:L^2}
\end{alignat}
as $m\rightarrow\infty$. Here, we made use of the fact that $F_1^\prime$ and $G_1^\prime$ are maximal monotone operators to identify the limit functions (cf.~Subsection~\ref{SUBSEC:EXAPP} or \cite{Knopf2025} for more details). From \eqref{Conv:Pot:L^2} we immediately deduce that
\begin{align}
    \abs{\phi} < 1 \quad\text{a.e.~in~}Q \quad\text{and}\quad \abs{\psi} < 1 \quad\text{a.e.~on~}\Sigma.
\end{align}
In view of these convergences, it is straightforward to pass to the limit $m\rightarrow\infty$ in the weak formulations \eqref{WF:PP:DISCR} and \eqref{WF:MT:DISCR}. This shows that the weak formulations \eqref{WF:PP} and \eqref{WF:MT} are fulfilled.

To verify the initial condition $(\phi,\psi)\vert_{t=0} = (\phi_0,\psi_0)$ a.e.~in $\Om\times\Ga$, we simply notice that
\begin{align*}
    (\phi_m(0),\psi_m(0)) = (\phi_0,\psi_0) \qquad\text{a.e.~in~}\Om\times\Ga
\end{align*}
for all $m\in\N$ as well as
\begin{align*}
    (\phi_m(0),\psi_m(0)) \rightarrow (\phi(0),\psi(0)) \qquad\text{strongly in~}\mathcal{L}^2
\end{align*}
as $m\rightarrow\infty$ because of \eqref{Conv:PP:Strong:C}.\pagebreak[1]

Next, to show that \eqref{WF:VW} holds, we rewrite \eqref{WF:VW:DISCR} again in the conservative form (see \eqref{eqs:NSCH:NC}) and integrate with respect to time. We get
\begin{align}\label{NonCons:Time:m}
    \begin{split}
        &\int_0^T\intO \delt(\rho(\phi_m)\bv_m)\cdot\wv\dxt + \int_0^T\intG \delt(\sigma(\psi_m)\bw_m)\cdot\ww\dGt \\
        &\qquad - \int_0^T\intO \rho(\phi_m)\bv_m\otimes\bv_m : \Grad\wv\dxt - \int_0^T\intG \sigma(\psi_m)\bw_m\otimes\bw_m : \Gradg\ww\dGt \\
        &\qquad + \int_0^T\intO 2\nu_\Om(\phi_m)\D\bv_m:\D\wv\dxt + \int_0^T\intG 2\nu_\Ga(\psi_m)\Dg\bw_m:\Dg\ww\dGt \\
        &\qquad + \int_0^T\intG \gamma(\phi_m, \psi_m)\bw_m\cdot\ww\dGt - \int_0^T\intO \bv_m\otimes \J_m:\Grad\wv\dxt \\
        &\qquad - \int_0^T\intG \bw_m\otimes\K_m : \Gradg\ww\dGt \\
        &\quad = \int_0^T\intO \mu_m\Grad\phi_m\cdot\wv\dxt + \int_0^T\intG \theta_m\Gradg\psi_m\cdot\ww\dGt \\
        &\qquad + \frac{\chi(L)}{2}\int_0^T\intG (\beta\sigma^\prime + \rho^\prime)(\beta\theta_m - \mu_m)\bw_m\cdot\ww\dGt
    \end{split}
\end{align}
for all $(\wv,\ww)\in C_c^\infty(0,T;\mathbfcal{V}_m)$, where
$(\J_m,\K_m) = (-\rho^\prime\,\Grad\mu_m, -\sigma^\prime\,\Gradg\theta_m)$.     
Now, proceeding similarly to the line of argument in \cite[Section 5]{Abels2013}), we use the uniform estimates \eqref{Est:Uniform:vw}-\eqref{Est:Uniform:KS} along with standard interpolation results to show that
\begin{align}\label{Est:Vel:AGD}
    \begin{split}
        \{(\rho(\phi_m)\bv_m\otimes\bv_m, \sigma(\psi_m)\bw_m\otimes\bw_m)\}_{m\in\N}
        &\ \text{~is uniformly bounded in~} \ L^2(0,T;\mathbfcal{L}^{3/2}), 
        \\
        \{(\bv_m\otimes\Grad\mu_m, \bw_m\otimes\Gradg\theta_m)\}_{m\in\N} 
        &\ \text{~is uniformly bounded in~} \ L^{8/7}(0,T;\mathbfcal{L}^{4/3}), 
        \\
        \{(\mu_m\Grad\phi_m, \theta_m\Gradg\psi_m)\}_{m\in\N}
        &\ \text{~is uniformly bounded in~} \ L^2(0,T;\mathbfcal{L}^{3/2}), 
        \\
        \{(\beta\theta_m - \mu_m)\bw_m\}_{m\in\N}
        &\ \text{~is uniformly bounded in~} \ L^{8/7}(0,T;\mathbf{L}^{4/3}(\Ga)).
    \end{split}
\end{align}
Moreover, via the trace embedding $H^1(\Omega) \emb L^2(\Gamma)$, we infer from \eqref{Est:Uniform:vw} that
\begin{align}\label{Est:Vel:AGD:2}
    \begin{split}
        \{(\D\bv_m,\Dg\bw_m)\}_{m\in\N}
        &\ \text{~is uniformly bounded in~} L^2(0,T;\mathbfcal{L}^2), 
        \\
        \{\bw_m\}_{m\in\N}
        &\ \text{~is uniformly bounded in~} L^2(0,T;\mathbf{L}^2(\Ga)).
    \end{split}
\end{align}
In particular, this means that all terms appearing in \eqref{Est:Vel:AGD} and \eqref{Est:Vel:AGD:2} are uniformly bounded in $L^{8/7}(0,T;\mathbfcal{L}^{4/3})$ or $L^{8/7}(0,T;\mathbf{L}^{4/3}(\Ga))$, respectively. 

Let now $(\wv,\ww)\in L^8(0,T;\mathbfcal{W}^{1,4}_{0,\Div})$ be arbitrary. Then $\mathbfcal{P}_{\mathbfcal{V}_m}(\wv,\ww)$ is an admissible test function in \eqref{NonCons:Time:m}. Hence, by a comparison argument, we infer that $\{\delt\mathbfcal{P}_{\mathbfcal{V}_m}(\rho(\phi_m)\bv_m,\sigma(\psi_m)\bw_m)\}_{m\in\N}$ is bounded in $L^{8/7}(0,T;(\mathbfcal{W}^{1,4}_{0,\Div})^\prime)$. In addition, we know that $\{\mathbfcal{P}_{\mathbfcal{V}_m}(\rho(\phi_m)\bv_m,\sigma(\psi_m)\bw_m)\}_{m\in\N}$ is bounded in $L^2(0,T;\mathbfcal{H}^1_{0,\Div})$. Therefore, according to the Aubin--Lions lemma, there exists $(\bv^\star,\bw^\star)\in L^\infty(0,T;\mathbfcal{L}^2)$ such that
\begin{align}\label{Conv:PVM:Strong}
    \mathbfcal{P}_{\mathbfcal{V}_m}(\rho(\phi_m)\bv_m,\sigma(\psi_m)\bw_m)\rightarrow (\bv^\star,\bw^\star) \qquad\text{strongly in~} L^2(0,T;\mathbfcal{L}^2),
\end{align}
as $m\rightarrow\infty$, along a non-relabeled subsequence. 
In order to show that $(\bv^\star,\bw^\star) = \mathbfcal{P}_\Div(\rho(\phi)\bv,\sigma(\psi)\bw))$, we first establish the convergence 
\begin{align}\label{Conv:rhovsigmaw:L^2}
(\rho(\phi_m)\bv_m,\sigma(\psi_m)\bw_m) \rightarrow (\rho(\phi)\bv,\sigma(\psi)\bw) \qquad\text{weakly in~} L^2(0,T;\mathbfcal{L}^2)
\end{align}
as $m\rightarrow\infty$. To this end, let 
$(\wv,\ww)\in L^2(0,T;\mathbfcal{L}^2)$ be arbitrary. 
Then
\begin{align}\label{WeakConv:Rho_mv_m}
    &\int_0^T\intO (\rho(\phi_m)\bv_m - \rho(\phi)\bv)\cdot\wv\dxt + \int_0^T\intG (\sigma(\psi_m)\bw_m - \sigma(\psi)\bw)\cdot\ww\dGt 
    \nonumber\\
    \begin{split}
    &\quad = \int_0^T\intO (\rho(\phi_m) - \rho(\phi))\wv\cdot\bv_m\dxt + \int_0^T\intO \rho(\phi)(\bv_m - \bv)\cdot\wv\dxt 
    \\
    &\qquad + \int_0^T\intG (\sigma(\psi_m) - \sigma(\psi))\ww\cdot\bw_m\dGt + \int_0^T \intG \sigma(\psi)(\bw_m - \bw)\cdot\ww\dGt    
    \\[1ex]
    &\quad\longrightarrow 0 
    \end{split}
\end{align}
as $m\rightarrow\infty$. Here, we have used the strong convergence $((\rho(\phi_m) - \rho(\phi))\wv,(\sigma(\psi_m) - \sigma(\psi))\ww) \rightarrow 0$ in $L^2(0,T;\mathbfcal{L}^2)$ as $m\rightarrow\infty$ as well as the weak convergence \eqref{Conv:vw:H^1}. 
This means that \eqref{Conv:rhovsigmaw:L^2} is established.
For any $(\wv,\ww)\in L^2(0,T;\mathbfcal{L}^2)$, we now use the convergence \eqref{Conv:rhovsigmaw:L^2} along with the strong convergence $\mathbfcal{P}_{\mathbfcal{V}_m}(\wv,\ww)\rightarrow \mathbfcal{P}_\Div(\wv,\ww)$ in $L^2(0,T;\mathbfcal{L}^2)$ (see Corollary~\ref{Cor:Proj:Stokes})
to deduce that
\begin{align}
    \begin{split}
        &\int_0^T\intO \bv^\star\cdot\wv\dxt + \int_0^T\intG \bw^\star\cdot\ww\dGt \\
        &\qquad = \lim_{m\rightarrow\infty} \Big( \int_0^T\intO \mathbf{P}_{\mathbfcal{V}_m}^\Om(\rho(\phi_m)\bv_m)\cdot\wv\dxt + \int_0^T \intG \mathbf{P}_{\mathbfcal{V}_m}^\Ga(\sigma(\psi_m)\bw_m)\cdot\ww\dGt\Big) \\
        &\qquad = \lim_{m\rightarrow\infty} \Big( \int_0^T\intO \rho(\phi_m)\bv_m\cdot\mathbf{P}_{\mathbfcal{V}_m}^\Om(\wv)\dxt + \int_0^T\intG \sigma(\psi_m)\bw_m\cdot\mathbf{P}_{\mathbfcal{V}_m}^\Ga(\ww)\dGt\Big) \\
        &\qquad = \int_0^T\intO \rho(\phi)\bv\cdot\mathbf{P}_\Div^\Om(\wv)\dxt + \int_0^T\intG \sigma(\psi)\bw\cdot\mathbf{P}_\Div^\Ga(\ww)\dGt \\
        &\qquad = \int_0^T \intO \mathbf{P}_\Div^\Om(\rho(\phi)\bv)\cdot\wv\dxt + \int_0^T \intG \mathbf{P}_\Div^\Ga(\sigma(\psi)\bw)\cdot\ww\dGt.
    \end{split}
\end{align}
This proves that $(\bv^\star,\bw^\star) = \mathbfcal{P}_\Div(\rho(\phi)\bv,\sigma(\psi)\bw))$. 

Next, due to \eqref{Conv:vw:H^1} and \eqref{Conv:PVM:Strong}, we have
\begin{align*}
    &\int_0^T \intO \rho(\phi_m)\abs{\bv_m}^2\dxt + \int_0^T\intG \sigma(\psi_m)\abs{\bw_m}^2\dGt 
    \nonumber \\
    \begin{split}
    &\quad = \int_0^T \intO \mathbf{P}_{\mathbfcal{V}_m}^\Om(\rho(\phi_m)\bv_m)\cdot\bv_m\dxt + \int_0^T\intG \mathbf{P}_{\mathbfcal{V}_m}^\Ga(\sigma(\psi_m)\bw_m)\cdot\bw_m\dGt 
    \\
    &\quad\longrightarrow \int_0^T \intO \mathbf{P}_\Div^\Om(\rho(\phi)\bv)\cdot\bv\dxt + \int_0^T \intG \mathbf{P}_\Div^\Ga(\sigma(\psi)\bw)\cdot\bw\dGt 
    \end{split}
    \\
    &\quad = \int_0^T\intO \rho(\phi)\abs{\bv}^2\dxt + \int_0^T\intG \sigma(\psi)\abs{\bw}^2\dGt,
    \nonumber
\end{align*}
as $m\rightarrow\infty$. Thus, arguing analogously to \eqref{Conv:rhovsigmaw:L^2}, we conclude that
\begin{align*}
    (\rho(\phi_m)^{1/2}\bv_m,\sigma(\psi_m)^{1/2}\bw_m)\rightarrow (\rho(\phi)^{1/2}\bv,\sigma(\psi)^{1/2}\bw)
    \qquad\text{strongly in~}
    L^2(0,T;\mathbfcal{L}^2).
\end{align*}

Lastly, since $\rho(\phi_m)\rightarrow\rho(\phi)$ almost everywhere in $Q$, and $\rho \geq \rho_\ast > 0$, we use Lebesgue's dominated convergence theorem to deduce that
\begin{align}\label{Conv:v:L^2:strong}
    \bv_m = \rho(\phi_m)^{-1/2}(\rho(\phi_m)^{1/2}\bv_m)\rightarrow\bv \qquad\text{strongly in~} L^2(0,T;\mathbf{L}^2(\Om))
\end{align}
as $m\rightarrow\infty$. In a similar manner, we conclude that
\begin{align}\label{Conv:w:L^2:strong}
    \bw_m = \sigma(\psi_m)^{-1/2}(\sigma(\psi_m)\bw_m)\rightarrow\bw \qquad\text{strongly in~} L^2(0,T;\mathbf{L}^2(\Ga))
\end{align}
as $m\rightarrow\infty$.

Finally, in order to recover the weak formulation \eqref{WF:VW}, 
let $\pi\in C^\infty_0(0,T)$ be arbitrary, and for any $j\in\N$, let $(\wv,\ww) = (\pi\tv_j,\pi\tw_j)$. Then, for $m\geq j$, we infer from \eqref{NonCons:Time:m} that
\begin{align}\label{WF:WV:pi:m}
    &\int_0^T\intO \delt(\rho(\phi_m)\bv_m)\cdot\pi\tv_j\dxt + \int_0^T\intG \delt(\sigma(\psi_m)\bw_m)\cdot\pi\tw_j\dGt \nonumber \\
    &\qquad - \int_0^T\intO \rho(\phi_m)\bv_m\otimes\bv_m:\Grad\tv_j\pi\dxt - \int_0^T\intG \sigma(\psi_m)\bw_m\otimes\bw_m:\Gradg\tv_j\pi\dGt \nonumber \\
    &\qquad + \int_0^T\intO 2\nu_\Om(\phi_m)\D\bv_m:\D\tv_j\pi\dxt + \int_0^T\intG 2\nu_\Ga(\psi_m)\Dg\bw_m:\Dg\tv_j\pi\dGt \nonumber \\
    &\qquad + \int_0^T\intG \gamma(\phi_m,\psi_m)\bw_m\cdot\tw_j\pi\dGt - \int_0^T \intO \bv_m\otimes\J_m:\Grad\tv_j\pi\dxt \\
    &\qquad - \int_0^T\intG \bw_m\otimes\K_m:\Gradg\tw_j\pi\dGt \nonumber \\
    &\quad = \int_0^T\intO \mu_m\Grad\phi_m:\tv_j\pi\dxt + \int_0^T\intG \theta_m\Gradg\psi_m\cdot\tw_j\pi\dGt \nonumber \\
    &\qquad + \frac{\chi(L)}{2}\int_0^T\intG(\beta\sigma^\prime + \rho^\prime)(\beta\theta_m - \mu_m)\bw_m\cdot\tw_j\pi\dGt. \nonumber
\end{align}
For the first two terms, we use integration by parts together with \eqref{Conv:rhovsigmaw:L^2} to deduce
\begin{align*}
&\int_0^T\intO \delt(\rho(\phi_m)\bv_m)\cdot\tv_j\pi\dxt + \int_0^T\intG\delt(\sigma(\psi_m)\bw_m)\cdot\tw_j\pi\dGt \\
&\quad = - \int_0^T\intO \rho(\phi_m)\bv_m\cdot\delt(\pi\tv_j)\dxt - \int_0^T\intG \sigma(\psi_m)\bw_m\cdot\delt(\pi\tw_j)\dGt \\
&\quad\rightarrow - \int_0^T\intO \rho(\phi)\bv\cdot\delt(\pi\tv_j)\dxt - \int_0^T\intG \sigma(\psi)\bw\cdot\delt(\pi\tw_j)\dGt
\end{align*}
as $m\rightarrow\infty$.\pagebreak[1] 

Moreover, recalling the boundedness of $\{(\bv_m,\bw_m)\}_{m\in\N}$ in $L^2(0,T;\mathbfcal{H}^1)\emb L^2(0,T;\mathbfcal{L}^6)$, the continuous embedding $L^2(0,T;\mathbfcal{L}^2)\emb L^2(0,T;\mathbfcal{L}^s)$ for $s\in [1,2]$, and the strong convergences \eqref{Conv:v:L^2:strong} and \eqref{Conv:w:L^2:strong}, we infer via an interpolation argument that
\begin{align*}
    (\bv_m,\bw_m)\rightarrow (\bv,\bw) \qquad\text{strongly in~} L^2(0,T;\mathbfcal{L}^{6-r})
\end{align*}
as $m\rightarrow\infty$ for all $0 < r \leq 5$. 
Furthermore, since $\{(\bv_m,\bw_m)\}_{m\in\N}$ is also bounded in $L^\infty(0,T;\mathbfcal{L}^2)$, an interpolation argument reveals that for any $1 \leq q < \infty$, there exists $\tau = \tau(q,r) > 0$ such that
\begin{align*}
    (\bv_m,\bw_m)\rightarrow (\bv,\bw) \qquad\text{strongly in~} L^q(0,T;\mathbfcal{L}^{2+\tau})
\end{align*}
as $m\rightarrow\infty$. To be precise, it holds that
\begin{align*}
2 + \tau = \frac{2(6-r)}{\frac{2}{q}(r - 4) + 6 - r}.
\end{align*}
In particular, choosing $r = 3$ yields
\begin{align}\label{Conv:vw:Strong:L^2L^3}
    (\bv_m,\bw_m) \rightarrow (\bv,\bw) \qquad\text{strongly in~} L^2(0,T;\mathbfcal{L}^3)
\end{align}
as $m\rightarrow\infty$. Then, for $q = 4$, we obtain $2 + \tau = 12/5$, and hence,
\begin{align}\label{Conv:vw:Strong:L^4L^12/5}
    (\bv_m,\bw_m) \rightarrow (\bv,\bw) \qquad\text{strongly in~} L^4(0,T;\mathbfcal{L}^{12/5})
\end{align}
as $m\rightarrow\infty$. 
Next, we recall that for any $j\in\N$, the eigenfunction $(\tv_j,\tw_j)$ belongs to $\mathbfcal{H}^2 \emb \mathbfcal{L}^{d+1}$ (cf.~Theorem~\ref{THM:SPEC:BSS}). Hence, because of Theorem~\ref{App:Theorem:BSStokes:Reg}, we even have $(\tv_j,\tw_j) \in \mathbfcal{W}^{2,d+1} \emb \mathbfcal{W}^{1,\infty}$.
This allows us to use Lebesgue's dominated convergence theorem to obtain the convergence 
\begin{align*}
    \big(\rho(\phi_m)\Grad\tv_j\pi,\sigma(\psi_m)\Gradg\tw_j\pi\big)\rightarrow \big(\rho(\phi)\Grad\tv_j\pi,\sigma(\psi)\Gradg\tw_j\pi\big) \qquad\text{strongly in ~} L^4(0,T;\mathbfcal{L}^{12})
\end{align*}
as $m\rightarrow\infty$. Consequently, in combination with \eqref{Conv:vw:H^1} and \eqref{Conv:vw:Strong:L^4L^12/5}, we obtain
\begin{align*}
    & -\int_0^T\intO \rho(\phi_m)\bv_m\otimes\bv_m:\Grad\tv_j\pi\dxt - \int_0^T\intG \sigma(\psi_m)\bw_m\otimes\bw_m:\Gradg\tw_j\pi\dGt \\
    &\qquad\rightarrow -\int_0^T\intO \rho(\phi)\bv\otimes\bv:\Grad\tv_j\pi\dxt - \int_0^T \intG \sigma(\psi)\bw\otimes\bw:\Gradg\tw_j\pi\dGt
\end{align*}
as $m\rightarrow\infty$. Next, we employ the continuity and boundedness of $\nu_\Om,\nu_\Ga$ (see \ref{Assumption:Coefficients}) together with Lebesgue's dominated convergence theorem to infer
\begin{align*}
    (2\nu_\Om(\phi_m)\D\tv_j\pi,2\nu_\Ga(\psi_m)\Dg\tw_j\pi) \rightarrow (2\nu_\Om(\phi)\D\tv_j\pi, 2\nu_\Ga(\psi)\Dg\tw_j\pi)
\end{align*}
strongly in $L^2(0,T;\mathbfcal{L}^2)$ as $m\rightarrow\infty$. As $(\D\bv_m,\Dg\bw_m)\rightarrow (\D\bv,\Dg\bw)$ weakly in $L^2(0,T;\mathbfcal{L}^2)$ as $m\rightarrow\infty$, we deduce
\begin{align*}
    &\int_0^T\intO 2\nu_\Om(\phi_m)\D\bv_m:\D\tv_j\pi\dxt + \int_0^T\intG 2\nu_\Ga(\psi_m)\Dg\bw_m:\Dg\tw_j\pi\dGt \\
    &\qquad\rightarrow \int_0^T\intO 2\nu_\Om(\phi)\D\bv:\D\tv_j\pi\dxt + \int_0^T\intG 2\nu_\Ga(\psi)\Dg\bw:\Dg\tw_j\pi\dGt
\end{align*}
as $m\rightarrow\infty$. Furthermore, arguing as above, we find that $\gamma(\phi_m,\psi_m)\tw_j\pi\rightarrow \gamma(\phi,\psi)\tw_j\pi$ strongly in $L^2(0,T;\mathbf{L}^2(\Ga))$ as $m\rightarrow\infty$. Thus, since $\bw_m\rightarrow\bw$ strongly in $L^2(0,T;\mathbf{L}^2(\Ga))$ as $m\rightarrow\infty$, this yields
\begin{align*}
    \int_0^T\intG \gamma(\phi_m,\psi_m)\bw_m\cdot\tw_j\pi\dGt \rightarrow \int_0^T\intG \gamma(\phi,\psi)\bw\cdot\tw_j\dGt
\end{align*}
as $m\rightarrow\infty$. Next, due to \eqref{Conv:mt:H^1}, we find that $(\J_m,\K_m)\rightarrow(\J,\K)$ weakly in $L^2(0,T;\mathbfcal{L}^2)$ as $m\rightarrow\infty$, where $(\J,\K) = (-\rho^\prime\,\Grad\mu, - \sigma^\prime\,\Gradg\theta)$.
Thus, using \eqref{Conv:v:L^2:strong}-\eqref{Conv:w:L^2:strong} along with the regularity $(\tv_j,\tw_j) \in \mathbfcal{W}^{1,\infty}$ that was already shown above, we infer that
\begin{align*}
    &-\int_0^T\intO \bv_m\otimes\J_m:\Grad\tv_j\pi\dxt - \int_0^T\intG \bw_m\otimes\K_m:\Gradg\tv_j\pi\dGt \\
    &\qquad\rightarrow -\int_0^T\intO \bv\otimes\J:\Grad\tv_j\pi\dxt - \int_0^T\intG \bw\otimes\K:\Gradg\tw_j\pi\dGt
\end{align*}
as $m\rightarrow\infty$.
For the first two terms on the right-hand side of \eqref{WF:WV:pi:m}, we deduce that
\begin{align*}
    &-\int_0^T\intO \phi_m\Grad\mu_m\cdot\tv_j\pi\dxt - \int_0^T\intG \psi_m\Gradg\theta_m\cdot\tw_j\pi\dGt \\
    &\qquad\rightarrow -\int_0^T\intO \phi\Grad\mu\cdot\tv_j\pi\dxt - \int_0^T\intG \psi\Gradg\theta\cdot\tw_j\pi\dGt
\end{align*}
as $m\rightarrow\infty$, due to the convergences \eqref{Conv:PP:Strong:C} and \eqref{Conv:mt:H^1}. Finally, for the last term on the right-hand side of \eqref{WF:WV:pi:m}, we obtain
\begin{align*}
    &\frac{\chi(L)}{2}\int_0^T\intG (\beta\sigma^\prime + \rho^\prime)(\beta\theta_m - \mu_m)\bw_m\cdot\tw_j\pi\dGt \\
    &\qquad\rightarrow \frac{\chi(L)}{2}\int_0^T\intG (\beta\sigma^\prime + \rho^\prime)(\beta\theta - \mu)\bw\cdot\tw_j\pi\dGt
\end{align*}
as $m\rightarrow\infty$, due to the convergences \eqref{Conv:mt:L^2G} and \eqref{Conv:w:L^2:strong}. 

Altogether, we can now pass to the limit $m\rightarrow\infty$ in \eqref{WF:WV:pi:m}. In this way, we conclude that
\begin{align*}
    &\int_0^T\intO \rho(\phi)\bv\cdot \delt(\pi\tv_j) \dxt + \int_0^T\intG \sigma(\psi)\bw\cdot \delt(\pi\tw_j)\dGt \\
    &\qquad - \int_0^T\intO \rho(\phi)\bv\otimes\bv:\Grad\wv\pi\dxt - \int_0^T \sigma(\psi)\bw\otimes\bw:\Gradg\tw_j\pi\dGt \\
    &\qquad + \int_0^T\intO 2\nu_\Om(\phi)\D\bv:\D\tv_j\pi\dxt + \int_0^T\intG 2\nu_\Ga(\psi)\Dg\bw:\Dg\tw_j\pi\dGt \\
    &\qquad + \int_0^T\intG \gamma(\phi,\psi)\bw\cdot\tw_j\pi\dGt - \int_0^T\intO \bv\otimes\J:\Grad\tv_j\pi\dxt \\
    &\qquad - \int_0^T\intG \bw\otimes\K:\Gradg\tw_j\pi\dGt \\
    &\quad = \int_0^T\intO \mu\Grad\phi\cdot\tv_j\pi\dxt + \int_0^T\intG \theta\Gradg\psi\cdot\tw_j\pi\dGt \\
    &\qquad + \frac{\chi(L)}{2}\int_0^T\intG (\beta\sigma^\prime + \rho^\prime)(\beta\theta_m - \mu_m)\bw\cdot\tw_j\pi\dGt
\end{align*}
for all $j\in\N$ and $\pi\in C_c^\infty(0,T)$. 
Since $\,\mathrm{span}\,\{(\tv_j,\tw_j):j\in\N\}$ is dense in $\mathbfcal{H}^2_{0,\Div}$ (cf.~Corollary~\ref{COR:ONS}), we know that
\begin{align*}
    \left\{
    \sum_{j=1}^N \pi_j (\tv_j,\tw_j) 
    \;\middle|\;
    N\in\N, \, 
    \pi_j \in C_c^\infty(0,T) \;\text{for}\; j=1,...,N
    \right\}
    \subset C_c^\infty((0,T);\mathbfcal{H}^2_{0,\Div})
\end{align*}
is dense. Consequently, the sextuplet $(\bv,\bw,\phi,\psi,\mu,\theta)$ satisfies the weak formulation \eqref{WF:VW}-\eqref{WF:MT}.

\subsection{Verification of the initial condition for \texorpdfstring{$(\bv,\bw)$}{v,w}}
In this step, we show that the initial conditions for $(\bv,\bw)$ are fulfilled. To this end, we start with the following lemma, which has been adapted from \cite[Section 5.2]{Abels2013} to the current setting.

\begin{lemma}\label{Lemma:IC}
    Let $(\bv,\bw), (\tv,\tw)\in\mathbfcal{H}^1_{0,\Div}$ and $(\ov\rho,\ov\sigma)\in\mathcal{L}^\infty$ with $\ov\rho \geq \ov\rho_0 > 0$ and $\ov\sigma \geq \ov\sigma_0 > 0$ such that
    \begin{align}
        \label{ID:IC}
        \intO \ov\rho\, \bv\cdot\wv\dx + \intG \ov\sigma\, \bw\cdot\ww\dG = \intO \ov\rho\tv\cdot\wv\dx + \intG \ov\sigma\tw\cdot\ww\dG
    \end{align}
    for all $(\wv,\ww)\in\mathbfcal{H}^2_{0,\Div}$. Then it holds $(\bv,\bw) = (\tv,\tw)$ a.e.~in $\Om\times\Ga$.
\end{lemma}

\begin{proof}
    According to Corollary~\ref{COR:DENS}, there exists a sequence $\{(\bv_N,\bw_N)\}_{N\in\N} \subset \mathbfcal{H}^2_{0,\Div}$ such that
    \begin{align*}
        (\bv_N,\bw_N) \to (\bv - \tv,\bw - \tw) 
        \quad\text{strongly in $\mathbfcal{L}^2$}
    \end{align*}
    as $N\to \infty$. We now test the difference of \eqref{ID:IC} written for $(\bv,\bw)$ and \eqref{ID:IC} written for $(\tv,\tw)$ with $(\bv_N,\bw_N)$. Passing to the limit $N\to\infty$, we conclude that   
    \begin{align*}
        \intO \ov\rho\abs{\bv-\tv}^2\dx + \intG \ov\sigma\abs{\bw-\tw}^2\dG = 0.
    \end{align*}
    Since $\ov\rho$ and $\ov\sigma$ are uniformly positive, we infer $\bv = \tv$ in $\mathbf{L}^2(\Om)$ and $\bw = \tw$ in $\mathbf{L}^2(\Ga)$. This proves the claim.
\end{proof}

At first, we show the weak continuity in time for the projection of $(\rho(\phi)\bv,\sigma(\psi)\bw)$ onto the divergence-free vector fields. To this end, we consider $(\bv_m^\star,\bw_m^\star) = \mathbfcal{P}_{\mathbfcal{V}_m}(\rho_m(\phi_m)\bv_m,\sigma_m(\psi_m)\bw_m)$, and deduce from the arguments above that
\begin{align}
    \label{BND:V*W*:1}
    \{(\bv_m^\star,\bw_m^\star)\}_{m\in\N} 
    &\quad\text{is bounded in}\quad W^{1,8/7}(0,T;(\mathbfcal{W}^{1,4}_{0,\Div})^\prime)\emb C([0,T];(\mathbfcal{W}^{1,4}_{0,\Div})^\prime), 
    \\
    \label{BND:V*W*:2}
    \{(\bv_m^\star,\bw_m^\star)\}_{m\in\N} 
    &\quad\text{is bounded in}\quad L^\infty(0,T;\mathbfcal{L}^2).
\end{align}
In view of these bounds, we already showed in \eqref{Conv:PVM:Strong} that $(\bv_m^\star,\bw_m^\star)\rightarrow (\bv^\star,\bw^\star)$ strongly in $L^2(0,T;\mathbfcal{L}^2)$ as $m\rightarrow\infty$.
This implies that
\begin{align}
    \label{CONV:V*W*:1}
    \begin{split}
    (\bv_m^\star,\bw_m^\star) \to (\bv^\star,\bw^\star)
    &\quad\text{weakly in $W^{1,8/7}(0,T;(\mathbfcal{W}^{1,4}_{0,\Div})^\prime)$}
    \\
    &\quad\text{and weakly-star in $L^\infty(0,T;\mathbfcal{L}^2)$}
    \end{split}
\end{align}
as $m\rightarrow\infty$. These convergences are first obtained after subsequence extraction, but as the limit is independent of the extracted subsequence, the convergences hold for the whole sequence.
In particular, using Lemma~\ref{Prelim:Lemma:BC_w}, we infer that $(\bv^\star,\bw^\star)\in C_w([0,T];\mathbfcal{L}^2)$.
Moreover, in view of the continuous embedding stated in \eqref{BND:V*W*:1}, we infer that the evaluation operator
\begin{align*}
    \mathrm{tr_0}: W^{1,8/7}(0,T;(\mathbfcal{W}^{1,4}_{0,\Div})^\prime)
    \to (\mathbfcal{W}^{1,4}_{0,\Div})^\prime,
    \quad
    (\tv,\tw) \mapsto \big(\tv(0),\tw(0)\big)
\end{align*}
is continuous. Consequently, in combination with the weak convergence \eqref{CONV:V*W*:1}, we obtain
\begin{align}
    \label{CONV:V*W*:2}
    \big(\bv_m^\star(0),\bw_m^\star(0)\big) \to \big(\bv^\star(0),\bw^\star(0)\big)
    \quad\text{weakly in $(\mathbfcal{W}^{1,4}_{0,\Div})^\prime$}
\end{align}
as $m\rightarrow\infty$.

Now, for any $t\in [0,T]$, we consider the auxiliary problems
\begin{subequations}
\label{AuxProblem:Bulk}
\begin{align}
    \label{AuxProblem:Bulk:1}
    -\Div\left(\frac{1}{\rho(\phi(t))}\Grad p(t)\right) &= \Div\left(\frac{1}{\rho(\phi(t))}\bv^\star(t)\right) &&\qquad\text{in~}\Om, 
    \\
    \label{AuxProblem:Bulk:2}
    \deln p(t) &= 0 &&\qquad\text{on~}\Ga,
\end{align}
\end{subequations}
and
\begin{align}\label{AuxProblem:Surface}
    -\Divg\left(\frac{1}{\sigma(\psi(t))}\Gradg q(t)\right)  = \Divg\left(\frac{1}{\sigma(\psi(t))}\bw^\star(t)\right) \qquad\text{on~}\Ga.
\end{align}
The corresponding weak formulations read, for all $t\in[0,T]$, as
\begin{align}\label{AuxProblem:Bulk:WF}
    \intO \frac{1}{\rho(\phi(t))}\Grad p(t)\cdot \Grad\zeta\dx = -\intO \frac{1}{\rho(\phi(t))}\bv^\star(t)\cdot\Grad\zeta\dx
\end{align}
for all $\zeta\in H^1(\Om)$ and
\begin{align}\label{AuxProblem:Surface:WF}
    \intG \frac{1}{\sigma(\psi(t))}\Gradg q(t)\cdot\Gradg\xi \dG = - \intG \frac{1}{\sigma(\psi(t))}\bw^\star(t)\cdot\Gradg\xi\dG
\end{align}
for all $\xi\in H^1(\Ga)$, respectively. By the Lax--Milgram theorem, we find unique weak solutions $p\in L^2(0,T;H^1(\Om)\cap L^2_{(0)}(\Om))$ and $q\in L^2(0,T;H^1(\Ga)\cap L^2_{(0)}(\Ga))$, which satisfy
\begin{align}\label{Est:AuxProblem:Bulk}
    \norm{\Grad p(t)}_{\mathbf{L}^2(\Om)}\leq C\norm{\bv^\star(t)}_{\mathbf{L}^2(\Om)}
\end{align}
and
\begin{align}\label{Est:AuxProblem:Surface}
    \norm{\Grad q(t)}_{\mathbf{L}^2(\Ga)} \leq C\norm{\bw^\star(t)}_{\mathbf{L}^2(\Ga)}
\end{align}
with a constant $C > 0$ independent of $t$, respectively. Here, we used the fact that both $\rho\in L^\infty(Q)$ and $\sigma\in L^\infty(\Sigma)$ are bounded from above and from below by positive constants.

We now claim that $(\Grad p,\Gradg q)\in C_w([0,T];\mathbfcal{L}^2)$. As the equations for $p$ and $q$ are completely uncoupled, we only prove that $\Grad p\in C_w([0,T];\mathbf{L}^2(\Om))$, while the other assertion follows by similar arguments. To prove the claim, let $t\in[0,T]$ be arbitrary, and let $\{t_n\}_{n\in\N}$ be an arbitrary sequence with $t_n\rightarrow t$ as $n\rightarrow\infty$. From $\bv^\star\in C_w([0,T];\mathbf{L}^2(\Om))$ we infer in particular that $\{\bv^\star(t_n)\}_{n\in\N}$ is bounded in $\mathbf{L}^2(\Om)$. Thus, due to \eqref{Est:AuxProblem:Bulk}, also $\{\Grad p(t_n)\}_{n\in\N}$ is bounded in $\mathbf{L}^2(\Om)$.
Since $\{\Grad \tilde p : \tilde p \in H^1(\Om)\cap L^2_{(0)}(\Om)\}$ is a closed subspace of $\mathbf{L}^2(\Om)$, we infer that there exists $\tilde{p}\in H^1(\Om)\cap L^2_{(0)}(\Om)$ such that, up to subsequence extraction, 
$\Grad p(t_{n})\rightarrow \Grad \tilde{p}$ weakly in $\mathbf{L}^2(\Om)$ as $n\rightarrow\infty$. 
By passing to the limit in the weak formulation \eqref{Est:AuxProblem:Bulk} written at $t_n$, we notice that $\tilde{p}$ is a weak solution to \eqref{AuxProblem:Bulk:WF}. Because of uniqueness, this readily yields $\Grad \tilde{p} = \Grad p(t)$. In particular, this means that the weak limit is independent of the extracted subsequence. Thus, the convergence $\Grad p(t_{n})\rightarrow \Grad \tilde{p}$ weakly in $\mathbf{L}^2(\Om)$, as $n\rightarrow\infty$, remains true for the whole sequence.
This shows that $ \Grad p\in C_w([0,T];\mathbf{L}^2(\Om))$. The statement $ \Gradg q\in C_w([0,T];\mathbf{L}^2(\Ga))$ can be verified analogously. Thus, the regularity $(\Grad p,\Gradg q)\in C_w([0,T];\mathbfcal{L}^2)$ is established. 

As seen before, we know that $(\bv^\star,\bw^\star) = \mathbfcal{P}_\Div(\rho(\phi)\bv,\sigma(\psi)\bw)$. In particular, it holds
\begin{align}
    \big(\bv^\star(t),\bw^\star(t)\big) 
    = \mathbfcal{P}_\Div\big(\rho(\phi(t))\bv(t),\sigma(\psi(t))\bw(t)\big) \qquad\text{almost everywhere in~}(0,T).
\end{align}
By the characterization of the projection $\mathbfcal{P}_\Div = (\mathbf{P}^\Omega_\Div,\mathbf{P}^\Gamma_\Div): \mathbfcal{L}^2 \to \mathbfcal{L}^2_\Div$, the last identity entails that
\begin{align}\label{Char:vw:P_Div}
    \big(\bv^\star(t),\bw^\star(t)\big) 
    = \big(\rho(\phi(t))\bv(t) - \Grad\hat{p}(t),\sigma(\psi(t))\bw(t) - \Gradg\hat{q}(t)\big)
\end{align}
almost everywhere in $(0,T)$, where $\hat{p}(t)$ can be fixed as the weak solution to the Poisson--Neumann problem
\begin{alignat}{2}
    \Lap\hat{p}(t) &= \Div\big(\rho(\phi(t))\bv(t)\big) &&\qquad\text{in~}\Om, \\
    \deln\hat{p}(t) &= 0 &&\qquad\text{on~}\Ga,
\end{alignat}
whereas $\hat{q}(t)$ can be fixed as the weak solution of the inhomogeneous Laplace--Beltrami equation
\begin{align}
    \Lapg\hat{q}(t) = \Divg(\sigma(\psi(t))\bw(t)) \qquad\text{on~}\Ga.
\end{align}
Dividing the first component and second component in \eqref{Char:vw:P_Div} by $\rho(\phi(t))$ and $\sigma(\psi(t))$, respectively, we infer
\begin{align}
    \bv(t) &= \frac{1}{\rho(\phi(t))}\bv^\star(t) + \frac{1}{\rho(\phi(t))}\Grad\hat{p}(t), \label{Redef:v}\\
    \bw(t) &= \frac{1}{\sigma(\psi(t))}\bw^\star(t) + \frac{1}{\sigma(\psi(t))}\Gradg\hat{q}(t). \label{Redef:w}
\end{align}
Hence, for almost all $t\in(0,T)$, we deduce that
\begin{align*}
    &\intO \frac{1}{\rho(\phi(t))}\Grad\hat{p}(t)\cdot\Grad\zeta\dx = - \intO \frac{1}{\rho(\phi(t))}\bv^\star(t)\cdot\Grad\zeta\dx
\end{align*}
for all $\zeta\in H^1(\Om)$, and 
\begin{align*}
    \intG \frac{1}{\sigma(\psi(t))}\Gradg\hat{q}(t)\cdot\Gradg\xi\dG = - \intG \frac{1}{\sigma(\psi(t))}\bw^\star(t)\cdot\Gradg\xi\dG
\end{align*}
for all $\xi\in H^1(\Ga)$. This means that $\hat p$ and $\hat q$ are weak solutions of \eqref{AuxProblem:Bulk} and \eqref{AuxProblem:Surface}, respectively. Due to the uniqueness of those weak solutions, we infer $\Grad\hat{p}(t) = \Grad p(t)$ as well as $\Gradg\hat{q}(t) = \Gradg q(t)$ for almost every $t\in(0,T)$. Thus, by redefinition of $(\Grad\hat{p}, \Gradg\hat{q})$ on a set of measure zero, we get $(\Grad\hat{p},\Gradg\hat{q})\in C_w([0,T];\mathbfcal{L}^2)$. Again, with a redefinition of $(\bv,\bw)$ on a set of measure zero, we finally conclude from \eqref{Redef:v}-\eqref{Redef:w} that
\begin{align}
    \label{REG:VW:CW}
    (\bv,\bw)\in C_w([0,T];\mathbfcal{L}^2).
\end{align}
Here, we also used the fact that $(\rho(\phi), \sigma(\psi))\in C([0,T];\mathcal{L}^2)$ with $\rho \geq\rho_\ast > 0$ and $\sigma \geq\sigma_\ast > 0$.

It remains to show that the pair $(\bv,\bw)$ fulfills the initial condition 
$\eqref{WF:INI}_1$. To this end, as a direct consequence of Corollary~\ref{Cor:Proj:Stokes} and \eqref{GalerkinApprox:IC}, we first notice that
\begin{align*}
    (\bv_m(0),\bw_m(0)) = \mathbfcal{P}_{\mathbfcal{V}_m}(\bv_0,\bw_0) \rightarrow \mathbfcal{P}_{\Div}(\bv_0,\bw_0) = (\bv_0,\bw_0) \qquad\text{strongly in~}\mathbfcal{L}^2
\end{align*}
as $m\rightarrow\infty$. Then, as $(\phi_m(0),\psi_m(0)) = (\phi_0,\psi_0)$ a.e. in $\Om\times\Ga$, it readily follows that
\begin{align}
    \label{CONV:RVSW}
    (\rho(\phi_m(0))\bv_m(0),\sigma(\psi_m(0))\bw_m(0))\rightarrow (\rho(\phi_0)\bv_0,\sigma(\psi_0)\bw_0) \qquad\text{strongly in~}\mathbfcal{L}^2
\end{align}
as $m\rightarrow\infty$. 
We finally conclude that
\begin{align}
    \label{ID:V0W0}
        &\intO \rho(\phi_0)\bv(0)\cdot\wv\dx + \intG \sigma(\psi_0)\bw(0)\cdot\ww\dG 
        \nonumber\\
        &\quad = \intO \mathbf{P}_\Div^\Om(\rho(\phi_0)\bv(0))\cdot\wv\dx + \intG \mathbf{P}_\Div^\Ga(\sigma(\psi_0)\bw(0))\cdot\ww\dG 
        \nonumber\\
        \begin{split}
        &\quad = \lim_{m\rightarrow\infty} \Big(\intO \mathbf{P}_{\mathbfcal{V}_m}^\Om(\rho(\phi_m(0))\bv_m(0))\cdot\wv\dx + \intG \mathbf{P}_{\mathbfcal{V}_m}^\Ga(\sigma(\psi_m(0))\bw_m(0))\cdot\ww\dG\Big) 
        \\
        &\quad = \lim_{m\rightarrow\infty} \Big(\intO \rho(\phi_m(0))\bv_m(0)\cdot\mathbf{P}_{\mathbfcal{V}_m}^\Om(\wv)\dx 
        + \intG \sigma(\psi_m(0))\bw_m(0)\cdot \mathbf{P}_{\mathbfcal{V}_m}^\Ga(\ww)\dG\Big) 
        \end{split}
        \\
        &\quad = \intO \rho(\phi_0)\bv_0\cdot\mathbf{P}_\Div^\Om(\wv)\dx + \intG \sigma(\psi_0)\bw_0\cdot\mathbf{P}_\Div^\Ga(\ww)\dG 
        \nonumber\\
        &\quad = \intO \rho(\phi_0)\bv_0\cdot\wv\dx + \intG \sigma(\psi_0)\bw_0\cdot\ww\dG
        \nonumber
\end{align}
for all $(\wv,\ww)\in\mathbfcal{H}^2_{0,\Div}$.
Here, the second equality follows from the convergence \eqref{CONV:V*W*:2} and the fourth equality follows by means of \eqref{CONV:RVSW}.
Due to \eqref{ID:V0W0}, we are now in the setting of Lemma~\ref{Lemma:IC} and we finally conclude that $\eqref{WF:INI}_1$ is fulfilled.

\subsection{Completion of the proof}

In summary, we have shown above that the sextuplet $(\bv,\bw,\phi,\psi,\mu,\theta)$ fulfills the conditions \ref{DEF:WF:1}-\ref{DEF:WF:3} of Definition~\ref{Definition:WeakSolution}. 
Moreover, condition \ref{DEF:WF:4} now easily follows by testing the variational formulation \eqref{WF:PP} with $(\zeta,\xi) \equiv (\beta,1)$ if $L\in (0,\infty)$ or with $(\zeta,\xi) \equiv (1,0)$ and $(\zeta,\xi) \equiv (0,1)$ if $L=\infty$.

To verify the energy inequality \eqref{WF:DISS} stated in Definition~\ref{Definition:WeakSolution}\ref{DEF:WF:5}, we first notice that 
\begin{align}
    \label{Conv:E0}
    E_{\mathrm{tot}}(\mathbfcal{P}_{\mathbfcal{V}_m}(\bv_0,\bw_0),\phi_0,\psi_0) \to E_{\mathrm{tot}}(\bv_0,\bw_0,\phi_0,\psi_0),
\end{align}
as $m\to\infty$, which follows from the convergence of the projection $\mathbfcal{P}_{\mathbfcal{V}_m}$ (cf.~Corollary~\ref{Cor:Proj:Stokes}).
Then, applying the limes inferior on both sides of \eqref{EnergyDiss:m}, we use the convergences \eqref{Conv:vw:H^1}-\eqref{Conv:mt:L^2G} and \eqref{Conv:E0} along with standard convergence and weak lower semicontinuity arguments to conclude that $(\bv,\bw,\phi,\psi,\mu,\theta)$ satisfies the energy inequality \eqref{WF:DISS}. For a more detailed proof in a very similar situation, we refer to \cite[Proof of Theorem~3.3, Step~4]{Giorgini2023}.

This means that all conditions \ref{DEF:WF:1}-\ref{DEF:WF:5} from Definition~\ref{Definition:WeakSolution} are verified. Consequently, the sextuplet $(\bv,\bw,\phi,\psi,\mu,\theta)$ is a weak solution in the sense of Definition~\ref{Definition:WeakSolution} and therefore, the proof of Theorem~\ref{Theorem:L>0} is complete.
\end{proof}

\section{Asymptotic limit \texorpdfstring{$L\rightarrow 0$}{L->} and existence of a weak solution for \texorpdfstring{$L=0$}{L=0}}
\label{Section:Proof:L=0}
Lastly, in this section, we prove the existence of a weak solution to system \eqref{eqs:NSCH} in the remaining case $L = 0$.

\begin{theorem}\label{Theorem:L->0}
    Suppose that the assumptions \ref{Assumption:Domain}-\ref{Assumption:Potential} hold. Let $K\in[0,\infty]$,  
    let $(\bv_0,\bw_0)\in\mathbfcal{L}^2_\Div$, and let $(\phi_0,\psi_0)\in\mathcal{H}^1_{K,\alpha}$ satisfy \eqref{Assumption:InitalCondition:Int}-\eqref{Assumption:InitalCondition:L}.
	Moreover, assume that $\beta(\tilde{\sigma}_2 - \tilde{\sigma}_1) = -(\tilde{\rho}_2 - \tilde{\rho}_1)$. For any $L\in(0,\infty)$, let $(\bv_L,\bw_L,\phi_L,\psi_L,\mu_L,\theta_L)$ denote a weak solution to system \eqref{eqs:NSCH} in the sense of Definition~\ref{Definition:WeakSolution} with initial data $(\bv_0,\bw_0,\phi_0,\psi_0)$. Then, there exists a sextuplet $(\bv_\ast,\bw_\ast,\phi_\ast,\psi_\ast,\mu_\ast,\theta_\ast)$ with $\mu_\ast = \beta\theta_\ast$ a.e.~on $\Sigma$ such that
	\begin{alignat*}{2}
		(\bv_L,\bw_L) &\rightarrow (\bv_\ast,\bw_\ast) &&\qquad\text{weakly-star in~} L^\infty(0,T;\mathbfcal{L}^2), \\
		& &&\qquad\text{weakly in~} L^2(0,T;\mathbfcal{H}^1_{0,\Div}), \\
		& &&\qquad\text{strongly in~} L^2(0,T;\mathbfcal{L}^2), \\
		(\delt\phi_L,\delt\psi_L) &\rightarrow (\delt\phi_\ast,\delt\psi_\ast) &&\qquad\text{weakly in~} L^2(0,T;(\mathcal{H}^1_{0,\beta})^\prime), \\
		(\phi_L,\psi_L) &\rightarrow (\phi_\ast,\psi_\ast) &&\qquad\text{weakly-star in~} L^\infty(0,T;\mathcal{H}_{K,\alpha}^1), \\
		& &&\qquad\text{strongly in~} C([0,T];\mathcal{H}^s) \text{~for all~}s\in[0,1), \\
		(\mu_L,\theta_L) &\rightarrow (\mu_\ast,\theta_\ast) &&\qquad\text{weakly in~} L^2(0,T;\mathcal{H}^1)
	\end{alignat*}
	as $L\rightarrow 0$, along a non-relabeled subsequence, with
	\begin{align*}
		\norm{\beta\theta_L - \mu_L}_{L^2(\Sigma)} \leq C\sqrt{L},
	\end{align*}
	and the sextuplet $(\bv_\ast,\bw_\ast,\phi_\ast,\psi_\ast,\mu_\ast,\theta_\ast)$ is a weak solution to system \eqref{eqs:NSCH} in the sense of Definition~\ref{Definition:WeakSolution} for $L = 0$. 
\end{theorem}

\begin{proof}
    In this proof, the letter $C$ will denote generic positive constants, which may depend only on $\Omega$, $T$, the initial data, and the constants introduced in \ref{Assumption:Domain}-\ref{Assumption:Potential} except for $L$. The exact value of $C$ may vary throughout the proof.
    
	We consider an arbitrary sequence $\{L_m\}_{m\in\N}\subset(0,\infty)$ such that $L_m\rightarrow 0$ and a corresponding weak solution $(\bv_{L_m},				\bw_{L_m},\phi_{L_m},\psi_{L_m},\mu_{L_m},\theta_{L_m})$ to the initial data $(\bv_0,\bw_0,\phi_0,\psi_0)$. Let now $m\in\N$ be arbitrary. Using the energy inequality \eqref{WF:DISS}, we conclude that
    \begin{align}
		&\norm{(\bv_{L_m},\bw_{L_m})}_{L^\infty(0,T;\mathbfcal{L}^2)} + \norm{(\bv_{L_m},\bw_{L_m})}_{L^2(0,T;\mathbfcal{H}^1)} \leq C, \label{Est:Uniform:vw:L=0} \\
		&\norm{(\phi_{L_m},\psi_{L_m})}_{L^\infty(0,T;\mathcal{H}^1)} + \chi(K)^{1/2}\norm{\alpha\psi_{L_m} - \phi_{L_m}}_{L^\infty(0,T;L^2(\Ga))} \leq C, \label{Est:Uniform:pp:L=0} \\
		&\norm{(\Grad\mu_{L_m},\Gradg\theta_{L_m})}_{L^2(0,T;\mathbfcal{L}^2)} + L_m^{-1/2}\norm{\beta\theta_{L_m} - \mu_{L_m}}_{L^2(0,T;L^2(\Ga))} \leq C. \label{Est:Uniform:mt:L=0}
	\end{align}
	In particular, it immediately follows that
	\begin{align}
		\norm{\beta\theta_{L_m} - \mu_{L_m}}_{L^2(0,T;L^2(\Ga))} \leq C\sqrt{L_m}.
	\end{align}
	Next, take $(\zeta,\xi)\in\mathcal{H}^1_{0,\beta}$ as a test function in \eqref{WF:PP}. This yields
	\begin{align*}
		\abs{\bigang{(\delt\phi_{L_m},\delt\psi_{L_m})}{(\zeta,\xi)}_{\mathcal{H}^1_{0,\beta}}} &\leq \norm{\bv_{L_m}}_{\mathbf{L}^2(\Om)}\norm{\Grad\zeta}_{\mathbf{L}^2(\Om)} + \norm{\Grad\mu_{L_m}}_{\mathbf{L}^2(\Om)}\norm{\Grad\zeta}_{\mathbf{L}^2(\Om)} \\
		&\quad + \norm{\bw_{L_m}}_{\mathbf{L}^2(\Ga)}\norm{\Gradg\xi}_{\mathbf{L}^2(\Ga)} + \norm{\Gradg\theta_{L_m}}_{\mathbf{L}^2(\Ga)}\norm{\Gradg\xi}_{\mathbf{L}^2(\Ga)}.
	\end{align*}
	Consequently, taking the supremum over all $(\zeta,\xi)\in\mathcal{H}^1_{0,\beta}$ with $\norm{(\zeta,\xi)}_{\mathcal{H}^1}\leq 1$ and taking into account the estimates \eqref{Est:Uniform:vw:L=0}-\eqref{Est:Uniform:mt:L=0}, we readily deduce that
	\begin{align*}
		\norm{(\delt\phi_{L_m},\delt\psi_{L_m})}_{L^2(0,T;(\mathcal{H}^1_{0,\beta})^\prime)}\leq C.
	\end{align*}
	Then, the estimates
	\begin{align*}
		\norm{(F^\prime(\phi_{L_m}),G^\prime(\psi_{L_m}))}_{L^2(0,T;\mathcal{L}^2)} + \norm{(\mu_{L_m},\theta_{L_m})}_{L^2(0,T;\mathcal{L}^2)}\leq C
	\end{align*}
	can be derived as in the proof of Theorem~\ref{Theorem:L>0}. Consequently, by the Banach--Alaoglu theorem and Aubin--Lions--Simon lemma, there exists a sextuplet $(\bv_\ast,\bw_\ast,\phi_\ast,\psi_\ast,\mu_\ast,\theta_\ast)$ such that
		\begin{alignat*}{2}
		(\bv_{L_m},\bw_{L_m}) &\rightarrow (\bv_\ast,\bw_\ast) &&\qquad\text{weakly-star in~} L^\infty(0,T;\mathbfcal{L}^2), \\
		& &&\qquad\text{weakly in~} L^2(0,T;\mathbfcal{H}^1_{0,\Div}), \\
		(\delt\phi_{L_m},\delt\psi_{L_m}) &\rightarrow (\delt\phi_\ast,\delt\psi_\ast) &&\qquad\text{weakly in~} L^2(0,T;(\mathcal{H}^1_{0,\beta})^\prime), \\
		(\phi_{L_m},\psi_{L_m}) &\rightarrow (\phi_\ast,\psi_\ast) &&\qquad\text{weakly-star in~} L^\infty(0,T;\mathcal{H}^1_{K,\alpha}), \\
		& &&\qquad\text{strongly in~} C([0,T];\mathcal{H}^s) \text{~for all~}s\in[0,1), \\
		(\mu_{L_m},\theta_{L_m}) &\rightarrow (\mu_\ast,\theta_\ast) &&\qquad\text{weakly in~} L^2(0,T;\mathcal{H}^1), \\
		\beta\theta_{L_m} - \mu_{L_m} &\rightarrow \beta\theta_\ast - \mu_\ast &&\qquad\text{weakly in~} L^2(0,T;L^2(\Ga)),
	\end{alignat*}
	as $m\rightarrow\infty$, along a non-relabeled subsequence. In addition, we have
	\begin{align*}
		\norm{\beta\theta_{L_m} - \mu_{L_m}}_{L^2(0,T;L^2(\Ga))} \leq C\sqrt{L_m} \rightarrow 0
	\end{align*}
	as $m\rightarrow 0$, which entails $\beta\theta_\ast = \mu_\ast$ a.e.~on $\Sigma$.
    Moreover, using Minty's trick to identify the limit (cf.~\eqref{CONV:F*G*}-\eqref{MINTY:2}), we deduce that
    \begin{align*}
        (F^\prime(\phi_{L_m}), G^\prime(\psi_{L_m})) 
        \rightarrow (F^\prime(\phi_*), G^\prime(\psi_*)) \qquad\text{weakly in~} L^2(0,T;\mathcal{L}^2). 
    \end{align*}
    In particular, this entails that 
    \begin{align*}
        \abs{\phi_*} < 1 \quad\text{a.e.~in~}Q 
        \quad\text{and}\quad 
        \abs{\psi_*} < 1 \quad\text{a.e.~on~}\Sigma.
    \end{align*}
    Using the above convergences, it is straightforward to pass to the limit in the weak formulations \eqref{WF:PP} and \eqref{WF:MT} written for 
    $(\bv_{L_m}, \bw_{L_m},\phi_{L_m},\psi_{L_m},\mu_{L_m},\theta_{L_m})$
    to show that the
    sextuplet $(\bv_\ast,\bw_\ast,\phi_\ast,\psi_\ast,\mu_\ast,\theta_\ast)$
    satisfies \eqref{WF:PP} and \eqref{WF:MT}. 
    
	Next we prove a compactness result for $\{(\bv_{L_m},\bw_{L_m})\}_{m\in\N}$. As by assumption it holds that $\beta(\tilde{\sigma}_2 - \tilde{\sigma}_1) = -(\tilde{\rho}_2 - \tilde{\rho}_1)$, the weak formulation \eqref{WF:VW} written for $(\bv_{L_m}, \bw_{L_m},\phi_{L_m},\psi_{L_m},\mu_{L_m},\theta_{L_m})$ reads as
	\begin{align}
		\begin{split}
			&-\int_0^T\intO \rho(\phi_{L_m})\bv_{L_m}\cdot\delt\wv\dxt - \int_0^T\intG \sigma(\psi_{L_m})\bw_{L_m}\cdot\delt\ww\dGt \\
			&\qquad  - \int_0^T\intO \rho(\phi_{L_m})\bv_{L_m}\otimes\bv_{L_m}:\Grad\wv\dxt - \int_0^T\intG \sigma(\psi_{L_m})\bw_{L_m}\otimes\bw_{L_m}:\Gradg\ww\dGt \\
			&\qquad + \int_0^T\intO 2\nu_\Om(\phi_{L_m})\D\bv_{L_m}:\D\wv\dxt + \int_0^T \intG 2\nu_\Ga(\psi_{L_m})\Dg\bw_{L_m}:\Dg\ww\dGt \\
			&\qquad + \int_0^T\intG \gamma(\bv_{L_m},\bw_{L_m})\bw_{L_m}\cdot\ww\dGt - \int_0^T\intO \bv_{L_m}\otimes\J_{L_m}:\Grad\ww\dxt \\
			&\qquad - \int_0^T\intG \bw_{L_m}\otimes\K_{L_m}:\Gradg\bw_{L_m}\dGt \\
			&\quad = \int_0^T\intO \mu_{L_m}\Grad\phi_{L_m}\cdot\wv\dxt + \int_0^T\intG \theta_{L_m}\Gradg\psi_{L_m}\cdot\ww\dGt
		\end{split}
	\end{align}
	for all $(\wv,\ww)\in C_c^\infty(0,T;\mathbfcal{H}^2_{0,\Div})$, where
    $\J_{L_m} = -\rho^\prime\,\Grad\mu_{L_m}$ and $\K_{L_m} = -\sigma^\prime\,\Gradg\theta_{L_m}$.
	Now, arguing as in the proof of Theorem~\ref{Theorem:L>0}, we can show that	
	\begin{align*}
		\{\delt\mathbfcal{P}_\Div(\rho(\phi_{L_m})\bv_{L_m},\sigma(\psi_{L_m})\bw_{L_m})\}_{m\in\N} \quad \text{is bounded in~} L^{8/7}(0,T;(\mathbfcal{W}^{1,4}_{0,\Div})^\prime).
	\end{align*}
	Thus, as we additionally know that
    \begin{align*}
        \{\mathbfcal{P}_\Div(\rho(\phi_{L_m})\bv_{L_m},\sigma(\psi_{L_m})\bw_{L_m})\}_{m\in\N} \quad \text{is bounded in~}
        L^2(0,T;\mathbfcal{H}^1_{0,\Div}),
    \end{align*}
    it follows by the Aubin--Lions--Simon lemma that, up to subsequence extraction,
	\begin{align*}
		\mathbfcal{P}_\Div(\rho(\phi_{L_m})\bv_{L_m},\sigma(\psi_{L_m})\bw_{L_m}) \rightarrow (\widetilde{\bv},\widetilde{\bw}) \qquad\text{strongly in~} L^2(0,T;\mathbfcal{L}^2)
	\end{align*}
	as $m\rightarrow\infty$ for some $(\widetilde{\bv},\widetilde{\bw})\in L^\infty(0,T;\mathbfcal{H}^1_{0,\Div})$. Then, as in the proof of Theorem~\ref{Theorem:L>0}, the limit can be identified as
	\begin{align*}
		(\widetilde{\bv},\widetilde{\bw}) = \mathbfcal{P}_\Div(\rho(\phi_\ast)\bv_\ast,\sigma(\psi_\ast)\bw_\ast).
	\end{align*}
	Moreover, arguing as in the derivation of \eqref{Conv:v:L^2:strong} and \eqref{Conv:w:L^2:strong}, we establish the strong convergence
	\begin{align*}
		(\bv_{L_m},\bw_{L_m})\rightarrow (\bv_\ast,\bw_\ast) \qquad\text{strongly in~} L^2(0,T;\mathbfcal{L}^2)
	\end{align*}
	as $m\rightarrow\infty$. Finally, following the line of argument in the proof of Theorem~\ref{Theorem:L>0}, we verify that the limit sextuplet $(\bv_\ast,\bw_\ast,\phi_\ast,\psi_\ast,\mu_\ast,\theta_\ast)$ is indeed a weak solution of \eqref{eqs:NSCH} in the sense of Definition~\ref{Definition:WeakSolution} in the case $L = 0$.
\end{proof}


\section*{Acknowledgement}
This research was funded by the Deutsche Forschungsgemeinschaft (DFG, German Research Foundation) - Project 524694286. Moreover, the authors were partially supported by Deutsche Forschungsgemeinschaft (DFG, German Research Foundation) - RTG 2339. The support is gratefully acknowledged.

\section*{Conflict of Interests and Data Availability Statement}

There are no conflicts of interest. 

\noindent
There are no data associated with the manuscript.




\footnotesize

\bibliographystyle{abbrv}
\bibliography{KS_NSCH}

\end{document}